\documentclass[
    british,
    12pt,
    reqno
    ]{amsart}

\usepackage[T1]{fontenc}
\usepackage[utf8]{inputenc}
\usepackage{geometry}
\usepackage[dvipsnames]{xcolor}
\usepackage[colorlinks, pagebackref]{hyperref}
\usepackage{booktabs} 
\usepackage[shortlabels]{enumitem} 

\usepackage{amsmath}
\usepackage{amsthm}
\usepackage{amssymb}
\usepackage{mathtools} 

\theoremstyle{plain}
\newtheorem{thm}{Theorem}[section]

\newtheorem{cor}[thm]{Corollary}
\newtheorem{lem}[thm]{Lemma}
\newtheorem{prop}[thm]{Proposition}

\theoremstyle{definition}
\newtheorem{defn}[thm]{Definition}
\newtheorem{ex}[thm]{Example}

\theoremstyle{definition} 

\newtheorem{remark}[thm]{Remark}

\numberwithin{equation}{section}
\numberwithin{figure}{section}
\mathtoolsset{showonlyrefs=true} 

\newcommand{\bC}{\mathbb C}

\newcommand{\bE}{\mathbb E}

\newcommand{\bN}{\mathbb N}

\newcommand{\bQ}{\mathbb Q}
\newcommand{\bR}{\mathbb R}
\newcommand{\bZ}{\mathbb Z}
\newcommand{\bT}{\mathbb T}

\newcommand{\ba}{\mathbf a}
\newcommand{\bb}{\mathbf b}
\newcommand{\bc}{\mathbf c}

\newcommand{\bg}{\mathbf g}

\newcommand{\bm}{\mathbf m}
\newcommand{\bn}{\mathbf n}

\newcommand{\bq}{\mathbf q}

\newcommand{\bu}{\mathbf u}
\newcommand{\bv}{\mathbf v}
\newcommand{\bx}{\mathbf x}

\newcommand{\bz}{\mathbf z}

\newcommand{\bzero}{\mathbf 0}

\newcommand{\balp}{{\boldsymbol{\alp}}}
\newcommand{\bbet}{{\boldsymbol{\beta}}}

\newcommand{\bgam}{{\boldsymbol{\gam}}}

\newcommand{\beps}{{\boldsymbol \eps}}

\newcommand{\btet}{{\boldsymbol{\theta}}}

\newcommand{\fF}{\mathfrak F}

\newcommand{\fL}{\mathfrak L}

\newcommand{\cA}{\mathcal A}
\newcommand{\cB}{\mathcal B}
\newcommand{\cC}{\mathcal C}

\newcommand{\cF}{\mathcal F}

\newcommand{\cI}{\mathcal I}

\newcommand{\cK}{\mathcal K}

\newcommand{\cP}{\mathcal P}

\newcommand{\cV}{\mathcal V}
\newcommand{\cW}{\mathcal W}

\newcommand{\alp}{{\alpha}}
\newcommand{\bet}{{\beta}}
\newcommand{\gam}{{\gamma}}
\newcommand{\del}{{\delta}}
\newcommand{\eps}{{\varepsilon}}

\newcommand{\lam}{{\lambda}}

\newcommand{\Del}{{\Delta}}
\newcommand{\Ome}{{\Omega}}

\renewcommand{\epsilon}{\varepsilon}

\renewcommand{\le}{\leqslant}
\renewcommand{\leq}{\leqslant}
\renewcommand{\ge}{\geqslant}
\renewcommand{\geq}{\geqslant}

\DeclareMathOperator{\lcm}{lcm}
\DeclareMathOperator{\meas}{meas}
\DeclareMathOperator{\sinc}{sinc}

\DeclareMathOperator{\vol}{vol}
\renewcommand{\d}{\mathrm{d}} 

\newcommand{\N}{\mathbb N}
\newcommand{\Z}{\mathbb Z}

\newcommand{\R}{\mathbb R}
\newcommand{\T}{\mathbb T}

\newcommand{\defeq}{\coloneq} 
\newcommand{\divs}{\mid} 
\newcommand{\ndivs}{\nmid} 
\newcommand{\edivs}{\parallel} 
\DeclarePairedDelimiter{\norm}{\lVert}{\rVert} 
\DeclarePairedDelimiter{\abs}{\lvert}{\rvert} 
\newcommand{\indicator}{1_} 
\renewcommand{\vec}[1]{\boldsymbol{\mathbf{#1}}} 
\newcommand{\mmod}{\: \mathrm{mod} \:}
\renewcommand{\mod}{\mmod} 

\newcommand{\F}{\mathfrak F} 
\renewcommand{\L}{\mathfrak L} 
\newcommand{\Ma}{\mathfrak M} 
\newcommand{\ma}{\mathfrak m} 
\newcommand{\Na}{\mathfrak N} 
\newcommand{\na}{\mathfrak n} 
\newcommand{\unf}{{\mathrm{unf}}}
\newcommand{\str}{{\mathrm{str}}}
\newcommand{\sml}{{\mathrm{sml}}}
\newcommand{\FL}{{\mathrm{FL}}}

\begin{document}

\author{Sam Chow \and Sean Prendiville \and Santiago Vazquez}

\address{Mathematics Institute, Zeeman Building, University of Warwick, Coventry CV4 7AL, United Kingdom}
\email{sam.chow@warwick.ac.uk}

\address{School of Mathematical Sciences, Lancaster University, UK}
\email{s.prendiville@lancaster.ac.uk}

\address{Department of Mathematics, Strand Building, King's College London, London WC2R 2LS, United Kingdom}
\email{santiago.vazquez\_saez@kcl.ac.uk}

\title[Arithmetic regularity approach]{Arithmetic regularity as an alternative to transference}

\begin{abstract}
Since Green (2005), the Fourier-analytic transference principle has dominated the landscape of combinatorial theorems relative to sparse arithmetic sets. We demonstrate a different approach using arithmetic regularity. This is more versatile and has the potential to succeed when no obvious `dense model' is forthcoming. Moreover, we contend that, just as the traditional circle method disassembles an arithmetic problem into real and $p$-adic parts which can be solved individually, the arithmetic regularity method generalises this to yield an additional `combinatorial' factor. This framework leads directly to a correct lower bound on the number of configurations in a dense set. We illustrate this using a system comprising a linear equation together with a higher-degree equation.
\end{abstract}

\maketitle


\section{Introduction}

Roth \cite{Rot1953} famously showed that if $\cA \subseteq \bN$ has positive upper density 
\[
\limsup_{N \in \bN} 
\frac{|\cA \cap [N]|}{N} > 0,
\]
then $\cA$ contains non-trivial three-term arithmetic progressions, where 
\[
[N] = \{ 1, 2, \ldots, N \}.
\]
This has inspired many outstanding works 
considering sparse analogues, e.g. \cite{BMS2018, CFSZ2021, CG2016, GT2008, ST2015}.

A non-trivial three-term arithmetic progression $(x,y,z)$ is precisely a solution to $x + z = 2y$ with distinct variables. Roth's density-increment argument extends to homogeneous linear equations
\[
a_1 x_1 + \cdots + a_s x_s = 0,
\]
where $s \ge 3$ and $a_1, \ldots, a_s$ are non-zero integers summing to zero. The condition $a_1 + \cdots + a_s = 0$ is clearly necessary, since arithmetic progressions have positive upper density.

Higher-degree equations were initially approached by including additional equations to form a translation-invariant system, i.e. if $\bx$ is a solution then so is $a \bx + b \boldsymbol 1$ for $a, b \in \bZ$. Translation-invariant systems are amenable to Roth's approach or, more simply, an application of Szemer\'edi's theorem \cite{Sze1975}. These observations were exploited by Keil \cite{Kei2014} and Henriot \cite{Hen, Hen2015}.

In \cite{BP2017}, Browning and the second author brought the Fourier-analytic transference principle to bear on a problem not amenable to this approach. The Fourier-analytic transference principle was introduced by Green in \cite{Gre2005}, and used by Green--Tao in their renowned work on primes and arithmetic progressions \cite{GT2008}. Although \cite{BP2017} rested heavily on the foundations laid by Green, it led to an increase in similar applications, in part due to the following:
\begin{enumerate}[(i)]
\item It operated in what was then an unusually sparse setting, albeit preceded by work of Harper \cite{Har2016}.
\item The $W$-trick was necessarily more involved.
\item It packaged these ideas into a general framework which appeared to lead to other applications. We also refer the reader to the second author's expository piece~\cite{Pre2017}.
\end{enumerate}
The key proposition of \cite{BP2017} was refined in \cite[Theorem 4.3]{CCH} using recent advances in and around the quantitative theory of Roth's theorem \cite{FHHK2024, KM2023, Kos}.

In the aforementioned work, Browning and the second author showed that if $\cA \subseteq \bN$ has positive upper density then
\[
a_1 x_1^2 + \cdots + a_s x_s^2 = 0
\]
has a solution over $\cA$, whenever $s \ge 5$ and $a_1, \ldots, a_s$ are non-zero integers summing to zero. This was extended to higher degrees and to prime powers in \cite{Cho2018}, to systems of diagonal equations of the same degree in \cite{Cha2022}. It was then generalised to sums of values of 
intersective polynomials \cite{CC, CC2025}. Similar problems were tackled using the transference principle in \cite{CCH, Chi2019, RSZZ2025, Sal2020}. There are also Ramsey-theoretic analogues, for which the transference principle has played a pivotal role \cite{Cha2022, CLP2021, Pre2021, Sch2021}.

Despite these manifold applications, the transference approach has at least one clear defect, namely that an ad-hoc `dense model problem' needs to be concocted, a problem whose solutions can then be transferred to the sparse setting. Unfortunately, it is not always apparent whether an appropriate dense problem exists, as is the case for the main problem we discuss in the present work. We hope to convince the reader that applying arithmetic regularity naturally outputs the appropriate dense problem, without the need for any sophisticated thinking. We discuss our arithmetic regularity approach after presenting the application.

\subsection{A system of equations in dense variables}

Let $k \geq 2$ and $s \geq s_0$ be integers. Here, $s_0 =s_0(k)$ is defined by
\begin{center}
    \label{table:s_0-small-k}
    \begin{tabular}{cccccc}
        \toprule
        $k$ & 2 & 3 & 4 & 5 & 6 \\ \midrule
        $s_0(k)$ & 7 & 11 & 17 & 25 & 35 \\
        \bottomrule
    \end{tabular}
\end{center}
for $2\leq k \leq 6$,
and by
\begin{equation}
    \label{eq:s_0-large-k-def}
    s_0(k) = k(k-1) + 2\lfloor\sqrt{2k+2}\rfloor + \theta(k)
\end{equation}
for $k\geq 7$, where
\begin{equation}
    \label{eq:theta-k-def}
    \theta(k) =
    \begin{cases}
        -1, & \text{when } 2k+2 \leq \lfloor\sqrt{2k+2}\rfloor^{2} + \lfloor\sqrt{2k+2}\rfloor,\\
        1, & \text{when } 2k+2 > \lfloor\sqrt{2k+2}\rfloor^{2} + \lfloor\sqrt{2k+2}\rfloor.
    \end{cases}
\end{equation}
Let $\ba, \bb \in \bZ_{\ne 0}^s$ with
\begin{equation}\label{zerosum}
    a_1 + \cdots + a_s = b_1 + \cdots + b_s = 0.
\end{equation}
Let $\cV$ be the affine variety defined by
\[
a_1 x_1^k + \cdots + a_s x_s^k = b_1 x_1 + \cdots + b_s x_s =  0,
\]
and let $\cV'$ be the affine variety defined by
\begin{equation}
\label{variety}
a_1 x_1^j + \cdots + a_s x_s^j = 0 \qquad (1 \le j \le k).
\end{equation}

\begin{thm} 
\label{MainThm}
Assume that
\begin{enumerate}[label=(\Alph*), ref={\Alph*}]
    \item \label{case:neq} $a_i b_j \ne  a_j b_i$ whenever $i \ne j$, \textbf{or}
    \item \label{case:eq} $\ba = \bb$, and $\cV'$ has a non-singular point over $\bR$ and over $\bQ_p$ for each prime $p$.
\end{enumerate}
Let $\del > 0$, and let $\cA \subseteq [N]$ with $\# \cA \ge \del N$. Then
\[
\# \cA^s \cap \cV \gg_{\cV, \del} N^{s-k-1}.
\]
\end{thm}

\begin{remark}
In the case $k = 3$, Wooley \cite{Woo2015} gives a suitable mean value estimate below exponent 10, using work of Br\"udern and Robert \cite{BR2015}, so one might expect ten variables to suffice for us here. However, the corresponding restriction estimate is not known below this exponent \cite{HW2022}.
For this reason, we have not been able to handle the case $s = 10$ when $k = 3$.
\end{remark}

In Case \ref{case:neq}, we have the smooth rational solution $(1,\ldots,1)$, while in Case~\ref{case:eq} we can construct via the translation and dilation invariance of $\cV'$ a smooth adelic solution sufficiently close to the diagonal. One needs to be careful of obstructions similar to the following.

\begin{ex}
Suppose $(a_1,\ldots,a_s) = (b_1,\ldots,b_s)$ and $a_1, \ldots, a_{s-1} > 0$, so that 
\[
a_s = - (a_1 + \cdots + a_{s-1}) < 0.
\]
Writing $p_i  = -a_i/a_s$ $(1 \le i \le s-1)$, we obtain
\[
p_1 x_1^3 + \cdots + p_{s-1} x_{s-1}^3 = (p_1 x_1 + \cdots + p_{s-1} x_{s-1})^3.
\]
This is the equality case of Jensen's inequality, and can only occur if $\bx$ lies on the diagonal. Then $\bx$ is singular. We conclude that $\cV$ contains no non-singular real solutions. As $\cV(\bR)$ is the diagonal in this case, we cannot hope for a density result here anyway.
\end{ex}

A pleasant feature of Theorem \ref{MainThm} is so-called supersaturation; we establish a sharp lower bound for the number of dense solutions to the system. For $k \ge 3$, the result is new even qualitatively; the existence of a single non-trivial solution over $\cA$ was not hitherto known, except in certain cases where we are in Case~\ref{case:eq}  and there are sufficiently many variables to enable a solution similar to Keil and Henriot's \cite{Hen}. The situation is different when $k=2$ and we are in Case~\ref{case:eq}. In this case, the existence of non-trivial solutions was demonstrated by Keil \cite{Kei2014}, but the lower bound $N^{s-3}$ was not known. Our framework has the advantage of providing a direct path to supersaturation, without needing to inherit it from another system of equations.

\subsection{An arithmetic regularity lemma for polynomial phases}

Arithmetic regularity lemmas are analogous to Szemer\'edi's regularity lemma in graph theory. As well as the original article of Green \cite{Gre2005b}, good references include \cite{GT2010} and \cite{Tao2012}. Of course, this is now a completely natural approach to diophantine equations in dense variables. Our main objective in this paper is to show that it can serve as a general framework in sparser, higher-degree settings, much like how \cite{BP2017} brought the transference principle to this circle of problems. We do note, however, that it has been used in a few somewhat similar contexts \cite{GL2019, Lin2018}, and for one of the steps in some more similar contexts \cite{CC, CC2025}.

Our arithmetic regularity lemma is as follows. The $\Phi$-seminorm will be defined in \S \ref{WeakAndFull}, and trigonometric polynomials will be defined in \S \ref{struct}. For $M, N \ge 1$, the notion of $(M,N)$-irrationality will be defined in \S \ref{ARproof}.

\begin{lem}
\label{ARLgen}
Let $J \in \bN$, let $\eps, \del > 0$, and let $\cF: \N \to [1,\infty)$.
For $j \in [J]$, let $\Psi_j(x) \in \bZ[x]$ be a polynomial of degree $k_j$. Let $\cA \subseteq [N]$ with $\# \cA \ge \del N$. Then there exist
\begin{enumerate}[(a)]
\item positive integers $M \ll_{\eps, \del, \cF, \boldsymbol \Psi} 1$ and $d \le M/J$,
\item for $j \in [J]$, a $(\cF(M), N^{k_j})$-irrational frequency $\btet^{(j)} \in \bT^{c_j}$ with \mbox{$c_j \le d$,}
\item an arithmetic progression $\cP \subseteq [N]$ with $\# \cP \ge N/M$,
\item and decompositions
\[
1_{\cA} = f + f_\unf,
\qquad
f 1_{\cP} = f_\str + f_\sml,
\]
\end{enumerate}
such that
\begin{enumerate}[(i)]
\item $f: [N] \to [0,1]$,
\item $f_\unf: [N] \to [-1,1]$ with
$
\| f_\unf \|_{\Phi} \le N / \cF(M),
$
where $\Phi = \bT\textnormal{-span}(\boldsymbol \Psi)$,
\item $f_\sml: \cP \to [-2, 2]$ with $\| f_\sml \|_2 \le \eps \| 1_{\cP} \|_2$,
\item $f_{\str}: \cP \to [0,1]$ with $\| f_\str \|_1 \gg \del \| 1_{\cP} \|_1$,
\item as well as a trigonometric polynomial 
\[
F: \bT^{c_1} \times \cdots \times \bT^{c_J} \to [0,1]
\]
of degree at most $M$ such that
\[
f_\str(n) = F(\btet^{(1)} \Psi_1(n), \ldots, \btet^{(J)} \Psi_J(n))
\qquad (n \in \cP).
\]
\end{enumerate}
\end{lem}

We will only use the following special case.

\begin{lem}
\label{ARL}
Let $k \in \bN$, let $\eps, \del > 0$, and let $\cF: \N \to [1,\infty)$. Let $\cA \subseteq [N]$ with $\# \cA \ge \del N$. Then there exist
\begin{enumerate}[(a)]
\item positive integers $M \ll_{\eps, \del, \cF} 1$ and $d \le M/2$,
\item for $j \in 
\{ 1, k \}$, a $(\cF(M), N^j)$-irrational frequency $\btet^{(j)} \in \bT^{c_j}$ with $c_j \le d$,
\item an arithmetic progression $\cP \subseteq [N]$ with $\# \cP \ge N/M$,
\item and decompositions
\[
1_\cA = f + f_\unf,
\qquad
f1_\cP = f_\str + f_\sml,
\]
\end{enumerate}
such that
\begin{enumerate}[(i)]
\item $f: \bN \to [0,1]$,
\item $f_\unf: [N] \to [-1,1]$ with
\[
\sup_{\alp,\bet} \Big|
\sum_{n \le N} f_\unf(n) e(\alp n^k + \bet n)
\Big| \le \frac{N}{\cF(M)},
\]
\item $f_\sml: \cP \to [-2, 2]$ with $\| f_\sml \|_2 \le \eps \| 1_\cP \|_2$,
\item $f_{\str}: \cP \to [0,1]$ with $\| f_\str \|_1 \gg \del \| 1_\cP \|_1$,
\item as well as a trigonometric polynomial $F: \bT^{c_1 + c_k} \to [0,1]$ of degree at most $M$ such that
\[
f_\str(n) = F(\btet^{(1)} n, \btet^{(k)} n^k)
\qquad (n \in \cP).
\]
\end{enumerate}
\end{lem}

In the usual arithmetic regularity lemma \cite{GT2010}, the part $f_{\unf}$ has small Gowers uniformity norm $U^{t+1}[N]$, where $t \in \bN$ can be chosen. This is amply strong information if $t \ge k$, but comes at the expense of having less information about $f_{\str}$. The case $t=1$ is much simpler and provides strong information about $f_{\str}$, see \cite{Ebe, Tao2012}. However, it only prevents $f_\unf$ from correlating with linear phases, and so it is insufficient for our purposes. Much like the baby bear's porridge, our arithmetic regularity lemma is neither too hot nor too cold, but just right.

\subsection{Methods and organisation}

The proof of our arithmetic regularity lemma is modelled on 
\cite{Tao2012}. We incorporate an additional technical device, replacing regular Bohr sets with what we term `random translate Bohr neighbourhoods'. Although not essential to our application, we believe that this machinery simplifies a number of calculations, and we hope that it receives further use. The general idea is to use an energy increment argument, conditioning on the relevant atom of a Bohr partition. When working with a collection of polynomial phases, we form a closed set by taking their span, inducing a seminorm which furnishes a suitable notion of uniformity. This is carried out in \S \ref{S:ARL}.

To then solve our arithmetic combinatorics problem, we require two types of configuration control (generalised von Neumann lemmas). Standard arguments control the relevant counting operator using Fourier analysis, so that $f_\unf$ makes a negligible contribution. The number of variables required --- the value of $s_0$ --- comes from the resolution of the main conjecture of Vinogradov's mean value theorem, and subsequent results \cite{BDG2016, Pie2019, Woo2019}. This stage is completed in \S \ref{S:Config}.

To handle $f_\sml$, we require $\ell^2$ control, which is far more involved. The structured part $f_\str$ comes equipped with an arithmetic progression $\cP$. An application of Cauchy's inequality breaks the problem down into two parts, each of which needs to be solved with $\cP$-dependence tracked:
\begin{enumerate}[(i)]
\item A lower bound on the number of solutions over $\cP$,
\item and an upper bound on the number of solutions to a more complicated auxiliary system of equations in $2s - 1$ many variables.
\end{enumerate}
We complete each of these tasks using the circle method~\cite{Vau1997}. The analysis is quite involved, as the need for precise dependence on $\cP$ complicates the analysis of the singular series and singular integral for each task and in each case. This is sometimes referred to as \emph{quantitative weak approximation} \cite{BL2019, Lin}. These tasks are carried out in \S \ref{S:QWA} and \S \ref{S:Aux}, respectively.

To estimate the structured count, we need to exploit the quantitative irrationality of the frequencies $\btet^{(1)}$ and $\btet^{(k)}$. We do so in \S \ref{S:Equi} by using Fourier series to establish equidistribution statements, leading to an expression for our count in terms of the density of solutions on a certain torus.

Finally, in \S \ref{S:Proof}, we apply a solution to the compact problem to infer supersaturation for our original problem. The former is in general deep \cite{CSV2016}, but in this particular setting is not \cite{Tao2014}.

\subsection*{Notation}

We adopt the Vinogradov and Bachmann--Landau notations, as we now describe. If $f$ and $g$ are complex-valued functions, then we write $f \ll g$ or $f = O(g)$ if $|f| \le C|g|$ pointwise, for some constant $C$. We write $f \asymp g$ if $f \ll g$, and $f = o(g)$ if $f/g \to 0$ in some specified limit. We use a subscript to indicate that the implied constant $C$ is allowed to depend on certain variables.

For $x \in \bR$, we define $\| x \| = \min \{ |x - n|: n \in \bZ \}$. For $x \in \bT$, we write $\| x \|$ or $\| x \|_{\bT}$ for the distance from $x$ to $0$ in $\bT$.

The symbol $p$ is reserved for primes.

\subsection*{Acknowledgements}

Some ideas were developed many years ago in collaboration with Sofia Lindqvist. We also thank Kevin Hughes for a helpful conversation.

\subsection*{Funding}
SV was supported by EPSRC CDT Grant EP/S021590/1 for the London School of Geometry and Number Theory.

\subsection*{Rights}

For the purpose of open access, the authors have applied a Creative Commons Attribution (CC-BY) licence to any Author Accepted Manuscript version arising from this submission.

\section{Arithmetic regularity for 
polynomial phases}
\label{S:ARL}

In this section, we prove Lemma \ref{ARLgen}.

\subsection{Bohr random translates}

The following infrastructure extends the model case of linear Bohr sets discussed in \cite[Lecture 6]{Pre2020}. For $\phi: [N] \to \bT^d$ and $\balp \in \bT^d$, we define the inhomogeneous Bohr sets
\begin{equation}
\label{partition}
B_\phi^\balp(\beps) =  \{
n \in [N]: \phi(n) \in \balp + \beps/2 + [0,1/2)^d \}
\quad (\beps \in 
\{ 0,1 \}^d),
\end{equation}
partitioning $[N]$.
We now define a conditional expectation operator projecting $f: [N] \to \bC$ to a function $f_\phi^\balp$ that is constant on the atoms of the partition \eqref{partition}. For $n \in [N]$, we write $B_\phi^\balp(n)$ for the atom $B_\phi^\balp(\beps)$ containing $n$.
For $\phi: [N] \to \bT^d$, $\balp \in \bT^d$, and $f: [N] \to \bC$, we define $f_\phi^\balp(n)$ to be the average of $f$ over the atom of the partition \eqref{partition} containing $n$, i.e.
\begin{equation}
\label{ConditionalExpectation}
f_\phi^\balp(n) \coloneq \displaystyle \bE_{m \in B_\phi^\balp(n)} f(m).
\end{equation}

We define an inner product of $f,g: \bN \to \bC$ by
\[
\langle f, g \rangle = \sum_{n \in [N]} f(n) \overline{g(n)}.
\]
The functions $f$ and $g$ are \emph{orthogonal} if $\langle f, g \rangle = 0$. A series of calculations yield the following properties.

\begin{lem}
\label{properties}
Let $f,g: [N] \to \bC$, $\phi: [N] \to \bT^d$, and $\balp \in \bT^d$. Then:
\begin{enumerate}[(i)]
\item $(f+g)_\phi^\balp = f_\phi^\balp + g_\phi^\balp$.
\item $\langle f, g_\phi^\balp \rangle = \langle f_\phi^\balp, g \rangle$.
\item If, for each $\beps$, the restriction $f\vert_{B_\phi^\balp(\eps)}$ is constant, then $f_\phi^\balp = f$.
\item $f - f_\phi^\balp$ is orthogonal to any function that is constant on every $B_\phi^\balp(\beps)$.
\end{enumerate}
\end{lem}

For $\phi: [N] \to \bT^d$ and $f: [N] \to \bC$, we define the Bohr random translate projection $f_\phi: \bN \to \bC$ by
\begin{equation}
\label{BRTproj}
f_\phi(n) = \int_{\bT^d} f_\phi^\balp(n) \d \balp.
\end{equation}

\subsection{Weak regularity and full regularity}
\label{WeakAndFull}

Throughout this subsection and the next, we fix a non-empty but otherwise arbitrary set $\Phi$ of functions $[N] \to \bT$. For $f: [N] \to \bC$, we define a seminorm
\[
\| f \|_\Phi = \sup_{\phi \in \Phi} \Big|
\sum_{n \in [N]} f(n) e(\phi(n))
\Big|.
\]
We identify the set of functions $[N] \to \bT$ with $\bT^N$, thus imbuing it with a topology.

\begin{lem}
[Random translate inverse theorem]
\label{RTinv}
Let $\Phi$ be a closed subset of $\bT^N$, let $\cI \subseteq \bT$ be an interval of length $\lam(\cI) \le 1/2$, and let $f: \bN \to \bC$. Then there exists $\phi \in \Phi$ such that
\[
\int_\bT \Big|
\sum_{n \in [N]} f(n) 1_{\alp - \cI}(\phi(n)) \Big| \d \alp \ge \lam(\cI) \| f \|_\Phi / 2.
\]
\end{lem}

\begin{proof} 
The map
\[
\bT^N \to \bR, \qquad
\phi \mapsto \Big| \sum_{n \in [N]} f(n) e(\phi(n)) \Big|
\]
is continuous. Hence, by compactness of $\Phi$, there exists $\phi \in \Phi$ such that
\[
\| f \|_\Phi = \Big| \sum_{n \in [N]} f(n) e(\phi(n)) \Big|.
\]
Let us choose representatives $a, b \in \bR$ with $b - a = \lam(\cI)$ such that, viewed as subsets of $\bT$,
\[
(a,b) \subseteq \cI \subseteq [a,b].
\]
For $x \in \bT$,
\[
\int_\bT 1_{x - \cI} (\alp) e(\alp) \d \alp = \frac{e(x-a) - e(x-b)}{2 \pi i} = \frac{e(-a) - e(-b)}{2 \pi i} e(x) .
\]
Consequently,
\begin{align*}
\Big| \sum_{n \in [N]} f(n) e(\phi(n)) \Big| 
&= \frac{2 \pi}{|e(b-a) - 1|}  \Big| \sum_{n \in [N]} f(n) \int_\bT 1_{\phi(n)-\cI}(\alp) e(\alp) \d \alp
\Big|
\\ &\le \frac{2 \pi}{|e(b-a) - 1|} \int_\bT \Big| \sum_{n \in [N]} f(n) 1_{\alp - \cI}(\phi(n)) \Big| \d \alp.
\end{align*}
The conclusion now follows from the fact that
\[
|e(b-a)-1| = 2|\sin(\pi(b-a))| \ge 4 \| b - a \|.
\]
\end{proof}

The following uses an argument of Green and Tao \cite{GT2010} adapted to our context.

\begin{lem}
[Weak regularity]
\label{WR}
Let $\Phi$ be a closed subset of $\bT^N$ containing zero, and let $\eps > 0$. Let $f: [N] \to \bC$ be $1$-bounded, meaning that $|f| \le 1$ pointwise. Then there exists $d \le (4/\eps)^2$ and $\phi \in \Phi^d$ such that
\begin{equation}
\label{WRerror}
\| f - f_\phi \|_\Phi \le \eps N.
\end{equation}
\end{lem}

\begin{proof} We perform an iterative procedure. Stage $d \ge 0$ begins with $\phi \in \Phi^d$ such that
\begin{equation}
\label{EnergyIncrement}
\int_{\bT^d} \| f_\phi^\balp \|_2^2 \d \balp \ge d(\eps/4)^2 N.
\end{equation}
We initialise at stage $d = 0$ with $f_0 = \bE_{[N]} f$. The process terminates if $\| f - f_\phi \|_\Phi \le \eps N$.

Otherwise $\| f - f_\phi \|_\Phi > \eps N$ and, by Lemma \ref{RTinv}, there exists $\phi' \in \Phi$ such that
\[
\int_\bT |
\langle f - f_\phi, g \rangle
| \d \alp'
> \eps N/4,
\]
where $g = 1_{\alp' + [0,1/2)} \circ \phi'$.
Thus, by the definition \eqref{BRTproj} of $f_\phi$, linearity, and the triangle inequality,
\[
\int_\bT \int_{\bT^d} | \langle f-f_\phi^\balp , g \rangle | \d \balp \d \alp' > \eps N / 4.
\]

Since $g$ and $f_\phi^\balp$ are constant on sets of the form $B_{\phi, \phi'}^{\balp, \alp'} (\beps, \eps')$, Lemma \ref{properties} and Cauchy--Schwarz give
\[
g_{\phi,\phi'}^{\balp, \alp'} = g, \qquad
(f_\phi^\balp)_{\phi, \phi'}^{\balp, \alp'} = f_\phi^\balp,
\]
and
\begin{align*}
|\langle f - f_\phi^\balp, g \rangle| &=
|\langle f - f_\phi^\balp, g_{\phi,\phi'}^{\balp, \alp'} \rangle|
= |\langle f_{\phi,\phi'}^{\balp,\alp'} - f_\phi^\balp, g \rangle|
\le N^{1/2} \| f_{\phi,\phi'}^{\balp,\alp'} -f_\phi^\balp \|_2.
\end{align*}
Consequently,
\[
\int_\bT \int_{\bT^d} \| f_{\phi,\phi'}^{\balp,\alp'} -f_\phi^\balp \|_2 \d \balp \d \alp' > \eps N^{1/2}/4.
\]
A further application of the Cauchy--Schwarz inequality yields
\[
\int_\bT \int_{\bT^d} \| f_{\phi,\phi'}^{\balp,\alp'} -f_\phi^\balp \|_2^2 \d \balp \d \alp' > (\eps/4)^2 N.
\]

By Lemma \ref{properties},
\[
\langle f_{\phi, \phi'}^{\balp, \alp'} - f_\phi^\balp, f_\phi^\balp \rangle =
\langle f_{\phi, \phi'}^{\balp, \alp'} - (f_\phi^\balp)_{\phi,\phi'}^{\balp,\alp'}, f_\phi^\balp \rangle = \langle f - f_\phi^\balp, f_\phi^\balp \rangle = 0.
\]
Hence, by Pythagoras's theorem,
\[
\| f_{\phi, \phi'}^{\balp, \alp'} \|_2^2 = \| f_{\phi,\phi'}^{\balp,\alp'} -f_\phi^\balp \|_2^2 + \| f_\phi^\balp \|_2^2.
\]
Now
\[
\int_\bT \int_{\bT^d} \| f_{\phi,\phi'}^{\balp,\alp'} \|_2^2 \d \balp \d \alp' > (\eps/4)^2 N + \int_{\bT^d} \| f_\phi^\balp \|_2^2 \d \balp,
\]
so we have \eqref{EnergyIncrement} with $f_{\phi, \phi'}^{\balp, \alp'}$ in place of $f_\phi^\balp$ and $d + 1$ in place of $d$, and we proceed to stage $d+1$.

The left hand side of \eqref{EnergyIncrement} is bounded above by $N$, so our procedure must terminate at some stage $d \le (4/\eps)^2$ with a projection satisfying \eqref{WRerror}.
\end{proof}

The argument underpinning Lemma \ref{WR} tolerates initialising with a projection more refined than the function $f_0 = \bE_{[N]}f$. This is useful in proving stronger regularity lemmas, as we will soon see.

\begin{lem}
[Refined weak regularity]
\label{RWR}
Let $\Phi$ be a closed subset of $\bT^N$, and let $\phi \in \Phi^d$. Let $\eps > 0$, and let $f: \bN \to \bC$ be a $1$-bounded function. Then there exist $d' \le (4/\eps)^2$ and $\phi' \in \Phi^{d'}$ such that
\begin{equation}
\label{RWRbound}
\| f- f_{\phi, \phi'} \|_\Phi \le \eps N.
\end{equation}
\end{lem}

\begin{proof} We carry out the energy increment process used to prove Lemma \ref{WR}, initialising with $f_\phi$ instead of $f_0$. At each stage, we query whether or not \eqref{RWRbound} holds. If not, then we increase the quantity
\begin{equation}
\label{incrementor}
\int_{\bT^{d+d'}} \| f_{\phi, \phi'}^{\balp, \balp'} \|_2^2 \d \balp \d \balp'
\end{equation}
by $(\eps/4)^2 N$. Since \eqref{incrementor} is bounded above by $N$, this procedure must terminate within $d' \le (4/\eps)^2$ many steps.
\end{proof}

This refinement of the weak regularity lemma is amenable to iteration, leading to a `full' regularity lemma.

\begin{lem}
[Full regularity]
\label{SFR}
Let $\cF: [1,\infty) \to [1,\infty)$ and $\eps > 0$. Let $\Phi$ be a closed subset of $\bT^{N}$ containing zero, and let $f: [N] \to \bC$ be $1$-bounded. Then there exist
\[
1 \le d, d' \ll_{\cF,\eps} 1, \qquad \phi \in \Phi^d, \qquad
\phi' \in \Phi^{d'}
\]
such that
\[
\| f - f_{\phi, \phi'} \|_{\Phi} \le N/\cF(d), \qquad
\| f_{\phi, \phi'} - f_{\phi} \|_2 \le \eps \sqrt{N}.
\]
\end{lem}

\begin{proof}
We apply Lemma \ref{WR}, furnishing $d_1 \le 16 \cF(0)^2$ and $\phi \in \Phi^{d_1}$ such that
\[
\| f - f_{\phi} \|_{\Phi} \le N/\cF(0).
\]
With $d_2 = 16 \cF(d_1)^2$, Lemma \ref{RWR} gives $\phi' \in \Phi^{d_2}$ such that
\[
\| f - f_{\phi, \phi'} \|_{\Phi} \le N / \cF(d_1).
\]
Iterating this procedure, at stage $j$ we obtain 
\[
d_j = 16 \cF(d_1 + \cdots + d_{j-1})^2, \qquad
\phi \in \Phi^{d_1 + \cdots + d_{j-1}}, \qquad \phi' \in \Phi^{d_j}
\]
such that
\[
\| f - f_{\phi, \phi'} \|_{\Phi} \le N/\cF(d_1 + \cdots + d_{j-1}).
\]

By Lemma \ref{properties} and Pythagoras's theorem, the quantity
\[
\int_{\bT^{d_1 + \cdots + d_j}} \| f_{\phi, \phi'}^{\balp, \balp'} \|_2^2 \d \balp \d \balp' 
= \int_{\bT^{d_1 + \cdots + d_j}} 
(\| f_{\phi}^{\balp} \|_2^2 +  \| f_{\phi, \phi'}^{\balp, \balp'} - f_{\phi}^{\balp} \|_2^2)
\d \balp \d \balp'
\]
is non-decreasing at each stage of the iteration. As $N$ is an upper bound for this quantity, it follows that at some stage $j \ll_{\eps} 1$,
\[
\int_{\bT^{d_1 + \cdots + d_j}} \| f_{\phi, \phi'}^{\balp, \balp'} - f_{\phi}^{\balp} \|_2^2
\d \balp \d \balp' \le \eps^2 N.
\]
The result now follows from the Cauchy--Schwarz inequality.
\end{proof}

\subsection{Structure of the projection}
\label{struct}

We write $\Del$ for the symmetric difference operator on pairs of sets.

\begin{lem} 
\label{SymDif}
Let $0 < \eps \le 1/2$ and $\bgam \in \bT^d$. Then
\[
\meas_{\bT^d}
([0,\eps)^d \Del ([0,\eps)^d + \bgam)) \le 2 \eps^{d-1} \sum_{j \in [d]} \| \gam_j \|_\bT.
\]
\end{lem}

\begin{proof} Let us represent each $\gam_j \in \bT$ by a real number $\gam_j$ with $|\gam_j| \le 1/2$. As $0 < \eps \le 1/2$,
\[
\meas_{\bT^d}
([0,\eps)^d \Del ([0,\eps)^d + \bgam))
=
\meas_{\bR^d}
([0,\eps)^d \Del ([0,\eps)^d + \bgam)).
\]

Next, we bound $\meas_{\bR^d}
([0,\eps)^d \cap ([0,\eps)^d + \bgam))$ from below. This intersection is a box whose $j^{\mathrm{th}}$ side has length $(\eps - |\gam_j|)_+$ for each $j$, with the notation $x_+ = \max \{ 0, x \}$. Consequently,
\[
\meas_{\bR^d}
([0,\eps)^d \cap ([0,\eps)^d + \bgam)) = (\eps - |\gam_1|)_+ \cdots (\eps - |\gam_d|)_+.
\]
We now induct on $d$ to show that this is bounded below by $\eps^d - \eps^{d-1} \displaystyle \sum_{j \in [d]} |\gam_j|$. This is clear for $d=1$. For $d \ge 2$, we compute using the inductive hypothesis that
\begin{align*}
(\eps - |\gam_1|)_+ \cdots (\eps - |\gam_d|)_+
&\ge (\eps - |\gam_1|)_+ \cdots (\eps - |\gam_{d-1}|)_+ \eps 
- \eps^{d-1} |\gam_d| \\
&\ge \eps^d - \eps^{d-1} \sum_{j \in [d-1]} |\gam_j| - \eps^{d-1} |\gam_d| \\
&= \eps^d - \eps^{d-1} \sum_{j \in [d]} |\gam_j|.
\end{align*}

The upshot is that
\[
\meas_{\bR^d}
([0,\eps)^d \cap ([0,\eps)^d + \bgam)) \ge \eps^d - \eps^{d-1} \sum_{j \in [d]} |\gam_j|.
\]
The conclusion now follows from the inclusion-exclusion formula for the symmetric difference.
\end{proof}

Given $M > 0$, a function $F: \bT^d \to \bC$ is \emph{$M$-Lipschitz} if
\[
|F(\balp) - F(\bbet)| \le M \sum_{j \in [d]} \| \alp_j - \bet_j \|_\bT
\qquad (\balp, \bbet \in \bT^d).
\]

\begin{lem}
[Lipschitz structure]
\label{Lip}
Let $\phi: [N] \to \bT^d$, and let $f: [N] \to \bC$ be $1$-bounded. Then there exists a $1$-bounded and $4$-Lipschitz function $F: \bT^d \to \bC$ such that $f_\phi = F \circ \phi$, and if $f \ge 0$ then $F \ge 0$.
\end{lem}

In the proof, we use the convolution operator given by
\[
G_1 * G_2(\bbet) = \int_{\bT^d} G_1(\bbet - \balp) G_2(\balp) \d \balp.
\]

\begin{proof}
For $\beps \in \{ 0, 1 \}^d$ and $\balp \in \bT^d$, define
\[
F_\beps(\balp) =
\begin{cases}
\mathbb E_{m\in B_\phi^{\balp}(\beps)}f(m),  &\text{if } B_\phi^{\balp}(\beps) \neq \emptyset, \\
0, &\text{if } B_\phi^{\balp}(\beps) = \emptyset.
\end{cases}
\]
Then
\begin{align*}
&\sum_{\beps \in \{0,1\}^d} \left( 1_{\beps/2 + [0,1/2)^d} * F_\beps \right) (\phi(n)) \\
&= \int_{\bT^d} \sum_{\beps \in \{0,1\}^d} 1_{\beps/2 + [0,1/2)^d}(\phi(n)-\balp) F_\beps(\balp) \d \balp \\
&= \int_{\bT^d} \sum_{\beps \in \{0,1\}^d} 1_{n \in B_\phi^\balp(\beps)} F_\beps(\balp) \d \balp
= \int_{\bT^d} \bE_{m \in B_\phi^\balp(n)} f(m) \d \balp = f_\phi(n).
\end{align*}
We define the $1$-bounded function
\[
F(\bbet) = \sum_{\beps \in \{0,1\}^d} \left( 1_{\beps/2 + [0,1/2)^d} * F_\beps \right) (\bbet)  = \sum_{\beps \in \{0,1\}^d} \left( 1_{[0,1/2)^d} * F_\beps \right) (\bbet - \beps/2),
\]
noting that $f_\phi = F \circ \phi$.

To show that $F$ is $4$-Lipschitz, it suffices to prove that if $G: \bT^d \to \bC$ is $1$-bounded and measurable then $1_{[0,1/2)^d} * G$ is $2^{2-d}$-Lipschitz. This follows from the triangle inequality, a change of variables, and Lemma \ref{SymDif}. Indeed, if $\bbet, \bgam \in \bT^d$ then
\begin{align*}
&|1_{[0,1/2)^d} * G (\bbet + \bgam) - 1_{[0,1/2)^d} * G(\bbet)|
\\
&\le \int_{\bT^d}
|1_{[0,1/2)^d}(\balp + \bgam) - 1_{[0,1/2)^d}(\balp)| \d \balp \\
&= \meas_{\bT^d}(
[0,1/2)^d \Delta
([0,1/2)^d - \bgam)) \le 2^{2-d} \sum_{j \in [d]} \| \gam_j \|_\bT.
\end{align*} 
\end{proof}

Next, we go from a bound on the Lipschitz constant to a good trigonometric approximation. A \emph{trigonometric polynomial} is a function $\bT^d \to \bC$ given by a finite linear combination of exponential functions $\balp \mapsto e(\bm \cdot \balp)$, where $\bm \in \bZ^d$. By orthogonality, this has the form
\[
F(\balp) = \sum_{\bm \in \bZ^d} \hat F(\bm) e(\bm \cdot \balp),
\]
where
\[
\hat F(\bm) \coloneq \int_{\bT^d} F(\balp) e(-\bm \cdot \balp) \d \balp
\]
is only non-zero for finitely many $\bm$. The \emph{degree} of $F$ is then
\[
\max \{ |m_1| + \cdots + |m_d|: \hat F(\bm) \ne 0 \}.
\]

\begin{lem}
[Trigonometric approximation]
\label{trig}
Let $K \ge 1$ and $0 < \eps < 1/2$, and let $F: \bT^d \to \bC$ be a $K$-Lipschitz function. Then there exists a $1$-bounded trigonometric polynomial $F_\eps: \bT^d \to \bC$ of degree $O(K^2 d^2 \eps^{-3})$ such that $\| F - F_\eps \|_\infty \le \eps$, and if $0 \le F \le 1$ then $0 \le F_\eps \le 1$.
\end{lem}

\begin{proof}
For $M \in \bN$, we introduce the Fej\'er kernel
\begin{align*}
\mu_M(\balp) &= \mu_M(\alp_1) \cdots \mu_M(\alp_d) \\
&= \sum_{\bm \in \bZ^d} (1-|m_1|/M)_+ \cdots (1-|m_d|/M)_+ e(\bm \cdot \balp),
\end{align*}
where
\[
\mu_M(\alp) =
M^{-1}  \left( \frac{\sin M \pi \alp}{\sin \pi \alp} \right)^2.
\]
This has the following properties:
\[
\mu_M \ge 0, \qquad
\int_{\bT^d} \mu_M = 1,
\]
and
\begin{equation}
\label{decay}
\mu_M(\balp) \le 4M^{-1} \| \alp_j \|^{-2} \prod_{i \ne j} \mu_M(\alp_i) \qquad (1 \le j \le d).
\end{equation}

For $M = \lceil 1000 (Kd)^2 \eps^{-3} \rceil$, we choose
\[
F_\eps(\balp) = F * \mu_M(\balp) = \int_{\bT^d}  F(\balp - \bbet) \mu_M(\bbet) \d \bbet.
\]
By orthogonality,
\[
F_\eps(\balp) = \sum_{\bm \in \bZ^d} (1 - |m_1|/M)_+ \cdots (1-|m_d|/M)_+ \hat F(\bm) e(\bm \cdot \balp)
\]
is a trigonometric polynomial of degree at most $M$. Moreover, the definition of $F_\eps$ as a convolution ensures that if $0 \le F \le 1$ then $0 \le F_\eps \le 1$.

Finally, we bound
\[
F(\balp) - F_\eps(\balp) = \int_{\bT^d} (F(\balp) - F(\balp - \bbet)) \mu_M(\bbet) \d \bbet.
\]
Let $\eta > 0$ be a threshold given by $\eta^{-3} = KdM$.
As $F$ is $K$-Lipschitz,
\[
\Big|
\int_{\| \bbet \|_\infty \le \eta} (F(\balp) - F(\balp - \bbet)) \mu_M(\bbet) \d \bbet
\Big| \le K d\eta = (Kd)^{2/3} M^{-1/3}.
\]
By \eqref{decay},
\[
\Big|
\int_{\| \bbet \|_\infty > \eta} (F(\balp) - F(\balp - \bbet)) \mu_M(\bbet) \d \bbet
\Big| \le 4/(M\eta^2) = 4 (Kd)^{2/3} M^{-1/3}.
\]
Therefore
\[
\| F - F_\eps \|_\infty
\le 5 (Kd)^{2/3} M^{-1/3} \le \eps.
\]
\end{proof}

\subsection{Proof of arithmetic regularity for polynomial phases}
\label{ARproof}

In this subsection, we finally prove Lemma \ref{ARLgen}. Given $M,N \ge 1$, a vector $\btet \in \bT^d$ is \emph{$(M,N)$-rational} if there exists $\bzero \ne \bq \in \bZ^d$ such that
\[
\| \bq \|_1 \le M,
\qquad
\| \bq \cdot \btet \| \le M/N,
\]
and otherwise \emph{$(M,N)$-irrational}.
The following decomposition is similar to \cite[Proposition 1.1.17]{Tao2012}.

\begin{lem}
\label{Ratner}
Let $\balp^{(j)} \in \bT^{d_j}$ for $j \in [J]$, where $d_j\ge 0$. Put $D = d_1 + \cdots + d_J$, and let $N_1, \ldots, N_J, M_0, \ldots, M_{D} \ge 1$. Then there are decompositions
\[
\balp^{(j)} = \ba^{(j)}/q + \bbet^{(j)}/N_j + L_j \btet^{(j)} \qquad (1 \le j \le J),
\]
where, for some $0 \le n \le D$ and all $j \in [J]$, we have:
\begin{enumerate}[(i)]
\item $q \in \bN$ with $q \le M_0\cdots M_{n-1}$
\item $\ba^{(j)} \in \bZ^{d_j}$
\item $\bbet^{(j)} \in \bR^{d_j}$ with $\| \bbet^{(j)} \|_1 \le \sum_{i=0}^{n-1} M_0\cdots M_{i}$
\item $\btet^{(j)}$ is $(M_n, N_j)$-irrational with $\btet^{(j)} \in \bT^{c_j}$ for some $0 \le c_j \le d_j$
\item $L_j$ is a $d_j \times c_j$ integer matrix whose rows and columns are bounded in $L^1$ by $M_0\cdots M_{n-1}$.
\end{enumerate}
\end{lem}

\begin{proof}
The inductive base case $D = 0$ is vacuously true, upon interpreting 
$R^0 = \{ 0 \}$ for any ring $R$. Next, let $D \ge 1$ and assume the conclusion with $D - 1$ in place of $D$.

If $\balp^{(j)}$ is $(M_0, N_j)$-irrational for every $j$, then we can take
\[
n = 0, \quad
q = 1, \quad
\ba^{(j)} = \bbet^{(j)} = \bzero, \quad
L_j = I_{d_j},
\quad
\btet^{(j)} = \balp^{(j)}.
\]
Thus, we may assume instead that there exist $i \in [J]$, as well as $\bzero \ne \bq \in \bZ^{d_i}$ such that
\[
\| \bq \|_1 \le M_0,
\qquad
\| \bq \cdot \balp^{(i)} \| \le M_0 / N_i.
\]
It follows on viewing $\balp^{(i)}$ as a real vector that
\[
\bq \cdot \balp^{(i)} = a + \bet / N_i,
\]
for some $a \in \bZ$ and some $\bet \in \bR$ with $|\bet| \le M_0$. Relabelling indices if necessary, we may assume that the first co-ordinate of $\bq$, denoted $q_1$, is non-zero. Now
\[
\balp^{(i)} - (a/q_1 + \bet/(q_1 N_i), 0, \ldots, 0)
\]
lies in the kernel of the real linear transformation $\bv \mapsto \bq \cdot \bv$. This kernel is given by $L \cdot \bR^{d_i - 1}$, where
\[
L = \begin{pmatrix}
q_2 & \ldots & q_{d_i} \\
-q_1 & & \\
& \ddots & \\
&& - q_1
\end{pmatrix}
\]
is a $d_i \times (d_i  - 1)$ integer matrix. Thus, we may write
\[
\balp^{(i)} - (a/q_1 + \bet/(q_1 N_i), 0, \ldots, 0)
= L \tilde \balp^{(i)},
\]
for some $\tilde \balp^{(i)} \in \bT^{d_i - 1}$.

Put $\tilde d_i = d_i - 1$.
Setting $\tilde \balp^{(j)} = \balp^{(j)}$ and $\tilde d_j = d_j$ for $j \ne i$, the induction hypothesis --- applied to $M_1, \ldots, M_D$ --- furnishes decompositions
\[
\tilde \balp^{(j)} = \tilde \ba^{(j)} / \tilde q + \tilde \bbet^{(j)} / N_j + \tilde L_j \btet^{(j)}
\]
where, for some $1 \le n \le D$ and all $j \in [J]$, we have:
\begin{enumerate}[(i)]
\item $\tilde q \in \bN$ with $\tilde q \le M_1\cdots M_{n-1}$
\item $\tilde \ba^{(j)} \in \bZ^{\tilde d_j}$
\item $\tilde \bbet^{(j)} \in \bR^{\tilde d_j}$, where 
$
\| \tilde \bbet^{(j)} \|_1 \le \sum_{i=1}^{n-1}M_1\cdots M_{i}
$
\item $\btet^{(j)}$ is $(M_n, N_j)$-irrational with $\btet^{(j)} \in \bT^{\tilde c_j}$ for some $0 \le \tilde c_j \le \tilde d_j$ 
\item $\tilde L_j$ a $\tilde d_j \times \tilde c_j$ integer matrix whose rows and columns are bounded in $L^1$ by $M_1\cdots M_{n-1}$.
\end{enumerate}
Now
\[
\balp^{(i)} = (a/q_1 + \bet/(q_1 N_i), 0, \ldots, 0) + L(\tilde \ba^{(i)}/ \tilde q + \tilde \bbet^{(i)}/ N_i + \tilde L_i \btet^{(i)}),
\]
and the claims follow from the inductive data upon choosing
\begin{align*}
q &= |q_1| \tilde q,
\\
\ba^{(i)} &= (\tilde q a |q_1|/q_1, 0, \ldots, 0) + L |q_1| \tilde \ba^{(i)},
\\
\bbet^{(i)} &= (\bet/q_1,0,\ldots,0) + L \tilde \bbet^{(i)}, \\ 
L_i &= L \tilde L_i,
\end{align*}
as well as
\[
\ba^{(j)} = |q_1| \tilde \ba^{(j)}, \quad
\bbet^{(j)} = \tilde \bbet^{(j)}, \quad
L_j = \tilde L_j
\qquad (j \ne i).
\]
\end{proof}

The following observation is our final preparatory step for the proof of Lemma \ref{ARLgen}.

\begin{lem}
\label{pigeonhole}
Let $\eps, \del > 0$. Let $f: [N] \to [-1,1]$ and $g: [N] \to [0,\infty)$ with
\[
\sum_{n \in [N]} f(n) \ge \del N, \qquad
\sum_{n \in [N]} g(n) \le \eps N.
\]
Then, for any partition $[N] = \cC_1 \cup \cdots \cup \cC_r$, there exists $i \in [r]$ such that
\[
|\cC_i| \ge \del N / (3r),
\qquad
\sum_{n \in \cC_i} f(n) \ge \del |\cC_i|/3, \qquad
\sum_{n \in \cC_i} g(n) \le 3 \del^{-1} \eps |\cC_i|.
\]
\end{lem}

\begin{proof} Observe that
\[
\sum_{i \in [r]} \Bigl(
\Bigl(
\sum_{n \in \cC_i} f(n)
\Bigr)
- \frac{\del |\cC_i|}{3} - \frac{\del}{3 \eps} \sum_{n \in \cC_i} g(n) \Bigr) \ge \frac{\del}{3} N.
\]
Thus, by the greedy algorithm, there exists $i \in [r]$ such that
\[
\Bigl(
\sum_{n \in \cC_i} f(n)
\Bigr)
- \frac{\del |\cC_i|}{3} - \frac{\del}{3 \eps} \sum_{n \in \cC_i} g(n) \ge \frac{\del}{3r} N.
\]
Now
\[
|\cC_i| \ge \sum_{n \in \cC_i} f(n) \ge \frac{\del |\cC_i|}{3} + \frac{\del}{3 \eps} \sum_{n \in \cC_i} g(n) + \frac{\del}{3r} N.
\]
\end{proof}

\begin{proof}
[Proof of Lemma \ref{ARLgen}]
Let $c = c(\boldsymbol \Psi)$ be a sufficiently small, positive constant, and put $\eta = c \eps \del$.
By replacing $\cF(x)$ by $\max \{ \cF(y): 1 \le y \le x \}$, we may assume that $\mathcal F$ is non-decreasing and defined on $[1, \infty)$.
Define, for $x \in \N$,
\begin{align*}
\cF_0(x) &= 
\cF(\eta ^{-4}x^{2}), \\
\cF_j(x) &= \cF \Bigl(
\eta ^{-4}x^{2} \Bigl(\sum_{i=0}^{j-1}\cF_{0}(x)\cdots \mathcal F_i(x)\Bigr)^{2}\Bigr) \qquad (1 \le j \le Jx),
\\
\cF_*(x) &= \cF_{Jx}(x).
\end{align*}

We apply Lemma \ref{SFR} to the function $g \coloneq 1_{\cA}$, with $\Phi = \bT\textnormal{-span}(\boldsymbol \Psi)$, growth function $\cF_*$ and approximation parameter $\eta > 0$. This yields positive integers $d, \tilde d \ll_{\cF,\eps,\del,\boldsymbol \Psi} 1$, as well as  
$
\balp^{(j)} \in \bT^d
$
and
$
\tilde \balp^{(j)} \in \bT^{\tilde d}
$
such that
\[
\phi \coloneq \balp^{(1)} \Psi_1 + \cdots + \balp^{(J)} \Psi_J,
\qquad
\tilde \phi \coloneq
\tilde \balp^{(1)}\Psi_1 + \cdots + \tilde \balp^{(J)} \Psi_J
\]
satisfy
\[
\| g - g_{\phi, \tilde \phi} \|_{\Phi} \le N/\cF_*(d), \qquad
\| g_{\phi, \tilde \phi} - g_{\phi} \|_2 \le \eta \sqrt N.
\]
Define $f, f_{\unf}, g_{\str}, g_{\sml}: [N] \to [-1,1]$ by
\begin{gather}
f = g_{\phi,\tilde \phi}, \qquad
f_\unf = g - f, \\
g_\str = g_{\phi}, \qquad
g_\sml = f - g_\str.
\end{gather}
Then we have
\begin{gather}
f \ge 0, \qquad
\| f_\unf \|_{\Phi} \le N/\cF_*(d), \\
\sum_{n \in [N]} g_\str(n) \ge \del N, \qquad
\sum_{n \in [N]} |g_\sml(n)|^2 \le \eta^2 N.
\end{gather}
Moreover, Lemma \ref{Lip} provides a $4$-Lipschitz function $G: \bT^d \to [0,1]$ such that
\[
g_{\str}(n) = G(\phi(n)) \qquad (1 \le n \le N).
\]

Next, we apply Lemma \ref{Ratner} with
\[
M_n = \cF_n(d)
\quad (0 \le n \le D), \qquad
D = Jd.
\]
Thus there is $0\leq n \leq D$ and decompositions
\[
\balp^{(j)} = \ba^{(j)}/q + \bbet^{(j)}/N^{k_j} + L_j \btet^{(j)}
\qquad (1 \le j \le J),
\]
such that on setting
\begin{equation}
Q = 
\begin{cases}
1,& \text{if } n = 0, \\
\sum_{i=0}^{n-1} \mathcal F_0(d)\cdots \mathcal F_i(d),&\text{if } n \ne 0,
\end{cases}
\end{equation}
we have:
\begin{itemize}
\item $q \in \bN$ with $q \le Q$
\item $\ba^{(j)} \in \bZ^d$
\item $\bbet^{(j)} \in \bR^d$ with $\| \bbet^{(j)} \|_1 \le Q$
\item $\btet^{(j)}$ is $(\cF_n(d), N^{k_j})$-irrational with
$\btet^{(j)} \in \bT^{c_j}$ for some $c_j \le d$
\item $L_j$ is a $d \times c_j$ integer matrix whose rows and columns are bounded in $L^1$ by $Q$.
\end{itemize}

We partition $[N]$ into $O(qQ/\eta)$ many arithmetic progressions of common difference $q$ and diameter at most $\eta N / Q$. By Lemma \ref{pigeonhole}, there exists an arithmetic progression $\cP \subseteq [N]$ of common difference $q$ such that
\begin{align*}
|\cP| &\gg \del \eta N / (qQ), \\
\sum_{n \in \cP} g_\str(n) &\gg \del |\cP|, \\
\sum_{n \in \cP} |g_\sml(n)|^2 &\ll \del^{-1} \eta^2 |\cP|.
\end{align*}

Fix $b \in \cP$ arbitrarily, and define
$
\tilde G: \bT^{c_1} \times \cdots \times \bT^{c_J}
\to [0,1]
$
by
\[
\tilde G (\bgam^{(1)}, \cdots, \bgam^{(J)}) =
G \Bigl(
\sum_{j \in [J]}
(\ba^{(j)} \Psi_j(b)/q + \bbet^{(j)} \Psi_j(b)/N^{k_j} +  L_j \bgam^{(j)})
\Bigr).
\]
Recall that $G$ is $4$-Lipschitz so, for $n \in \cP$,
\begin{align*}
g_\str(n) &=
 G  \Bigl(
\sum_{j \in [J]}
(\ba^{(j)} \Psi_j(n)/q + \bbet^{(j)} \Psi_j(n)/N^{k_j} + \Psi_j(n) L_j \btet^{(j)})
\Bigr)
\\
&=  G \Bigl(
\sum_{j \in [J]}
(\ba^{(j)} \Psi_j(b)/q + \bbet^{(j)} \Psi_j(n)/N^{k_j} + \Psi_j(n) L_j \btet^{(j)})
\Bigr)
\\
&= \tilde G (\btet^{(1)} \Psi_1(n), \ldots, \btet^{(J)} \Psi_J(n))
+ O_{\boldsymbol \Psi}(\eta).
\end{align*}
Moreover $\tilde G$ is $4Q$-Lipschitz so, by Lemma \ref{trig}, there exists a trigonometric polynomial 
\[
F: \bT^{c_1} \times \cdots \times \bT^{c_J} \to [0,1]
\]
of degree $O(Q^2 (Jd)^2 \eta^{-3})$ such that $\| \tilde G - F \|_\infty \le \eta$. We define, for $n \in \cP$,
\begin{align*}
f_\str(n) &=
F(\btet^{(1)} \Psi_1(n), \ldots, \btet^{(J)} \Psi_J(n)),
\\
f_\sml(n) &=
(g_\sml + g_\str - f_\str)(n).
\end{align*}

Next, we verify the first list of  assertions.
\begin{enumerate}[(a)]
\item These hold with the choice $M = \lfloor Q^{2} d^2\eta^{-4} \rfloor$.
\item This holds because
$\cF_n(d) \ge \cF(M)$.
\item This holds because $|\cP| \gg \del \eta N / (qQ)$ and $c$ is sufficiently small.
\item The two equations hold by construction.
\end{enumerate}
Finally, we verify the second list of  assertions.
\begin{enumerate}[(i)]
\item We have already observed this.
\item This holds because
$\mathcal F_*(d)\geq \mathcal F_n(d) \ge \cF(M)$.
\item As
$
\| f_\sml - g_\sml \|_{L^\infty(\cP)} \ll_{\boldsymbol \Psi} \eta$ and
$\| g_\sml 1_{\cP} \|_2 \ll \del^{-1/2} \eta \| 1_\cP \|_2,
$
the triangle inequality implies
$
\| f_\sml \|_2 \ll \del^{-1/2} \eta \| 1_{\cP} \|_2.
$
Since $c$ is small and $\del \le 1$, it follows that $\|f_\sml\|_2 \le \eps \| 1_{\cP} \|_2$.
\item As $\|f_\str - g_\str\|_{L^\infty(\cP)} \ll_{\boldsymbol \Psi} \eta$ and $\| g_\str 1_{\cP} \|_1 \gg \del |\cP|$, the triangle inequality gives $\| f_\str \|_1 \gg \del |\cP|$.
\item The degree is bounded by $M$ because $c$ is sufficiently small.
\end{enumerate}
\end{proof}

\section{Quantitative weak approximation}
\label{S:QWA}

Let
\[
\fF(\bx) = \sum_{i \le s} a_i x_i^k,
\qquad
\fL(\bx) = \sum_{i \le s} b_i x_i.
\]
We may assume that $\ba$ and $\bb$ are primitive, i.e.
\[
\gcd(a_1, \ldots, a_s) = \gcd(b_1, \ldots, b_s) = 1.
\]
For arithmetic functions $g_1,\ldots,g_s,g$ supported on $[N]$, define
\[
T(\bg) = \sum_{\fF(\bx) = \fL(\bx)=0} \: \prod_{i \le s} g_i(x_i), \qquad
T(g) = T(g,\ldots,g),
\]
as well as
\[
\tilde g(\alp,\bet) = \sum_{n \le N} g(n) e(\alp n^k + \bet n) \qquad
(\alp,\bet \in \bT)
\]
and
\[
\| g \|_{\FL} = \sup_{\alp,\bet \in \bT} |\tilde g(\alp,\bet)|.
\]
Additionally, write
\[
\tilde \bg(\alp,\bet) = \prod_{i \le s} \tilde{g_i} (a_i \alp, b_i \bet) \qquad (\alp,\bet \in \bT)
\]
and
\[
\| \bg \|_{\FL} = \sup_{\alp,\bet \in \bT} |\tilde \bg(\alp,\bet)|.
\]
Observe that
\[
T(\bg) = \int_{\bT^2} \tilde \bg(\alp,\bet) \d \alp \d \bet.
\]

Let $\cP = (L,R] \cap (u + m \bZ)$ be an arithmetic progression in $[N]$, where 
\begin{equation}
    \label{eq:P-def}
    1 \le u \le m \le M, \quad
    L \ge u, \quad M \ge M_0(\fF,\fL), \quad |\cP| \ge N/M, \quad N \ge N_0(M),
\end{equation}
and put $g = \gcd(m,u)$. In this section, we show that
\begin{equation} \label{T1a}
T(1_\cP) \gg \frac{g^{k-1} \# \cP^{s-2}}{R^{k-1}}
\end{equation}
in Case \ref{case:neq} and
\begin{equation} \label{T1b}
T(1_\cP) \gg 
\frac{g^{k-2} \# \cP^{s-3}}{R^{k-2}}
\end{equation}
in Case \ref{case:eq}.

\begin{remark} [Adjusted naive heuristic] Let us view $\bx \in \cP^s$ as a random variable. Put 
\[
\bu = (u,\ldots,u), \qquad \bu + m \bz = \bx = (R-y_1,\ldots, R-y_s),
\]
and note from \eqref{zerosum} that $\fF(1,\ldots,1) = \fL(1,\ldots,1) = 0$. As
\[
\fF(\bx) = \fF(\bu + m \bz) \equiv 0 \mmod g^{k-1} m
\]
and
\[
\fF(\bx) = \fF(R-y_1,\ldots,R-y_s) \ll R^{k-1} (R-L),
\]
the probability that $\fF(\bx) = 0$ is roughly
\[
\frac{g^{k-1}m} { R^{k-1} (R-L)},
\]
and similarly the probability that $\fL(\bx) = 0$ is $m/(R-L)$. 

There is an extra congruence in Case \ref{case:eq}:
\[
\fF(\bu+m\bz) \equiv ku^{k-1} L(\bu + m\bz) \mmod g^{k-2} m^2.
\]
There is also an extra archimedean factor in Case \ref{case:eq}, for if $\fL(\bx) = 0$ then
\[
\fF(\bx) = \fF(R-y_1,\ldots,R-y_s) \ll R^{k-2}(R-L)^2.
\]
Thus, we expect roughly
\[
\frac{g^{k-1} m^2 \# \cP^s}{R^{k-1}(R-L)^2} \sim \frac{g^{k-1} \# \cP^{s-2}}{R^{k-1}}
\]
many points in $\cV(P)$ in Case \ref{case:neq}, and roughly 
\[
\frac{g^{k-2} \# \cP^{s-3}}{R^{k-2}}
\]
many points in Case \ref{case:eq}.
\end{remark}

By homogeneity, we may assume that $g=1$. We proceed via the circle method \cite{Vau1997}. For $\alpha , \beta \in \T$, let
\begin{equation}
    f(\alpha, \beta) = \sum_{n\in \mathcal P}e(\alpha n^{k} + \beta n).
\end{equation}
we abbreviate
\begin{equation}
    f_i(\alpha, \beta) = f(a_i \alpha, b_i \beta) \quad(1\leq i \leq s),
    \qquad
    F(\alpha, \beta) = \prod_{i\leq s}f_i(\alpha, \beta).
\end{equation}
By orthogonality, we have
\begin{equation}
    T(\indicator{\mathcal P}) = \int_{\T^{2}} F(\alpha,\beta)\d \alpha \d \beta.
\end{equation}

Let us fix a small constant $\Delta > 0$. For $q \in \N$ and $a,b,\in \Z$, denote by $\Ma(q; a, b)$ the set of $(\alpha, \beta)\in\T^{2}$ such that
\begin{equation}
    \label{eq:major-arc-def}
    \abs{q \alpha - a} < N^{\Delta -k},
    \qquad
    \abs{q \beta - b} < N^{\Delta -1}.
\end{equation}
These are called \emph{major arcs}, and we denote by $\Ma = \Ma^{(\Delta)}$ their union over
\begin{equation}
    \label{eq:major-arc-range}
    q < N^{\Delta},
    \qquad
    \gcd(q, a, b) = 1.
\end{equation}
The \emph{minor arcs} are given by $\ma = \T^{2}\setminus\Ma$.

The major arcs are pairwise disjoint. To see this, suppose that
\begin{equation}
    \Ma(q_1; A_1, b_1)
    \cap
    \Ma(q_2; A_2, b_2)
\end{equation}
is non-empty for some $q_1, A_1, B_1,q_2, A_2, B_2$ satisfying
\begin{equation}
    q_j < N^{\Delta},
    \quad
    \gcd(q_j, A_j, B_j) = 1
    \qquad (j=1,2).
\end{equation}
By the triangle inequality
\begin{equation}
    \abs{q_2A_1 - q_1A_2} < 2N^{2 \Delta -k} < 1,
    \qquad
    \abs{q_2B_1 - q_1B_2} < 2N^{2 \Delta -1} < 1,
\end{equation}
so
\begin{equation}
    q_2A_1 = q_1A_2,
    \qquad
    q_2B_1 = q_1B_2,
\end{equation}
which means that $(q_1, A_1, B_1)$ and $(q_2, A_2, B_2)$ are linearly dependent. Coprimality then forces $(q_1, A_1, B_1)$ and $(q_2, A_2, B_2)$ to coincide, and we conclude that the major arcs are indeed disjoint.

\bigskip

In the sequel, let $\eps$ be a small, positive constant, whose value is allowed to differ between instances. The implied constants in our asymptotic notations are allowed to depend on $\eps$.

\subsection{Minor arcs}

\begin{lem}
\label{lm:fi-minor-arc-bound}
If $i\in[s]$ and $(\alpha, \beta)\in\ma$ then
\begin{equation}
f_i(\alpha, \beta) \ll N^{1+ \epsilon - \Delta/k}.
\end{equation}
\end{lem}

\begin{proof}
    Observe that
    \begin{equation*}
        f_i(\alpha, \beta)
        = \sum_{L<mx + u\leq R}e(a_i\alpha(mx+u)^{k} + b_i\beta (mx + u))
        = \sum_{\frac{L-u}{m} < x \leq \frac{R-u}{m}} e(h_i(x)),
    \end{equation*}
    for the polynomial $h_i$ given by
    \begin{equation}
        \label{eq:hi-def}
        h_i(x) = h_i(x; \alpha, \beta) = \eta_{i, k} x^{k} +\dots+ \eta_{i, 1} x + \eta_{i, 0},
    \end{equation}
    where
    \begin{equation}
        \label{eq:hi-coefficients}
        \begin{gathered}
            \eta_{i, j} = a_i\alpha m^{j}u^{k-j} \binom k j\qquad (2\leq j \leq k),
            \\
            \eta_{i, 1} = a_i\alpha m u^{k-1}k + b_i\beta m,
            \qquad
            \eta _{i, 0} = a_i\alpha u^{k} + b_i\beta u.
        \end{gathered}
    \end{equation} 

    Suppose
    \begin{equation}
        \abs{f_i(\alpha, \beta)} > 2N^{1+\epsilon - \Delta/k}.
    \end{equation}
    Then for some $X\in\{(L-u)/m, (R-u)/m\}$, we have
    \begin{equation}
    \abs[\Big]{\sum_{x\leq X} e(h_i(x; \alpha, \beta))} > N^{1+ \epsilon - \Delta/k}.
    \end{equation}
    Now, by \cite[Theorem 5.1]{Bak1986} applied with $M=1$ and $P=N^{1+ \epsilon - \Delta/k}$, there exist $r\in\N$ and $\vec u\in\Z ^{k}$ for which 
    \begin{equation}
        r < N^{\Delta}/\log N,\qquad \gcd(r, u_k, \dots, u_1) = 1,
    \end{equation}
    and
    \begin{equation}
        \abs{r\eta_{i, j} - u_j} < \frac{N^{\Delta -j}}{\log N}
        \qquad (1\leq j \leq s).
    \end{equation}
    By the triangle inequality, we have
    \begin{equation}
        \abs{rb_im^{k}\beta - (u_1m^{k-1}-ku_ku^{k-1})} < \frac{2m^{k-1}N^{\Delta -1}}{\log N}.
    \end{equation}
    Set
    \begin{equation}
        q' = ra_ib_im^{k}, \qquad a' = b_i u _k, \qquad b'=a_i(u_1m^{k-1}-ku_2u^{k-1}).
    \end{equation}
    Then \eqref{eq:major-arc-def} and \eqref{eq:major-arc-range} are satisfied by
    \begin{equation}
        q = \frac{q'}{\gcd(q', a', b')},\qquad a = \frac{a'}{\gcd(q', a', b')}, \qquad b = \frac{b'}{\gcd(q', a', b')},
    \end{equation}
    so $(\alpha, \beta) \in \Ma$.
\end{proof}

The following bound is a direct consequence of the results from \cite[\S 14]{Woo2019}.
\begin{lem}
\label{lm:moment-bound}
For $i \in [s]$,
\begin{equation}
\int_{\T^{2}}
\abs{ f_i(\alpha, \beta)}^{s_0 -1} \d \alpha \d \beta \ll N^{s_0 - k - 2 + \epsilon}.
\end{equation}
\end{lem}

\begin{proof}
    Define
    \begin{equation}
        f'(\alpha, \beta) = \sum_{n\leq N}e(\alpha n^{k} + \beta n).
    \end{equation}
    For each natural number $t$, by orthogonality we may interpret
    \begin{equation}
        \int_{\T^{2}}\abs{f(\alpha, \beta)}^{2t} \d \alpha \d \beta
        \qquad \text{and} \qquad
        \int_{\T^{2}}\abs{ f'(\alpha, \beta)}^{2t} \d \alpha \d \beta
    \end{equation}
    as counts of solutions to
    \begin{align}
        x_1^{k} + \dots + x_{t}^{k} &= 
        y_1^{k} + \dots + y_{t}^{k},
        \\
        x_1 + \dots + x_{t} &= 
        y_1 + \dots + y_{t},
    \end{align}
    in $\mathcal P^{2t}$ and in $[N]^{2t}$, respectively.
    Noting that $s_0$ is an odd integer, it follows that
    \begin{equation}
        \label{eq:no-progression-moment-is-larger}
        \int_{\T^{2}}\abs{f_i(\alpha, \beta)}^{s_0 -1} \d \alpha \d \beta
        = \int_{\T^{2}}\abs{f(\alpha, \beta)}^{s_0 -1} \d \alpha \d \beta \leq 
        \int_{\T^{2}}\abs{f'(\alpha, \beta)}^{s_0 -1} \d \alpha \d \beta,
    \end{equation}
    so it suffices to prove the bound with $f'$ in place of $f_i$.

    The case $k=2$ is a classical result, see for instance Rogovskaya \cite{Rog1986}.
    When $3\leq k \leq 6$ the result follows from the discussion after \cite[Equation (14.26)]{Woo2019}.
    Finally, when $k\geq 7$, put
    \begin{equation}
        r = 1 + k - \lfloor \sqrt{2k+2} \rfloor,
        \qquad
        v = k(k+1) - \frac{k(k+1)- r(r+1)}{1 + k -r}.
    \end{equation}
    Then we have
    \begin{align}
        v
        &= k^{2}-r -2 + \frac{2k + 2}{1+k-r} 
        \\
        &= k(k-1) + \lfloor \sqrt{2k+2} \rfloor - 3 + \frac{2k + 2}{\lfloor \sqrt{2k+2} \rfloor}
        \\
        &= k(k-1) + 2\lfloor \sqrt{2k+2} \rfloor - 3 + \frac{2k + 2 - \lfloor \sqrt{2k+2} \rfloor^{2}}{\lfloor \sqrt{2k+2} \rfloor},
    \end{align}
    and since
    \begin{equation}
        \frac{2k + 2 - \lfloor \sqrt{2k+2} \rfloor^{2}}{\lfloor \sqrt{2k+2} \rfloor} \in [0, 2],
    \end{equation}
    it is clear that $v \geq k(k-1) + 2$. An application of \cite[Theorem 14.5]{Woo2019} now yields
    \begin{equation}
        \int_{\T^{2}}\abs{f'(\alpha, \beta)}^{v} \d \alpha \d \beta \ll N^{v-k-1 + \epsilon}.
    \end{equation}
    It follows from \eqref{eq:s_0-large-k-def} and \eqref{eq:theta-k-def} that $v \leq s_0 -1$, concluding the proof of the lemma.
    \end{proof}
    \begin{remark}
    A direct application of \cite[Corollary 14.8]{Woo2019} when $k\geq 7$ is insufficient for the proof above, as we do not require $\tau >0$. Indeed, a comparison of \eqref{eq:theta-k-def} and \cite[Equation (14.22)]{Woo2019} reveals that we get an improvement in the case of equality, achieved when $2k + 2$ is of the form $2k+2 = n(n+1)$.

    When $3\leq k \leq 8$, a more careful analysis of the argument following \cite[Equation (14.26)]{Woo2019} yields the same moment bounds with $f$ in place of $f'$. This however is not enough to improve the number of variables needed for our main theorem, as we require a bound at an even moment below $s_0$ for the proof of the restriction estimate \eqref{eq:restricion-estimate}.
    \end{remark}

The combination of the previous two lemmas now furnishes the estimate
\begin{equation}
    \label{eq:moment-bound-minor-arc}
    \int_{\ma}\abs{f_i(\alpha, \beta)}^{s}\d \alpha\d \beta
    \ll
    N^{1 + \epsilon-\Delta/k}\int_{\T^{2}}\abs{f_i(\alpha, \beta)}^{s-1}\d \alpha \d \beta
    \ll N^{s -k -1 -\frac{\Delta}{2k}}
\end{equation}
for $i \in [s]$.
Now by H\"older's inequality,
\begin{equation*}
\int_{\ma}\abs{F(\alpha,\beta)}\d \alpha\d \beta \ll N^{s -k -1 - \frac{\Delta}{2k}} = o(\# \mathcal P^{s-2}/R^{k-1})
\end{equation*}
and thus
\begin{equation}
\label{eq:minor-arc-result}
T(\indicator{\mathcal P}) = \int_{\Ma} F(\alpha , \beta)\d \alpha\d \beta  + o(\#\mathcal P^{s-2}/R^{k-1})
\end{equation}
as $N \to \infty$.

\subsection{Major arcs}

Suppose $(\alpha, \beta) \in \Ma(q; a, b)$, where $q\in \N$ and $a, b\in\Z$ satisfy \eqref{eq:major-arc-range}.
For $i\in[s]$, let
\begin{equation}
    \label{eq:u_ij-definition}
    \begin{gathered}
    u_{i, j} = a_ia m^{j}u^{k-j} \binom k j \qquad (2\leq j \leq k),
    \\
    u_{i, 1} = a_ia m u^{k-1}k + b_ib m,
    \qquad
    u_{i, 0} = a_ia u^{k} + b_ib u.
\end{gathered}
\end{equation}
Recalling \eqref{eq:hi-coefficients}, we see from \eqref{eq:major-arc-def} and the triangle inequality that
\begin{equation}
    \abs{q\eta_{i, j} - u_{i, j}} \ll M^{k}N^{\Delta -j}
    \qquad (0\leq j \leq k).
\end{equation}
Write
\begin{gather}
    \beta_{i, j} = \eta_{i, j} - \frac{u_{i,j}}{q} \quad(0\leq j \leq k),
    \\
    g_i(x) = \sum_{j = 0}^{k}\beta_{i, j}x^{j},
    \qquad
    v(\vec \beta^{(i)}) = \int_{(L-k)/m}^{(R-k)/m} e(g_i(y))\d y,
    \\
    G_i(x) = \sum_{j = 0}^{k}u_{i, j}x^{j},
    \qquad
    S_i(q; a, b) = \sum_{x\leq q} e_q(G_i(x)).
\end{gather}
It follows from \cite[Lemma 4.4]{Bak1986} that
\begin{equation}
    \label{eq:fi-major-arc-approx}
    f_i(\alpha, \beta)
    = q^{-1}S_i(q; a, b)v(\vec \beta^{(i)}) + O(q^{1+\epsilon})
    \qquad (1\leq i \leq s).
\end{equation}

Since the major arcs are disjoint, we may define
\begin{equation*}
    f^{*}_i(\alpha, \beta)
    = q^{-1}S_i(q; a, b)v(\vec \beta^{(i)})
\end{equation*}
for $i\in [s]$ and $(\alpha, \beta) \in \Ma$, as well as
\begin{equation*}
    F^{*}(\alpha, \beta) = \prod _{i\leq s} f_i^{*}(\alpha, \beta).
\end{equation*}
By the telescoping identity
\begin{equation}
\label{eq:telescoping-identity}
x_1 \cdots x_s - y_1 \cdots y_s = \sum_{i\leq s}\Bigl((x_i - y_i)\prod_{j< i}x_j \prod_{j> i}y_j\Bigr),
\end{equation}
we have
\begin{equation}
\label{eq:major-arc-pointwise-bound}
F(\alpha,\beta) - F^{*}(\alpha,\beta) \ll q^{1+\epsilon}N^{s-1},
\end{equation}
and so
\begin{equation*}
    \int_{\Ma} F(\alpha, \beta)\d \alpha\d \beta - \int_{\Ma} F^{*}(\alpha, \beta)\d \alpha\d \beta
    \ll N^{s -2 -k +5 \Delta} = o(\#\mathcal P^{s-2}/R^{k-1}).
\end{equation*}
Pairing this with \eqref{eq:minor-arc-result} gives
\begin{equation}
    \label{eq:major-arc-result}
    T(\indicator{\mathcal P})
    = \int_{\Ma} F^{*}(\alpha, \beta)\d \alpha\d \beta
    + o(\#\mathcal P^{s-2}/R^{k-1}).
\end{equation}

Write
\begin{equation}
    \tau = \alpha -a/q,
    \qquad
    \kappa = \beta -b/q,
\end{equation}
and observe that
\begin{equation}
    g_i(x) = h_i(x; \tau, \kappa),
\end{equation}
where we recall \eqref{eq:hi-def}. Therefore, 
\begin{equation}
    v(\vec \beta^{(i)}) = \int_{(L-u)/m}^{(R-u)/m}e(a_i\tau (my+u)^{k} + b_i \kappa (my+u)) \d y.
\end{equation}
Changing variables, we find that
\begin{equation}
    \label{eq:vi-def}
    v(\vec \beta^{(i)}) = m^{-1}\int_L^{R}e(a_i\tau x^{k} + b_i \kappa x) \d x \eqcolon v_i(\tau, \kappa).
\end{equation}

Denote
\begin{equation}
    V(\tau, \kappa) = \prod_{i\leq s}v_i(\tau, \kappa).
\end{equation}
Applying the classical inequality \cite[Theorem 7.3]{Vau1997}, we may now harvest the estimate
\begin{equation}
    V(\tau, \kappa)^{1/s} \ll \frac{R}{m(1+R\abs \kappa + R^{k}\abs \tau)^{1/k}}.
\end{equation}
We relax this to
\begin{equation}
    \label{eq:V-bound}
    V(\tau, \kappa) \ll N^{s} (1+N\abs\kappa)^{-\eta}(1+N^{k}\abs\tau)^{\eta - s/k},
\end{equation}
where we will choose $\eta \in (1, s/k -1)$.

Next we bound $S_i(q; a,b)$. Writing
\begin{equation}
    d = \gcd(q, u_{i, k},\dots, u_{i, 1}),
    \qquad
    \tilde q = q/d,
    \qquad
    \tilde u_{i, j} = u_{i, j}/d \quad (1\leq j \leq k),
\end{equation}
and exploiting periodicity, gives
\begin{equation*}
    S_i(q; a, b) = e_q(u_{i, 0})d\sum_{x\leq \tilde q}e_{\tilde q}\Bigl(\sum_{j = 1}^{k}\tilde u_{i, j}x^{j}\Bigr).
\end{equation*}
By \cite[Theorem 7.1]{Vau1997}, we thus have
\begin{equation}
    \label{eq:Si-bound-full-gcd}
    S_i(q; a, b) \ll d {\tilde q}^{1-1/k + \epsilon} \leq d^{1/k}q^{1-1/k + \epsilon}.
\end{equation}
Since
\begin{equation}
\label{eq:gcd-bound}
d \leq \gcd(q, u_{i,k}, m^{k-1}u_{i,1}-u^{k-1}u_{i,k})
= \gcd(q, m^{k}a_ia, m^{2}b_ib)
\ll \gcd(q,m^{k}),
\end{equation}
we obtain
\begin{equation}
    \label{eq:Si-bound}
    S_i(q; a,b)
    \ll \gcd(q,m^{k})^{1/k} q^{1 - 1/k + \epsilon}
    \leq mq^{1 - 1/k + \epsilon}
    \leq Mq^{1 - 1/k + \epsilon}
\end{equation}
for $1\leq i \leq s$.
We require more precision in the case $q = p^{t}$: here \cite[Equation (7.9)]{Vau1997} yields
\begin{equation}
    \label{eq:Si-bound-prime-pow}
    S_i(p^{t}; a, b) \ll d \tilde q ^{1- 1/k}= d^{1/k}p^{t-t/k}.
\end{equation}

For $q\in\N$ and $a,b\in\Z$, write
\begin{equation}
    \label{eq:S-def}
    S(q; a,b) = \prod_{i\leq s}S_i(q; a,b),
    \qquad
    S(q) = q^{-s}\sum_{\substack{a,b\leq q\\\gcd(q,a,b) = 1}}S(q;a,b).
\end{equation}
By \eqref{eq:Si-bound},
\begin{equation}
    \label{eq:S-bound}
    S(q)
    \ll M^{s}q^{2-s/k + \epsilon}.
\end{equation}

We are now equipped to extend to infinity. As the major arcs are disjoint, we have
\begin{equation}
    \int_{\Ma}F^{*}(\alpha, \beta)\d \alpha \d \beta = \sum_{q< \Delta}S(q)\int_{\abs{q\tau}<N^{\Delta -k}}\int_{\abs{q\kappa}<N^{\Delta -1}} V(\tau, \kappa)\d\kappa\d\tau.
\end{equation}
We compute using \eqref{eq:V-bound} that
\begin{equation}
    \int_{\abs{q\kappa}\geq N^{\Delta -1}} \abs{V(\tau, \kappa)}\d\kappa
    \ll \frac{N^{s - \eta}\int_{N^{\Delta -1}/q}^{\infty}\kappa^{-\eta}\d\kappa}{(1+N^{k}\abs\tau)^{s/k - \eta}}
    \ll \frac{q^{\eta-1}N^{s-1-\Delta(\eta -1)}}{(1+N^{k}\abs\tau)^{s/k - \eta}}.
\end{equation}
Applying \eqref{eq:S-bound} and choosing $\eta = 1 + \epsilon$ now gives
\begin{align}
    \sum_{q\geq 1} \abs{S(q)} \int_{\R}\int_{\abs{q\kappa}\geq N^{\Delta -1}}\abs{V(\tau,\kappa)}\d\kappa\d\tau
    &\ll N^{-k}N^{s-1- \Delta \epsilon}M^{s}\sum_{q\geq 1}q^{2-s/k + \epsilon} q^{\epsilon}
    \\
    &\ll N^{s-k-1 - \Delta \epsilon} M^{s}
    = o(\#\mathcal P^{s-2}/R^{k-1}),
\end{align}
since $s\geq 3k + 1$. Choosing $\eta = 1 - s/k + \epsilon$ instead, we similarly compute
\begin{align*}
\sum_{q\geq 1} \abs{S(q)} \int_{\abs{q\tau}\geq N^{\Delta -2}}
\int_{\R}\abs{V(\tau,\kappa)}\d\kappa\d\tau
&\ll N^{-1}N^{s-k-(\Delta + k -2)\epsilon}M^{s}
\\
&= o(\#\mathcal P^{s-2}/R^{k-1}).
\end{align*}
Finally,
\begin{align*}
\sum_{q\geq N^\Delta} \abs{S(q)}\int_{\R^{2}}\abs{V(\tau, \kappa)}\d\kappa\d\tau
&\ll N^{s-k-1} M^{s}\sum_{q\geq N^\Delta} q^{2-s/k + \epsilon}
    \\
&\ll N^{s-k-1} M^{s}N^{\Delta (3-s/k + \epsilon)}
    = o(\#\mathcal P^{s-2}/R^{k-1}).
\end{align*}
In view of \eqref{eq:major-arc-result}, we now have
\begin{equation}
    T(\indicator{\mathcal P}) = \mathfrak S \mathfrak J_0 + o(\#\mathcal P^{s-2}/R^{k-1}),
\end{equation}
where
\begin{equation}
    \mathfrak S = \sum_{q\geq 1} S(q),
    \qquad
    \mathfrak J_0 = \int_{\R^{2}}V(\tau, \kappa) \d \kappa\d\tau.
\end{equation}

Define
\begin{equation}
    \mathfrak J = \int_{\R^{2}}\int_B e(\tau \F(\vec x) + \kappa\L(\vec x))\d{\vec x}\d\kappa\d\tau,
    \qquad
    B = [L/R, 1]^{s}.
\end{equation}
As $\mathfrak J_0 = R^{s-k-1}m^{-s}\mathfrak J$, we at last have
\begin{equation}
    \label{eq:cm-result}
    T(\indicator{\mathcal P}) = R^{s-k-1} m^{-s}\mathfrak S \mathfrak J + o(\#\mathcal P ^{s-2}/R^{k-1}).
\end{equation}

\subsection{Singular integral}
Here we show that
\begin{equation}
    \label{eq:singular-integral-objective-A}
    \mathfrak J
    \gg \biggl(1-\frac{L}{R}\biggr)^{s-2}
\end{equation}
in Case \ref{case:neq}, and
\begin{equation}
    \label{eq:singular-integral-objective-B}
    \mathfrak J
    \gg \biggl(1-\frac{L}{R}\biggr)^{s-3}
\end{equation}
in Case \ref{case:eq}.

For $\eta \in \R$ and $T > 0$, define $ \lambda(\eta) = \max\{1-\abs \eta , 0\}$ and $\lambda_{T}(\eta) = T\lambda(T\eta)$.
For $(\eta_k, \eta_1)\in \R^{2}$ and $T>0$, define
\begin{equation}
\lambda_T(\eta_k, \eta_1) = \lambda_T(\eta_k)\lambda_T(\eta_1)
\end{equation}
and
\begin{equation}
    \mu_T = \int_{B}\lambda_T(\F(\vec x), \L(\vec x)) \d{\vec x}.
\end{equation}

Thanks to \cite[\S 3]{Sch1985}, we know that
\begin{equation}
    \mathfrak J = \lim_{T\to \infty}\mu_T.
\end{equation}
Thus, it suffices to prove that if $T$ is large, then the volume of the region
\begin{equation}
    \mathcal R \defeq \{ \vec x \in B: \abs{\F(\vec x)}, \abs{\L(\vec x)} \leq (2T)^{-1} \}
\end{equation}
satisfies
\begin{equation}
    \vol(\mathcal R) \gg T^{-2}\biggl( 1-\frac{L}{R} \biggr)^{s-2}
\end{equation}
in Case \ref{case:neq} and
\begin{equation}
    \vol(\mathcal R) \gg T^{-2}\biggl( 1-\frac{L}{R} \biggr)^{s-3}
\end{equation}
in Case \ref{case:eq}.

\bigskip

In Case \ref{case:neq}, the variety $\mathcal V$ contains the smooth real point
\begin{equation}
    \vec 1 = (1,\dots,1)\in\R^{s}.
\end{equation}
Observe that $J(\vec 1)$ is invertible, where
\begin{equation}
    J(\vec x) =
    \begin{pmatrix}
        ka_{s-1}x_{s-1}^{k-1} & ka_{s}x_{s}^{k-1} \\
        b_{s-1} & b_{s}
    \end{pmatrix}.
\end{equation}
Let $r=r(\mathcal V)$ be a small positive constant. The implicit function theorem furnishes a $C^{1}$ function $g:(1-r, 1+r)^{s-2}\to \R^{2}$ such that
\begin{equation}
    \{ (x_1,\dots,x_{s-2}, g(x_1,\dots,x_{s-2})) \} \subseteq \mathcal V,
\end{equation}
and
\begin{equation}
\partial_i g(x_1,\dots,x_{s-2}) = -J^{-1}\cdot (\partial_i\F, \partial_i\L)^{T}\Big\rvert_{(x_1,\dots,x_{s-2}, g(x_1,\dots,x_{s-2}))}
\end{equation}
for $1\leq i \leq s-2$.
In particular, we can bound the norm of the gradient of $g$ as follows:
\begin{equation}
    \norm{\nabla g(x_1,\dots,x_{s-2})}_{\infty}
    \ll \norm{J(\vec x)^{-1}}_{\infty}
    \ll \norm{J(\vec 1)^{-1}}_{\infty}
    \ll 1,
\end{equation}
for $r$ small enough,
where $\vec x = (x_1,\dots,x_{s-2}, g(x_1,\dots,x_{s-2}))$.

Now put $t=(1+\frac{L}{R})/2\in[1/2, 1-(2M)^{-1}]\subset [1/2, 1]$, and define
\begin{equation}
    \tilde g: ((1-r)t, (1+r)t)^{s-2}\to \R^{2},
    \qquad
    \tilde g(\cdot) = tg(\cdot/t).
\end{equation}
By homogeneity, we have
\begin{equation}
    \{ (x_1,\dots,x_{s-2},\tilde g(x_1,\dots,x_{s-2})) \} \subseteq \mathcal V(\R),
\end{equation}
and by the chain rule, $\norm{\nabla \tilde g}_{\infty} \ll 1$.
Finally, as $\F$ and $\L$ are uniformly Lipschitz on $B$, it follows from the mean value theorem that $\mathcal R$ contains the set of $\vec x\in \R^{s}$ for which
\begin{gather}
    \abs{x_i -t}\leq r(1-L/R) \qquad (1\leq i\leq s-2),
    \\
    \norm{(x_{s-1},x_{s})-\tilde g(x_1, \dots, x_{s-2})}_{\infty}\leq r/T.
\end{gather}
Therefore
\begin{equation}
    \vol(\mathcal R) \gg T^{-2}(1-L/R)^{s-2},
\end{equation}
completing the proof of \eqref{eq:singular-integral-objective-A} in Case \ref{case:neq}.

\bigskip

In Case \ref{case:eq}, let $\vec y'$ be a fixed smooth real point on $\mathcal V'$, normalised so that $\norm{\vec y'}_\infty = 1$. By symmetry, we may assume that $y'_{s-1} \neq y'_s$.
Let $r = r(\mathcal V, \vec y')$ be a small positive constant, and let
\begin{equation}
    \vec y = \lambda \vec y' + \mu \vec 1 \in [1/2, 1]^{s},
\end{equation}
where
\begin{equation}
    \lambda = r(1-L/R),
    \qquad \mu = \frac{1 + L/R}{2}.
\end{equation}
It follows from translation and dilation invariance that $\vec y \in \mathcal V'(\R)$.

A direct appeal to the general theory of effective radii for the implicit function theorem \cite[Theorem 3.1]{Pap2005} is insufficient for our purposes. Instead, we substitute one variable using the linear equation, and infer the implicit function using the intermediate value theorem. This substitution improves a certain Lipschitz constant close to $\vec y$, giving rise to an effective radius exceeding that available in greater generality.

If $\vec x \in \mathcal V(\R)$ then
\begin{equation}
    \label{eq:xs-as-dependent}
    x_s = -a_s^{-1}(a_1x_1 + \dots + a_{s-1}x_{s-1}),
\end{equation}
and so
\begin{align*}
    0 &= F(x_1,\dots,x_{s-1})
    \\
    &\defeq a_1 x_1^{k} + \dots + a_{s-1}x_{s-1}^{k} +(-1)^{k}a_{s}^{1-k}(a_1x_1 + \dots + a_{s-1}x_{s-1})^{k}.
\end{align*}
With \eqref{eq:xs-as-dependent}, observe that
\begin{equation}
    \partial_{s-1}F(x_1,\dots,x_{s-1}) = ka_{s-1}(x_{s-1}^{k-1} - x_{s}^{k-1})
\end{equation}
has constant sign on the region $\mathcal X$ defined by
\begin{equation}
    \abs{x_i -y_i} < r\lambda \qquad (1\leq i \leq s-1),
\end{equation}
and
\begin{equation}
    \abs{\partial_{s-1}F(x_1,\dots,x_{s-1})} \asymp \abs{x_{s-1} - x_s} \asymp \lambda.
\end{equation}
Thus, the intermediate value theorem uniquely determines a function
\begin{equation}
    g:\prod_{i\leq s-2}(y_i-r^{2}\lambda, y_i + r^{2}\lambda) \to \R^{2}
\end{equation}
whose graph lies in $\mathcal V(\R)$. It follows from the mean value theorem that $\operatorname{Lip}(F) \ll \lambda$ on $\mathcal X$.

Let
\begin{equation}
    (x_1, \dots, x_{s-2}) \in \prod_{i\leq s-2}(y_i-r^{3}\lambda, y_i + r^{3}\lambda).
\end{equation}
Then
\begin{equation}
    \abs{F(x_1, \dots, x_{s-2}, y_{s-1})} \leq r^{2}\lambda^{2},
\end{equation}
so
\begin{equation}
    \norm{g(x_1,\dots,x_{s-2}) - (y_{s-1}, y_s)}_{\infty} \leq r\lambda.
\end{equation}
Put
\begin{equation}
    (z_{s-1}, z_s) = g(x_1,\dots,x_{s-2}),
\end{equation}
and suppose
\begin{equation}
    \abs{x_{s-1}-z_{s-1}} \leq r^{2}/(\lambda T).
\end{equation}
Then
\begin{equation}
    \abs{F(x_1, \dots., x_{s-1})} \leq r/T.
\end{equation}
Next, put
\begin{equation}
    w = -a_s^{-1}(a_1x_1 + \dots + a_{s-1}x_{s-1}),
\end{equation}
so that
\begin{equation}
    \L(x_1,\dots, x_{s-1}, w) = 0
\end{equation}
and
\begin{equation}
    \F(x_1,\dots,x_{s-1}, w) = F(x_1,\ldots,x_{s-1}) \in [ -r/T, r/T ].
\end{equation}
Finally, if
\begin{equation}
    \abs{x_s -w} \leq r/T,
\end{equation}
then $\vec x \in \mathcal R$. We conclude that
\begin{equation}
    \vol(\mathcal R) \gg T^{-2}\lambda^{s-3}\gg T^{-2}(1-L/R)^{s-3},
\end{equation}
which completes the proof of \eqref{eq:singular-integral-objective-B} in Case \ref{case:eq}.

\subsection{Singular series}

Here we show that
\begin{equation}
    \label{eq:singular-series-objective-A}
    \mathfrak S \gg m^{2}
\end{equation}
in Case \ref{case:neq}, and that
\begin{equation}
    \label{eq:singular-series-objective-B}
    \mathfrak S \gg m^{3}
\end{equation}
in Case \ref{case:eq}. These are the final ingredients needed for \eqref{T1a} in Case \ref{case:neq} and \eqref{T1b} in Case \ref{case:eq}. In Case \ref{case:neq}, for instance, the inequalities \eqref{eq:singular-integral-objective-A} and \eqref{eq:singular-series-objective-A} give
\begin{equation}
    m^{-s}\mathfrak S \mathfrak J \gg m^{2-s}(1-L/R)^{s-2}\gg R^{2-s}\# \mathcal P^{s-2},
\end{equation}
and substituting this into \eqref{eq:cm-result} yields \eqref{T1a} when $g=1$.

Put $\vec u = (u, \dots, u)$.
Note that
\begin{equation}
    \label{eq:Si-as-exp-sum}
    S_i(q; a, b) = \sum_{x\leq q} e_q(aa_i(mx+u)^{k} + bb_i(mx+u)),
\end{equation}
and thus
\begin{equation}
\label{SAlt}
    S(q) = q^{-s}\sum_{\substack{a,b\leq q\\ \gcd(q,a,b) = 1}}\sum_{x_1,..,x_s\leq q}e_q(a\F(m\vec x + \vec u) + b\L(m\vec x + \vec u)).
\end{equation}
For $q\in\N$, denote by $M(q)$ the number of solutions $x_1, \dots, x_s\in [q]$ to
\begin{equation}
    \F(m\vec x + \vec u) \equiv \L(m\vec x + \vec u) \equiv 0 \mod q.
\end{equation}

\begin{lem}
    \label{lm:S-from-sol-count}
    For $q\in\N$, 
    \begin{equation}
        \sum_{d\divs q}S(d) = q^{2-s}M(q).
    \end{equation}
\end{lem}

\begin{proof}
By orthogonality, 
\begin{align*}
&M(q)
= q^{-2}\sum_{v, w\leq q} \ \sum_{x_1,\dots,x_s \leq q}e_q(v\F(m\vec x + \vec u) + w\L(m\vec x + \vec u))
\\
&= q^{-2} \sum_{d \divs q} \sum_{\substack{a,b \leq d \\ \gcd(d, a, b)=1}}
(q/d)^{s} \sum_{x_1,\dots,x_s \leq d} e_d(a\F(m\vec x + \vec u) + b\L(m\vec x + \vec u))
\\
&= q^{s-2}\sum_{d\divs q}S(d).
\end{align*}
\end{proof}

Observe using the Chinese remainder theorem that $M(\cdot)$ is a multiplicative arithmetic function. The M\"obius inversion formula now implies that $S(\cdot)$ is also multiplicative.
The series defining $\mathfrak S$ converges absolutely by \eqref{eq:S-bound}, therefore there is an absolutely convergent product representation
\begin{equation}
    \mathfrak S = \prod_{p}\chi_p,
\end{equation}
where
\begin{equation}
    \chi_p = \sum _{t\geq 0} S(p^{t}).
\end{equation}
In light of Lemma \ref{lm:S-from-sol-count}, we can interpret $\chi_p$ as the $p$-adic density of points on the variety, as follows.

\begin{cor}
    If $p$ is prime then
    \begin{equation}
        \chi_p = \lim_{t\to \infty}p^{-t(s-2)}M(p^{t}).
    \end{equation}
\end{cor}

Let $p$ be prime, and suppose that $p^{r}\edivs m$, where $r\geq 0$. For the purpose of estimating $\chi_p$, we may assume that $m = p^{r}$ and $u=1$.
Let $C$ be a large positive constant depending only on $\mathcal V$.

\subsubsection{Case \ref{case:neq}}
\textbf{Case A1:} $p\leq C$.
For $t\in\N$ large, define
\begin{equation}
    \F'(\vec x) =
    \sum_{i\leq s}a_i\sum_{j=1}^{k}\binom k j p^{r(j-1)}x_i^{j}
    = p^{-r}\F(p^{r}\vec x + \vec 1).
\end{equation}
Observe that $M(p^{t})$ is $p^{rs}$ times the number of solutions $\vec x \in(\Z/p^{t-r}\Z)^{s}$ to the system
\begin{equation}
    \label{eq:system-lower-pow}
    \F'(\vec x) \equiv \L(\vec x) \equiv 0 \mod p^{t-r}.
\end{equation}
Writing $\nu_p$ for $p$-adic order, note that
\begin{equation}
    \delta \defeq\nu_p(\det(J(\vec 0))) = \nu_p(k(a_{s-1}b_s -a_sb_{s-1})) < \infty,
\end{equation}
where
\begin{equation}
    J = 
    \begin{pmatrix}
        \partial_{s-1}\F' & \partial_{s}\F'\\
        \partial_{s-1}\L & \partial_{s}\L
    \end{pmatrix}.
\end{equation}
Take $x_1,\dots,x_{s-2}\in \Z/p^{t-r}\Z$ such that
\begin{equation}
    x_i \equiv 0 \mod p^{2 \delta + 1}\qquad (1\leq i \leq s-2).
\end{equation}
Since $\F'(\vec 0) \equiv \L(\vec 0) \mod p^{2 \delta + 1}$,
Hensel's lemma \cite[Proposition 5.20]{Gre1969} yields $(x_{s-1}, x_s) \in (\Z/p^{t-r}\Z)^{2}$
such that $\vec x = (x_1,\dots,x_s)$ is a solution to \eqref{eq:system-lower-pow}.

There are $p^{(t-r-2 \delta -1)(s-2)}$ many choices for $x_1,\dots,x_{s-2}$, and thus
\begin{equation}
    M(p^{t}) \geq p^{(t-r-2 \delta -1)(s-2) + rs}.
\end{equation}
Therefore,
\begin{equation}
    \chi_p\geq p^{2r - (2 \delta +1)(s-2)}\gg p^{2r}.
\end{equation}

\textbf{Case A2:} $p> C$.
By \eqref{SAlt},
\begin{equation}
    \label{eq:S-bound-up-to-r}
    S(0) = 1,
    \qquad
    S(p^{t}) = p^{2t} - p^{2(t-1)} \quad (1\leq t \leq r).
\end{equation}

We now consider $t > r$. By \eqref{eq:u_ij-definition} and \eqref{eq:Si-bound-prime-pow},
\begin{equation}
    \label{eq:Si-gcd-bound-at-power-p}
    S_i(p^{t}; a, b)
    \ll \gcd(p^{t}, u_{i, 1})^{1/k}p^{t- t/k}
    = \gcd(p^{t-r}, ka_ia+b_ib)^{1/k}p^{t- (t-r)/k},
\end{equation}
and therefore,
\begin{equation}
    \label{eq:S-gcd-bound-at-power-p}
    S(p^{t}) \ll \sum_{\substack{a,b\leq p^{t}\\ p\ndivs\gcd(a,b)}}\prod_{i \leq s}\left(\frac{\gcd(p^{t-r}, ka_ia + b_ib)}{p^{t-r}}\right)^{1/k}.
\end{equation}
Suppose that $p\ndivs \gcd(a, b)$, and for some $i\neq j$,
\begin{equation}
    \label{eq:more-than-one-multiple-of-p}
    ka_ia +b_ib \equiv ka_ja + b_jb \equiv 0 \mod p.
\end{equation}
Then
\begin{align}
    ka(a_ib_j -a_jb_i)\equiv b(b_jb_i-b_ib_j)\equiv 0 \mod p,
    \\
    b(b_ia_j-b_ja_i)\equiv ka(a_ja_i-a_ia_j)\equiv 0 \mod p.
\end{align}
Since $p\ndivs \gcd(a,b)$, we must have $p\divs k(a_ib_j-a_jb_i)$, contradicting $p> C$.
It follows that $p\divs ka_ia + b_ib$ for at most one value of $i$, and thus by \eqref{eq:S-gcd-bound-at-power-p},
\begin{align}
S(p^{t})
&\ll p^{2t}\cdot p^{-s(t-r)/k} + p^{-(s-1)(t-r)/k}
    \#\bigl\{a, b\leq p^{t}: p\divs\prod_{i\leq s}(ka_ia+b_ib) \bigr\} \\
\label{eq:S-bound-at-large-power-p}
&\ll p^{2t}\cdot p^{-s(t-r)/k} + p^{2t-1}\cdot p^{-(s-1)(t-r)/k}.
\end{align}
Now
\begin{equation}
\label{eq:sum-S-large-t}
\sum_{t>r} S(p^{t}) \ll p^{2r + 2 -s/k}
\leq p^{2r -3/2},
\end{equation}
whence $\chi_p \geq p^{2r}(1 + O(p^{-3/2}))$.

Finally, since $\prod_{p>C} (1+O(p^{-3/2})) \gg 1$, the combination of both cases delivers \eqref{eq:singular-series-objective-A} in Case \ref{case:neq}.

\subsubsection{Case \ref{case:eq}}
\textbf{Case B1:} $p\leq C$.

\underline{Case B1 (i): $r \leq C$.}
Let $\vec y' \in \mathcal V'(\Z_p)$ be a smooth point. By translation invariance we may take $y'_s=0$. By dilation invariance, relabelling if necessary, we may further assume that $p\ndivs y'_{s-1}$.
Let $t\in\N$ be large, and observe that
\begin{equation}
    M(p^{t}) = p^{rs}\#\{ \vec z \in [p^{t}]^{s}: \vec z \equiv \vec 1 \mod p^{r},\, \F(\vec z)\equiv \L(\vec z) \equiv 0 \mod p^{t} \}.
\end{equation}
Define
\begin{equation}
    \vec y = p^{C}\vec y' + \vec 1,
    \qquad
    J =
    \begin{pmatrix}
        \partial_{s-1}\F & \partial_{s}\F\\
        \partial_{s-1}\L & \partial_{s}\L
    \end{pmatrix}.
\end{equation}
By translation and dilation invariance, we have $\F(\vec y) = \L(\vec y) = 0 \in \Z_p$. Let
\begin{equation}
    \delta \defeq \nu_p(\det(J(\vec y)))
    = \nu_p(ka_{s-1}a_s(y_{s-1}^{k-1}-y_s^{k-1}))
    = \nu_p(k(k-1)a_{s-1}a_s p^{C}),
\end{equation}
and choose $x_1,\dots,x_{s-2}\in\Z/p^{t}\Z$ such that
\begin{equation}
    x_i \equiv y_i \mod p^{2 \delta + 1} \qquad (1\leq i \leq s-2).
\end{equation}
Then Hensel's lemma \cite[Proposition 5.20]{Gre1969} yields $(x_{s-1}, x_s) \in (\Z/p^{t}\Z)^{2}$ such that $\vec x = (x_1,\dots,x_s)$ satisfies
\begin{equation}
    \vec x \equiv \vec y \equiv \vec 1 \mod p^{r},
    \qquad
    \F(\vec x) \equiv \L(\vec x) \equiv 0 \mod p^{t}.
\end{equation}
Therefore,
\begin{equation}
    M(p^{t}) \geq p^{rs + (t-2 \delta -1)(s-2)}.
\end{equation}
As $p,r \leq C$, we thus have
\begin{equation}
    \chi_p \gg p^{rs - 2r(s-2)} \gg 1 \gg p^{3r}.
\end{equation}

\underline{Case B1 (ii): $r > C$.}
Write
\begin{equation}
    \mathfrak Q (\vec x) = a_1x_1^{2} + \dots + a_s x_s^{2}.
\end{equation}
Observe that $M(p^{3r})$ is the number of solutions $\vec x \in (\Z/p^{3r}\Z)^{s}$ to the system
\begin{equation}
    \label{eq:system-at-3r}
    p^{2r}\frac{1}{2}k(k-1)\mathfrak Q(\vec x) \equiv p^{r}\L(\vec x) \equiv 0 \mod p^{3r}.
\end{equation}
Moreover, we have
\begin{equation}
    \label{eq:lower-bound-M-from-M'}
    M(p^{3r}) \geq p^{r(2s-1)} M'(p^{r}),
\end{equation}
where $M'(p^{r})$ is the number of solutions $\vec y \in [p^r]^{s}$ to
\begin{equation}
\label{eq:quadratic-system-mod-pr}
    \mathfrak Q(\vec y) \equiv \L(\vec y) \equiv 0 \mod p^{r}.
\end{equation}
Indeed, given such a vector $\vec y$, consider
\begin{equation}
    \vec x = \vec y + p^{r}\vec y' + p^{2r}\vec y'',
\end{equation}
where $\vec y',\vec y'' \in [p^r]^{s}$ with
\begin{equation}
    \L(\vec y') \equiv -p^{-r} \L(\vec y)\mod p^{r}.
\end{equation}
A short computation shows that $\vec x$ is a solution to \eqref{eq:system-at-3r}. There are $p^{rs}$ many choices for $\vec y''$, and since $p\ndivs \gcd(a_1,\dots,a_s)$, there are $p^{r(s-1)}$ many choices for $\vec y'$, yielding the lower bound \eqref{eq:lower-bound-M-from-M'}.

We now 
estimate $M'(p^{r})$.
Let $\vec y\in\mathcal V'(\Z_p)$ be a smooth point where, similarly to before, we assume that $p$ does not divide $y_{s-1}-y_s$.
Write
\begin{equation}
    J = 
    \begin{pmatrix}
        \partial_{s-1}\mathfrak Q & \partial_{s}\mathfrak Q\\
        \partial_{s-1}\L & \partial_{s}\L
    \end{pmatrix},
\end{equation}
and define
\begin{equation}
    \delta = \nu_p(\det(J(\vec y)))
    = \nu_p(2a_{s-1}a_s( y_{s-1}- y_s))
    = \nu_p(2a_{s-1}a_s).
\end{equation}
Choose $x_1,\dots,x_{s-2}\in \Z/p^{r}\Z$ such that
\begin{equation}
    x_i \equiv y_i \mod p^{2 \delta + 1}.
\end{equation}
Hensel's lemma \cite[Proposition 5.20]{Gre1969} furnishes $(x_{s-1}, x_s) \in (\Z/p^{r}\Z)^{2}$ such that $\vec x = (x_1,\dots,x_s)$ is a solution to \eqref{eq:quadratic-system-mod-pr}.
We deduce that
\begin{equation}
    \label{eq:M'-bound-with-delta}
    M'(p^{r}) \geq p^{(r-2 \delta -1)(s-2)}.
\end{equation}
Since $\delta$ can be bounded above uniformly in $p$, we may choose a constant $C$ satisfying $C > 2(2 \delta + 1)(s-2)$. Then the combination of \eqref{eq:M'-bound-with-delta} and \eqref{eq:lower-bound-M-from-M'} gives
\begin{equation}
    M(p^{3r}) \geq p^{3r(s-1) - C/2},
\end{equation}
and by Lemma \ref{lm:S-from-sol-count},
\begin{equation}
    \label{eq:B1-below-3r}
    \sum_{t=0}^{ 3r} S(p^{t}) \geq p^{3r - C/2}.
\end{equation}

We next consider the sum over $t>3r$.
From \eqref{eq:Si-bound-prime-pow}, for $i\in [s]$ and $p\ndivs \gcd(a,b)$ we have
\begin{align}
    p^{-t}S_i(p^{t}; a, b)) &\ll p^{-t/k}  \gcd\bigl(p^{t},\, p^{2r}a_ia k(k-1)/2,\, p^{r}a_i(ka+b)\bigr)^{1/k}
    \\
    &\ll p^{(r -t)/k}\gcd(p^{r},\, ka+b)^{1/k},
\end{align}
where the implicit constant is uniform in $C$. Therefore,
\begin{equation}
    S(p^{t}) 
    \ll p^{(r-t)s/k}\sum_{a,b\leq p^{t}}\gcd(p^{r},\, ka+b)^{s/k}.
\end{equation}
Splitting the sum according to the $p$-adic order of $ka + b$, we get
\begin{align}
    S(p^{t})
&\ll p^{(r-t)s/k}\sum_{v = 0}^{r}p^{2t-v}p^{vs/k}\\
&\ll p^{(r-t)s/k}  p^{2t-r}p^{rs/k} = p^{3r +(t-2r)(2-s/k)},
\label{eq:S-bound-case-b1-above-3r}
\end{align}
and so
\begin{equation}
\sum_{t>3r} S(p^{t}) 
\ll p^{3r + r(2-s/k)} < p^{3r -3C/2}.
\end{equation}
Recall that the implicit constant is uniform in $C$, and thus by choosing $C$ large enough, we finally conclude
\begin{equation}
\abs[\Big]{\sum_{t>3r} S(p^{t})} < p^{3r-C}.
\end{equation}
Combined with \eqref{eq:B1-below-3r}, this yields
\begin{equation}
    \chi_p \gg p^{3r}
\end{equation}
in this case.

\textbf{Case B2:} $p>C$.
We again have \eqref{eq:S-bound-up-to-r}. Suppose $r<t\leq 2r$. Then
\begin{equation}
    S(p^{t}; a, b) =
    \sum_{\vec x \in [p^{t}]^{s}}e_{p^{t}}\Bigl(\sum_{i\leq s}a_ip^{r}(ka + b)x_i\Bigr)
    = \prod_{i\leq s}\sum_{x\leq p^{t}} e_{p^{t}}(a_ip^{r}(ka+b)x).
\end{equation}
As $p\ndivs \gcd(a_1,\dots,a_s)$, we have
\begin{equation}
p^{-ts} S(p^{t}; a, b) = 
\begin{cases}
1, & \text{if } p^{t-r} \divs (ka + b),
\\
0, & \text{if } p^{t-r} \ndivs (ka + b).
\end{cases}
\end{equation}
We conclude that
\begin{equation}
\label{eq:S-bound-case-B2-below-2r}
S(p^{t}) = (p^{t}-p^{t-1}) p^{r}
\end{equation}
and hence
\begin{equation}
    \sum_{t = 0}^{2r}S(p^{t}) = p^{3r}.
\end{equation}

We now consider $t>2r$. We again have \eqref{eq:S-bound-case-b1-above-3r}, hence
\begin{equation}
    \label{eq:bound-S-sum-above-2r-case-B2}
    \sum_{t>2r} S(p^{t}) \ll p^{3r + 2 - s/k} \leq p^{3r - 3/2},
\end{equation}
and so
\begin{equation}
    \chi_p = p^{3r}(1+O(p^{-3/2})).
\end{equation}

Since $\prod_{p>C}(1 + O(p^{-3/2})) \gg1$, the combination of both cases furnishes \eqref{eq:singular-series-objective-B} in Case \ref{case:eq}.

\section{An auxiliary system}
\label{S:Aux}

We retain the notation from the previous section, and denote by $T_2$ the number of tuples 
\[
    (x_1, x_2, \ldots, x_s, x'_2, \ldots, x'_s) \in \cP^{2s-1}
\]
such that $\bx, \bx' \in \cV$, where
\begin{equation}
    \label{xx'}
    \bx = (x_1,\ldots, x_s), \qquad \bx' = (x_1, x'_2, \ldots, x'_s).
\end{equation}
In this section, we show that
\begin{equation} \label{T2a}
    T_2 \ll \frac{g^{2k-2} \#\cP^{2s-5}}{R^{2k-2}}
\end{equation}
in Case \ref{case:neq}, and
\begin{equation} \label{T2b}
    T_2 \ll \frac{g^{2k-4} \#\cP^{2s-7}}{R^{2k-4}}
\end{equation}
in Case \ref{case:eq}. In particular, if we couple this with the outcome of the previous section then we obtain
\begin{equation}
    \label{upshot}
    T_2 \ll \# \cP^{-1} T(1_\cP)^2.
\end{equation}

By homogeneity, we may assume that $g=1$. For $\vec \alpha = (\alpha, \beta, \alpha', \beta')\in\T^{4}$, define
\begin{equation}
    g_1(\vec \alpha) = f_1(\alpha + \alpha', \beta + \beta'),
    \qquad
    g_i(\vec \alpha) = f_i(\alpha, \beta)f_i(\alpha', \beta') \quad (2\leq i \leq s),
\end{equation}
and write $G(\vec \alpha) = \prod_{i\leq s}g_i(\vec \alpha)$. By orthogonality, we have
\begin{equation}
    T_2 = \int_{\T^{4}}G(\vec \alpha) \d{\vec \alpha}.
\end{equation}
Define
\begin{equation}
    \Na = \{ \vec \alpha \in\T^{4}: (\alpha, \beta), (\alpha', \beta') \in \Ma \},
    \qquad \na = \T^{4}\setminus \Na.
\end{equation}

\subsection{Minor arcs}

Recall from \eqref{eq:moment-bound-minor-arc} that
\begin{equation}
    \int_{\ma}\abs{f_i(\alpha, \beta)}^{s}\d \alpha\d \beta \ll N^{s -k -1 - \frac{\Delta}{2k}}
    \qquad (2\leq i \leq s),
\end{equation}
and by Lemma \ref{lm:moment-bound},
\begin{equation}
    \int_{\T^{2}}\abs{f_1(\alpha + \alpha', \beta + \beta')}^{s}\d \alpha \d \beta
    \ll N^{s -k -1 + \epsilon}
    \qquad (\alpha', \beta') \in \T^{2}.
\end{equation}
Now by H\"older's inequality, 
\begin{equation}
    \int_{\ma} \abs{f_1(\alpha + \alpha ', \beta + \beta') f_2(\alpha, \beta) \cdots f_{s}(\alpha, \beta)} \d \alpha \d \beta
    \ll N^{s-k-1 - \frac{\Delta}{3k}}
\end{equation}
for $\alpha', \beta' \in \T$. 
H\"older's inequality and Lemma \ref{lm:moment-bound} again now yield
\begin{align}
\int_{\T^{2}}\int_{\ma} \abs{G(\vec \alpha)} \d{\vec \alpha}
&\ll N^{s-k-1- \frac{\Delta}{3k}}
    \int_{\T^{2}} \abs{f_2(\alpha', \beta') \cdots f_s(\alpha', \beta')} \d{\alpha'} \d{\beta'}
    \\
    &\ll N^{2s-2k -3  -\frac{\Delta}{4k}}.
\end{align}
By symmetry, we thus have
\begin{equation}
    \int_{\na} \abs{G(\vec \alpha)} \d{\vec \alpha}
    \ll N^{2s-2k -3  -\frac{\Delta}{4k}}
    = o(\#\mathcal P ^{2s-5}/R^{2k -2}),
\end{equation}
whence
\begin{equation}
    \label{eq:aux-minor-arc-result}
    T_2 = \int_{\Na}G(\vec \alpha) \d{\vec \alpha} + o(\#\mathcal P^{2s-5}/R^{2k-2}).
\end{equation}

\subsection{Major arcs}
Suppose that $(\alpha, \beta)\in \Ma(q; a, b)$ and $(\alpha', \beta') \in \Ma( q'; a', b')$, where 
\begin{gather}
    1\leq a, b\leq q < N^{\Delta},
    \qquad
    1\leq a', b'\leq q' < N^{\Delta},
    \\
    \gcd(q, a, b) = \gcd(q', a', b') = 1.
\end{gather}
By the triangle inequality, we have
\begin{equation}
    \abs{q''(\alpha + \alpha') -\alpha''} < N^{3 \Delta -k},
    \qquad
    \abs{q''(\beta + \beta') -\beta''} < N^{3 \Delta -1},
\end{equation}
where
\begin{equation}
    \label{eq:double-prime-notation}
    q'' = qq',
    \qquad
    a'' = q'a + qa',
    \qquad
    b'' = q'b + qb'.
\end{equation}
Consequently $(\alpha + \alpha', \beta + \beta') \in \Ma^{(3 \Delta)}$, so we may define
\begin{equation}
    g_i ^{*} (\vec \alpha)= f_i^{*}(\alpha, \beta) f_i^{*}(\alpha', \beta')
    \quad (2\leq i \leq s),
    \qquad
    g^{*}_1 (\vec \alpha)= f_1 ^{*}(\alpha + \alpha' , \beta + \beta'),
\end{equation}
and $G^{*}(\vec \alpha) = \prod_{i\leq s}g_i^{*}(\vec \alpha)$.
Using \eqref{eq:fi-major-arc-approx} and the telescoping identity \eqref{eq:telescoping-identity}, as in \eqref{eq:major-arc-pointwise-bound}, we get
\begin{equation}
    G(\vec \alpha) - G^{*}(\vec \alpha)
    \ll (q'')^{1+ \epsilon}N^{2s-2}
    \ll N^{2s-2 + 3 \Delta}
\end{equation}
for $\vec \alpha \in \Na$. Now
\begin{equation}
    \int_{\Na}G(\vec \alpha) \d{\vec \alpha} -   \int_{\Na}G^{*}(\vec \alpha) \d{\vec \alpha}
    \ll N^{2s-2 + 3 \Delta}N^{6 \Delta -2k -2}
    = o(\#\mathcal P^{2s-5}/R^{2k -2}).
\end{equation}
Together with \eqref{eq:aux-minor-arc-result}, this gives
\begin{equation}
    T_2 = \int_{\Na}G^{*}(\vec \alpha)\d{\vec \alpha} + o(\#\mathcal P^{2s-5}/R^{2k-2}).
\end{equation}

Define
\begin{equation}
    \label{eq:S-star-long-def}
    S^{*}(q, q'; a, b, a', b') = S_1(q ''; a '', b '')\prod_{i=2}^{s}S_i(q; a, b)S_i(q'; a', b')
\end{equation}
and
\begin{equation}
    S^{*}(q, q') = (qq')^{-s}
    \sum_{\substack{a , b \leq q \\ \gcd(q ,a , b )=1}} \
    \sum_{\substack{a',  b'\leq q'\\ \gcd(q',a',  b')=1}}
    S^{*}(q, q'; a, b, a', b').
\end{equation}
Write $\vec \tau = (\tau, \kappa, \tau' , \kappa')$, where
\begin{equation}
    \tau  = \alpha  -a /q ,
    \qquad
    \kappa  = \beta  -b /q ,
    \qquad
    \tau' = \alpha' -a'/q',
    \qquad
    \kappa' = \beta' -b'/q'.
\end{equation}
Periodicity gives
\begin{equation}
    f_1^{*}(\alpha+ \alpha', \beta + \beta') = (q '')^{-1}S_1(q ''; a '', b '')v_1(\tau + \tau', \kappa + \kappa'),
\end{equation}
wherein we recall \eqref{eq:vi-def}. Therefore
\begin{equation}
    G^{*}(\vec \alpha) = (qq')^{-s}S^{*}(q, q'; a, b, a', b')V^{*}(\vec \tau),
\end{equation}
where
\begin{equation}
    V^{*}(\vec \tau)
    = v_1(\tau + \tau', \kappa + \kappa')\prod_{i=2}^{s}v_i(\tau, \kappa)v_i(\tau', \kappa').
\end{equation}

We now bound $S^{*}(q, q'; a, b, a', b')$. By \eqref{eq:Si-bound-full-gcd} and \eqref{eq:gcd-bound},
\begin{align}
    S_1(q''; a'', b'') &\ll \gcd(q'', m^{k}a_1a '', m^{2}b_1b '')^{1/k}(qq')^{1-1/k + \epsilon}
    \\
    &\ll m\gcd(q'', a'', b'')^{1/k} (qq')^{1-1/k + \epsilon}.
\end{align}
Moreover,
\begin{align}
    \gcd(q'', a'', b'')
    &= \gcd(qq', q'a+qa', q'b + qb')
    \\
    &\leq \gcd(q, q'a, q'b) \gcd(q', qa', qb')
    = \gcd(q, q')^{2},
\end{align}
and therefore
\begin{equation}
    \label{eq:S1-bound}
    S_1(q''; a'', b'') \ll m \gcd(q, q')^{2/k} (qq')^{1-1/k + \epsilon}.
\end{equation}
Now by \eqref{eq:Si-bound},
\begin{equation}
    \label{eq:long-S-bound}
    (qq')^{-s}S^{*}(q, q'; a, b, a', b')
    \ll m^{2s-1}\gcd(q , q')^{2/k}(qq')^{-s/k + \epsilon}.
\end{equation}

It follows from \cite[Theorem 7.3]{Vau1997} that
\begin{equation}
    V^{*}(\vec \tau)
    \ll (R/m)^{2s-1}
    (1 + R\abs{\kappa} + R^{k}\abs{\tau})^{(1-s)/k}
    (1 + R\abs{\kappa'} + R^{k}\abs{\tau'})^{(1-s)/k}.
\end{equation}
Hence
\begin{multline}
    G^{*}(\vec \alpha)
    \ll
    N^{2s-1}
    \gcd(q, q')^{2/k}
    (qq')^{-s/k + \epsilon}
    \\
    \cdot(1+N\abs{\kappa})^{-\eta}
    (1+N\abs{\kappa'})^{-\eta}
    (1+N^{k}\abs{\tau})^{(1-s)/k + \eta}
    (1+N^{k}\abs{\tau'})^{(1-s)/k + \eta},
\end{multline}
whenever $1<\eta < (s-1)/k -1$. Therefore,
\begin{multline}
    S^{*}(q, q')V^{*}(\vec \tau)
    \ll
    N^{2s-1}
    \gcd(q, q')^{2/k}
    (qq')^{2-s/k + \epsilon}
    \\
    \cdot(1+N\abs{\kappa})^{-\eta}
    (1+N\abs{\kappa'})^{-\eta}
    (1+N^{k}\abs{\tau})^{(1-s)/k + \eta}
    (1+N^{k}\abs{\tau'})^{(1-s)/k + \eta}.
\end{multline}

As the major arcs are disjoint,
\begin{equation}
    \int_{\Na}G^{*}(\vec \alpha)\d{\vec \alpha}
    = \sum_{q,q' < N^{\Delta}}
    S^{*}(q, q')
    \int_{\substack{
            \abs{q\tau}, \abs{q'\tau'} < N^{\Delta-k}\\
            \abs{q\kappa}, \abs{q'\kappa'} < N^{\Delta-1}
    }}
    V^{*}(\vec\tau)\d{\vec \tau}.
\end{equation}
Choosing $\eta = 1 + \epsilon$, we compute that
\begin{align}
&\sum_{q,q'\geq 1} \abs{S^{*}(q, q')}
\int_{\R^{3}}\int_{N^{\Delta-1}/q}^{\infty}
    \abs{V^{*}(\vec\tau)}\d{\kappa}\d\tau\d{\tau'}\d{\kappa'}
    \\
&\ll N^{2s-1}N^{-2k -1}N^{-1-\epsilon \Delta}
    \sum_{q,q'\geq 1}\gcd(q,q')^{2/k}(qq')^{2-s/k + \epsilon} q^\eps
    \\
&\ll N^{2s-2k -3-\epsilon \Delta}
    = o(\#\mathcal P^{2s-5}/R^{2k-2}),
\end{align}
where we bound the last sum by factoring out $d=\gcd(q, q')$:
\begin{equation}
    \sum_{q,q'\geq 1}\gcd(q,q')^{2/k}(qq')^{2-s/k + \epsilon} q^\eps
    \ll \sum_{d\geq 1}d^{4 -2(s-1)/k + 2 \epsilon}\Bigl( \sum_{q\geq 1} q^{2 - s/k +\epsilon} \Bigr)^{2} \ll 1,
\end{equation}
since $s \geq 3k+1$.
Choosing $\eta = (s-3)/2 + \epsilon$, we similarly compute that
\begin{align}
&\sum_{q,q'\geq 1} \abs{S^{*}(q, q')}
    \int_{\R^{3}}\int_{N^{\Delta-k}/q}^{\infty}
    \abs{V^{*}(\vec\tau)}\d{\tau}\d{\tau'}\d\kappa\d{\kappa'}
    \\
&\ll N^{2s-1}N^{-k-2}N^{-k-\epsilon \Delta}
    = o(\#\mathcal P^{2s-5}/R^{2k -2}).
\end{align}
Thus by symmetry,
\begin{equation}
    \int_{\Na}G^{*}(\vec \alpha)\d{\vec \alpha}
    = \sum_{q,q'< N^{\Delta}} S^{*}(q, q')
    \int_{\R^{4}}V^{*}(\vec\tau)\d{\vec \tau}
    + o(\#\mathcal P ^{2s-5}/R^{2k-2}).
\end{equation}
Moreover, we have
\begin{align}
&\sum_{q'\geq 1}\sum_{q\geq N^{\Delta}}S^{*}(q, q')
\int_{\R^{4}}V^{*}(\vec\tau)\d{\vec \tau}
\\
&\ll N^{2s-2k -3}\sum_{q' \geq 1}\sum_{q\geq N^{\Delta}} \gcd(q, q')^{2/k}(qq')^{2-s/k + \epsilon}
\\
&\ll N^{2s-2k -3}(N^{\Delta})^{3-s/k+\epsilon}
= o(\#\mathcal P ^{2s-5}/R^{2k-2}),
\end{align}
and by symmetry,
\begin{equation}
    \int_{\Na}G^{*}(\vec \alpha)\d{\vec \alpha}
    = \sum_{q,q'\geq 1} S^{*}(q, q')
    \int_{\R^{4}}V^{*}(\vec\tau)\d{\vec \tau}
    + o(\#\mathcal P ^{2s-5}/R^{2k-2}).
\end{equation}

Changing variables yields
\begin{align}
&\int_{\R^{4}}V^{*}(\vec\tau)\d{\vec \tau}
\\
&= m^{1-2s}\int_{\R^{4}}\int_{[L, R]^{2s-1}}
e(\tau\F(\vec x) + \tau'\F(\vec x') + \kappa\L(\vec x) + \kappa'\L(\vec x')) \d{\vec x}\d{x'_2}\cdots \d{x'_s}\d{\vec\tau}
\\
&= R^{2s-2k -3}m^{1-2s} \mathfrak J^{*},
\end{align}
where
\begin{equation}
    \mathfrak J^{*} = \int_{\R^{4}}\int_{B^{*}}
    e(\tau\F(\vec x) + \tau'\F(\vec x') + \kappa\L(\vec x) + \kappa'\L(\vec x')) \d{\vec x}\d{x'_2}\cdots \d{x'_s}\d{\vec\tau},
\end{equation}
and $B^{*} = [L/R, 1]^{2s-1}$.
Therefore
\begin{equation}
    \int_{\Na}G^{*}(\vec \alpha)\d{\vec \alpha}
    = R^{2s-2k-3}m^{1-2s}\mathfrak S^{*}\mathfrak J^{*}
    + o(\#\mathcal P^{2s-5}/R^{2k-2}),
\end{equation}
where
\begin{equation}
    \mathfrak S^{*} = \sum_{q, q'\geq 1}S^{*}(q, q').
\end{equation}
Thence,
\begin{equation}
    \label{eq:aux-cm-result}
    T_2
    = R^{2s-2k-3}m^{1-2s}\mathfrak S^{*}\mathfrak J^{*}
    + o(\#\mathcal P^{2s-5}/R^{2k-2}).
\end{equation}

\subsection{Singular integral}
Here we show that
\begin{equation}
    \label{eq:aux-singular-integral-objective-A}
    \mathfrak J^{*}
    \ll \biggl(1-\frac{L}{R}\biggr)^{2s-5}
\end{equation}
in Case \ref{case:neq}, and
\begin{equation}
    \label{eq:aux-singular-integral-objective-B}
    \mathfrak J^{*}
    \ll \biggl(1-\frac{L}{R}\biggr)^{2s-7}
\end{equation}
in Case \ref{case:eq}.

In Case \ref{case:neq}, we follow \cite[\S11]{Sch1982} and interpret $\mathfrak J^{*}$ as a real density. Recall the identity
\begin{equation}
    \lambda(y) \defeq \max\{ 1-\abs y, 0 \}
    = \int_{\R}e(\alpha y)\sinc^{2}(\pi \alpha) \d \alpha,
\end{equation}
from which we obtain
\begin{equation}
    \lambda_{T}(y)
    \defeq T\lambda(Ty)
    =\int_{\R}e(\tau y)\sinc^{2}(\pi \tau/T)\d \tau
\end{equation}
for $T\geq 1$. Writing $\vec x^{*} = (\vec x, x'_2,\dots, x'_s)$ and $\vec x' = (x_1, x'_2,\dots,x'_s)$, we thus have
\begin{align}
\mu^{*}_T &\defeq
\int_{B^{*}}\lambda_T(\F(\vec x))\lambda_T(\F(\vec x'))\lambda_T(\L(\vec x))\lambda_T(\L(\vec x'))\d{\vec x^{*}} \\
&= \int_{\R^{4}}\int_{B^{*}}e(\tau\F(\vec x) + \tau'\F(\vec x') + \kappa\L(\vec x) + \kappa'\L(\vec x'))
\\
&\qquad \cdot \sinc^{2}(\pi\tau/T)\sinc^{2}(\pi\tau'/T)\sinc^{2}(\pi\kappa/T)\sinc^{2}(\pi\kappa'/T) \d{\vec x^{*}}\d{\vec \tau}.
\end{align}
Now
\begin{equation}
    \mathfrak J^{*} -\mu^{*}_T = \int_{\R^{4}}f(\vec \tau)g(\vec \tau, T)\d{\vec \tau},
\end{equation}
where
\begin{equation}
    f(\tau) = \int_{B^{*}}e(\tau\F(\vec x) + \tau'\F(\vec x') + \kappa\L(\vec x) + \kappa'\L(\vec x'))\d{\vec x^{*}}
\end{equation}
and
\begin{equation}
    g(\tau, T) = 1 - \sinc^{2}(\pi\tau/T)\sinc^{2}(\pi\tau'/T)\sinc^{2}(\pi\kappa/T)\sinc^{2}(\pi\kappa'/T).
\end{equation}

Note that
\begin{equation}
    f(\vec \tau) = v^{*}_1(\tau + \tau', \kappa + \kappa')\prod_{i=2}^{s}v^{*}_i(\tau, \kappa)v^{*}_i(\tau', \kappa'),
\end{equation}
where
\begin{equation}
    \label{eq:def-vistar}
v^{*}_i(\tau, \kappa) = \int_{L/R}^{1}e(a_i\tau x^{k} + b_i \kappa x)\d x
\qquad (1\leq i \leq s).
\end{equation}
It now follows from \cite[Theorem 7.3]{Vau1997} that
\begin{equation}
    f(\vec \tau) \ll (1 + \abs{\tau} + \abs{\kappa})^{(1-s)/k}(1 + \abs{\tau'} + \abs{\kappa'})^{(1-s)/k}.
\end{equation}
Further, it follows from the Taylor expansion that 
\begin{equation}
    \sinc^{2}(y) = 1 + O(\min\{1, y^{2}\}),
\end{equation}
and consequently
\begin{align}
g(\vec \tau, T) &\ll
\min\{ 1, (\tau/T)^{2} \} + \min\{ 1, (\tau'/T)^{2} \}
\\
&\qquad + \min\{ 1, (\kappa/T)^{2} \} + \min\{ 1, (\kappa'/T)^{2} \}.
\end{align}
Thus, by symmetry, we have $\mathfrak J^{*} - \mu^{*}_T \ll I_1I_2(T)$, where
\begin{equation}
    I_1 = \int_{\R^{2}}(1+\abs \tau + \abs \kappa)^{(1-s)/k}\d\tau\d\kappa \ll 1
\end{equation}
and
\begin{equation}
I_2(T) = \int_{\R^{2}}(1+\abs\tau + \abs \kappa)^{(1-s)/k}\min\Bigl\{ 1, \frac{\tau^{2} + \kappa^{2}}{T^{2}} \Bigr\}\d\tau\d\kappa.
\end{equation}
By symmetry again, we have $I_2(T)\ll I_3(T) + I_4(T)$, where
\begin{align}
I_3(T)
&= T^{-2}\int_{0\leq \tau\leq\kappa\leq T}(1+\kappa)^{2 + (1-s)/k}\d\tau\d\kappa
\\ &\leq T^{-2}\int_0^{T}(1+\kappa)^{3+(1-s)/k}\d\kappa 
    \\
&\ll T^{-2} + T^{2 + (1-s)/k} \log T \to 0
    \qquad (T\to\infty)
\end{align}
and
\begin{align}
I_4(T) &= \int_{T}^{\infty}\int_0^{\kappa}\kappa^{(1-s)/k}\d\tau\d\kappa
    = \int_T^{\infty} \kappa^{1+(1-s)/k}\d\kappa
    \\
&\ll T^{2 + (1-s)/k} \to 0
    \qquad (T\to\infty).
\end{align}
In particular, 
\begin{equation}
    \mathfrak J^{*} = \lim_{T\to\infty}\mu^{*}_T.
\end{equation}

It remains to show that
\begin{equation}
    \vol(\mathcal R^{*}) \ll (1-L/R)^{2s-5}T^{-4},
\end{equation}
where
\begin{equation}
    \mathcal R^{*} = \{ \vec x^{*}\in B^{*} : \abs{\F(\vec x)}, \abs{\F(\vec x')}, \abs{\L(\vec x)}, \abs{\L(\vec x')} \leq T^{-1} \}.
\end{equation}
By symmetry, given $x_1,\dots,x_{s-2}\in [ L/R, 1]^{s-2}$, it now suffices to prove
\begin{equation}
    \vol(\{ (x_{s-1}, x_s) \in [L/R, 1]^{2} : \abs{\F(\vec x)}, \abs{\L(\vec x)} \leq T^{-1}\}) \ll T^{-2}.
\end{equation}

We may assume that $1-L/R < c_0$, where $c_0 = c_0(\mathcal V)$ is a small, positive constant. This ensures that $\det (J) \asymp 1$ for $(x_{s-1}, x_s) \in [L/R ,1]^{2}$, where
\begin{equation}
    J = 
    \begin{pmatrix}
        ka_{s-1}x_{s-1}^{k-1} & ka_{s}x_{s}^{k-1} \\
        b_{s-1} & b_{s}
    \end{pmatrix}
\end{equation}
is the Jacobian matrix of $(x_{s-1}, x_s)\mapsto (\F(\vec x), \L(\vec x))$. Changing variables completes the proof of \eqref{eq:aux-singular-integral-objective-A} in Case \ref{case:neq}.

It remains to show that \eqref{eq:aux-singular-integral-objective-B} holds in Case \ref{case:eq}. 
With the notation \eqref{eq:def-vistar}, denote
\[
\mathfrak J^{\dag} = \int_{\R^{2}}\prod_{i=2}^{s}\abs{v_i^{*}(\tau, \kappa)} \d \tau \d \kappa,
\]
and observe that
$
    \mathfrak J^{*} \ll (1-L/R)(\mathfrak J^{\dag})^{2}.
$
By symmetry and H\"older's inequality, for \eqref{eq:aux-singular-integral-objective-B} it is enough to show that
\begin{equation}
    \label{eq:aux-singular-integral-objective-B-intermediate}
    \int_{\R^{2}}\abs{v^{*}_1(\tau, \kappa)}^{s-1}\d \tau \d \kappa \ll (1-L/R)^{s-4}.
\end{equation}

A change of variables gives
\begin{equation}
v^{*}_1(\tau, \kappa) = \int_0^{X} e(\beta_0 + \beta_1 y + \dots+\beta_k y^{k})\d y,
\end{equation}
where $X = 1-L/R$, $\beta_0 \in \R$, and
\begin{gather}
\beta_1 = a_1 \tau k\left( \frac{L}{R} \right)^{k-1} + b_ 1\kappa,
\qquad
\beta_i = a_1 \tau \binom{k}{i}\left( \frac{L}{R}\right)^{k-i}
\quad (2\leq i \leq k).
\end{gather}
Now \cite[Theorem 7.3]{Vau1997} furnishes
\begin{equation}
v^{*}_1(\tau, \kappa) \ll X(1 + \abs{\beta_1}X + \dots + \abs{\beta_k}X^{k} )^{-1/k}.
\end{equation}
We may suppose that $L/R\gg 1$, so
\begin{equation}
v^{*}_1(\tau, \kappa) \ll X(1 + \abs{\beta_1}X + \abs{\tau}X^{2})^{-1/k}.
\end{equation}
Changing variables from $\kappa$ to $\eta = a_1 \tau k\left( \frac{L}{R} \right)^{k-1} + b_1 \kappa$, we have
\begin{equation}
\int_{\R^{2}}\abs{v_1^{*}(\tau, \kappa)}^{s-1} \d \tau \d \kappa \ll X^{s-1}\int_{\R^{2}}(1 + \abs{\eta}X + \abs{\tau}X^{2})^{(1-s)/k}\d \tau \d \eta
\ll X^{s-4},
\end{equation}
delivering \eqref{eq:aux-singular-integral-objective-B-intermediate} and hence \eqref{eq:aux-singular-integral-objective-B} in Case \ref{case:eq}.

\subsection{Singular series}
Here we show that
\begin{equation}
    \label{eq:aux-singular-series-objective-A}
    \mathfrak S^{*} \ll m^{4}
\end{equation}
in Case \ref{case:neq}, and
\begin{equation}
    \label{eq:aux-singular-series-objective-B}
    \mathfrak S^{*} \ll m^{6}
\end{equation}
in Case \ref{case:eq}.
These are the final ingredients needed for \eqref{T2a} in Case \ref{case:neq} and \eqref{T2b} in Case \ref{case:eq}. In Case \ref{case:neq}, for instance, the inequalities \eqref{eq:aux-singular-integral-objective-A} and \eqref{eq:aux-singular-series-objective-A} give
\begin{equation}
    m^{1-2s}\mathfrak S^{*}\mathfrak J^{*}
    \ll m^{5-2s} (1-L/R)^{2s-5}
    \ll R^{5-2s}\#\mathcal P^{2s-5},
\end{equation}
and substituting this into \eqref{eq:aux-cm-result} yields \eqref{T2a} when $g=1$.

For $Q\in \N$, denote by $M^{*}(Q)$ the number of solutions $(x_1,\dots,x_s,x'_2,\dots,x'_s)\in[Q]^{2s-1}$ to
\begin{equation}
    \label{eq:aux-system-mod-Q}
    \F(m\vec x + \vec u)
    \equiv \F(m\vec x' + \vec u)
    \equiv \L(m\vec x + \vec u)
    \equiv \L(m\vec x' + \vec u)
    \equiv 0 \mod Q.
\end{equation}
\begin{lem}
    \label{lm:sol-count-from-Sstar}
    For $Q\in\N$, 
    \begin{equation}
        \sum_{d,d'\divs Q}S^{*}(d,d') = Q^{5-2s}M^{*}(Q).
    \end{equation}
\end{lem}

\begin{proof}
By orthogonality, 
\begin{align}
M^{*}(Q)
&=
\begin{multlined}[t]
Q^{-4} \sum_{v,w,v',w' \leq Q} \ \sum_{x_1,\dots,x_s,x'_2,\dots,x_s'\leq Q}
e_Q(v\F(m\vec x + \vec u) + w\L(m\vec x + \vec u))
\\
\cdot e_Q(v'\F(m\vec x' + \vec u) + w'\L(m\vec x' + \vec u))
\end{multlined}
\\
&=
Q^{-5} \sum_{d, d'\divs Q}(Q/d)^{s}(Q/d')^{s}
\sum_{\substack{a,b \leq d, \, \gcd(d,a, b) = 1
\\
a',b' \leq d',\, \gcd(d',a', b') = 1
}} \
\sum_{\substack{
x_1 \leq dd' \\
x_2, \dots, x_s \leq d \\
x_2', \dots, x'_s \leq d'}}
\\
&\qquad\qquad e_d(a\F(m\vec x + \vec u) + b\L(m\vec x + \vec u)) e_{d'}(a'\F(m\vec x' + \vec u) + b'\L(m\vec x' + \vec u))
\\
&= Q^{2s-5} \sum_{d, d' \mid Q}
S^{*}(d, d').
\end{align}
\end{proof}

It follows from the Chinese remainder theorem that $M^{*}(\cdot)$ is a multiplicative arithmetic function. The Möbius inversion formula now implies that $S^{*}(\cdot)$ is multiplicative, where
\begin{equation}
    S^{*}(Q) = \sum_{\lcm(q, q') = Q}S^{*}(q, q').
\end{equation}
Observe that
\begin{equation}
    \mathfrak S^{*} = \sum_{Q\geq 1}S^{*}(Q),
\end{equation}
and by \eqref{eq:long-S-bound},
\begin{align}
\sum_{Q\geq 1}\abs{S^{*}(Q)}
&\leq \sum_{q,q'\geq 1} \abs{S^{*}(q, q')}
    \ll_M \sum_{q,q'\geq 1}\gcd(q, q')^{2/k}(qq')^{2-s/k + \epsilon} 
    \\
&\leq \sum_{d\geq 1}d^{4-2(s-1)/k + 2 \epsilon}\Bigl(\sum_{q\geq 1}q^{2-s/k + \epsilon}  \Bigr)^{2} \ll 1,
\end{align}
since $s\geq 3k+1$. Hence, we have the absolutely convergent product representation
\begin{equation}
    \mathfrak S^{*} = \prod_{p}\chi^{*}_p,
\end{equation}
where
\begin{equation}
    \chi^{*}_p = \sum_{t\geq 0} S^{*}({p^{t}}).
\end{equation}
The following is an immediate consequence of Lemma \ref{lm:sol-count-from-Sstar}.
\begin{cor}
    \label{cor:chistar-as-limit}
    If $p$ is prime, then
    \begin{equation}
        \chi^{*}_{p} = \lim_{t\to \infty}p^{-t(2s-5)}M^{*}(p^{t}).
    \end{equation}
\end{cor}

Let $p$ be prime, and suppose that $p^{r}\edivs m$, where $r\geq 0$. In light of Corollary \ref{cor:chistar-as-limit}, for the purpose of estimating $\chi^{*}_p$ we may assume that $m=p^{r}$ and $u=1$.
Let $C$ be a large constant depending only on $\mathcal V$.

\subsubsection{Case \ref{case:neq}}
We decompose $\chi^{*}_p$ into three sums:
\begin{equation}
    \label{eq:chi-star-decomposition}
    \chi^{*}_p = \Sigma_1 + 2\Sigma_2 + \Sigma_3,
\end{equation}
where
\begin{equation}
    \label{eq:decomposition-pieces}
    \Sigma_1 = \sum_{0\leq t,t'\leq r}S^{*}(p^{t}, p^{t'}),
    \qquad
    \Sigma_2 = \sum_{0\leq t\leq r < t'}S^{*}(p^{t}, p^{t'}),
    \qquad
    \Sigma_3 = \sum_{t,t'>r} S^{*}(p^{t}, p^{t'}).
\end{equation}
Let $t, t', a,b,a', b' \in \N$, and suppose that $t\leq r$. It follows from \eqref{eq:Si-as-exp-sum} and \eqref{eq:S-star-long-def} that in this case
\begin{equation}
    \label{eq:Sstar-splits-at-t-below-r}
    S^{*}(p^{t}, p^{t'}; a, b, a', b')
    = S(p^{t}; a, b)S(p^{t'}; a' , b').
\end{equation}
Therefore, from \eqref{eq:S-bound-up-to-r},
\begin{equation}
    \label{eq:Sigma_1-result}
    \Sigma_1 
    = \Bigl(\sum_{t\leq r} S(p^{t})\Bigr)^{2} = p^{4r}.
\end{equation}

Similarly,
\begin{equation}
    \label{eq:Sigma_2-splits}
    \Sigma_2
    =\sum_{t\leq r}S(p^{t})\sum_{t'>r}S(p^{t'})
    = p^{2r}\sum_{t'>r}S(p^{t'}),
\end{equation}
and so by \eqref{eq:S-gcd-bound-at-power-p},
\begin{equation}
    \Sigma_2 \ll p^{2r} \sum_{t > r}
    \sum_{\substack{ a,b\leq p^{t}\\ p\ndivs \gcd(a,b)}}
    \prod_{i\leq s}\left(\frac{\gcd(p^{t-r}, ka_ia + b_ib)}{p^{t-r}}\right)^{1/k}.
\end{equation}
If $p>C$, then \eqref{eq:S-bound-at-large-power-p} implies
\begin{equation}
    \Sigma_2 \ll p^{2r}\cdot p^{2r + 2 - s/k} \leq p^{4r -3/2}.
\end{equation}
If $p\leq C$, we similarly have $p^{C} \divs (ka_ia + b_ib)$ for at most one value of $i$, and we recover the same estimate. Indeed, we have $p^{C}\ll 1$ in this context, so
\begin{align}
&\sum_{\substack{ a,b\leq p^{t}\\ p\ndivs \gcd(a,b)}}
\prod_{i\leq s}\left(\frac{\gcd(p^{t-r}, ka_ia + b_ib)}{p^{t-r}}\right)^{1/k}
\\
&\ll p^{2t}\cdot p^{-s(t-r)/k} + p^{-(s-1)(t-r)/k} \cdot
\#\bigl\{ (a,b) \in [p^{t}]^2: p^{C} \divs \prod_{i\leq s}(ka_ia+b_ib)\bigr\}
\\
\label{eq:ab-bound-small-prime-all-terms}
&\ll p^{2t} \cdot p^{-s(t-r)/k} + p^{2t-1}\cdot p^{-(s-1)(t-r)/k}.
\end{align}
Once again,
\begin{equation}
    \label{eq:Sigma_2-result}
    \Sigma_2 \ll p^{2r}\cdot p^{2r + 2 - s/k} \leq p^{4r -3/2}.
\end{equation}

To estimate $\Sigma_3$, we use the bound $S_1(p^{t+t'}; a '', b '') \leq p^{t + t'}$ and \eqref{eq:Si-gcd-bound-at-power-p} to get
\begin{align}
&(p^{t+t'})^{-s}S^{*}(p^{t}, p^{t'}; a, b, a', b')
\\
&\ll
\prod_{i= 2}^{s}\left(\frac{\gcd(p^{t-r}, ka_ia + b_ib)}{p^{t-r}}\right)^{1/k}
\cdot
\prod_{i= 2}^{s}\left(\frac{\gcd(p^{t'-r}, ka_ia' + b_ib')}{p^{t'-r}}\right)^{1/k}.
\end{align}
If $p>C$, we again know that $p$ divides $ka_ia+b_ib$ for at most one value of $i$. Similarly to \eqref{eq:S-bound-at-large-power-p}, for $t>r$ we have
\begin{align}
&\sum_{\substack{ a,b\leq p^{t}\\ p\ndivs \gcd(a,b)}}
\prod_{i = 2}^{s}\left(\frac{\gcd(p^{t-r}, ka_ia + b_ib)}{p^{t-r}}\right)^{1/k}
\\
&\ll p^{2t}\cdot p^{-(s-1)(t-r)/k} + p^{-(s-2)(t-r)/k}\cdot
\#\bigl\{ (a,b) \in [p^{t}]^2: p\divs\prod_{i=2}^{s}(ka_ia+b_ib)\bigr\}
    \\ 
    \label{eq:ab-bound-big-prime-minus-first}
    &\ll p^{2t}\cdot p^{-(s-1)(t-r)/k} + p^{2t-1}\cdot p^{-(s-2)(t-r)/k}.
\end{align}
If $p\leq C$, then $p^{C}$ divides $ka_ia + b_ib$ for at most one value of $i$, so as in \eqref{eq:ab-bound-small-prime-all-terms}, we recover the same estimate as above:
\begin{align}
&\sum_{\substack{ a,b \leq p^{t} \\ p\ndivs \gcd(a,b)}}
\prod_{i = 2}^{s}\left(\frac{\gcd(p^{t-r}, ka_ia + b_ib)}{p^{t-r}}\right)^{1/k}
\\
&\ll p^{2t} \cdot p^{-(s-1)(t-r)/k} + p^{-(s-2)(t-r)/k}\cdot \#\bigl\{(a,b) \in [p^{t}]^2: p^{C}\divs\prod_{i=2}^{s}(ka_ia+b_ib)\bigr\}
\\
\label{eq:ab-bound-small-prime-minus-first}    
&\ll p^{2t}\cdot p^{-(s-1)(t-r)/k} + p^{2t-1}\cdot p^{-(s-2)(t-r)/k}.
\end{align}
By symmetry, the combination of these bounds yields
\begin{equation}
\label{eq:Sigma_3-result}
\begin{aligned}
\Sigma_3
&\ll
\Bigl(
\sum_{t>r} (p^{2t-(s-1)(t-r)/k} + p^{2t-1-(s-2)(t-r)/k})
\Bigr)^{2}
    \\
    &\ll p^{4r + 4 - 2(s-1)/k} \leq p^{4r -2}.
\end{aligned}
\end{equation}

Combining \eqref{eq:Sigma_1-result}, \eqref{eq:Sigma_2-result}, and \eqref{eq:Sigma_3-result} gives
\begin{equation}
    \chi^{*}_p = p^{4r}(1 + O(p^{-3/2})).
\end{equation}
Since $\prod_{p}(1 + O(p^{-3/2}) \ll 1$, we conclude that \eqref{eq:aux-singular-series-objective-A} holds in Case \ref{case:neq}.

\subsubsection{Case \ref{case:eq}} 
As $\ba$ is primitive and $a_1 + \cdots + a_s = 0$, we have
\begin{equation}
    \label{eq:relaxed-gcd-conditon}
    p\ndivs \gcd(a_2,\dots,a_s).
\end{equation}
As before, we start by decomposing $\chi_p^{*}$ according to the ranges of $t$ and $t'$, writing
\begin{equation}
    \chi^{*}_p =\Sigma_1 + 2\Sigma_2 + \Sigma_4 + 2\Sigma_5 + \Sigma_6.
\end{equation}
Here $\Sigma_1$, $\Sigma_2$ are as in \eqref{eq:decomposition-pieces}, and
\begin{equation}
    \Sigma_4 = \sum_{r<t, t'\leq 2r}S^{*}(p^{t}, p^{t'}),
    \qquad
    \Sigma_5 = \sum_{r<t\leq 2r< t'}S^{*}(p^{t}, p^{t'}),
    \qquad
    \Sigma_6 = \sum_{t, t' > 2r}S^{*}(p^{t}, p^{t'}).
\end{equation}
As in \eqref{eq:Sigma_1-result}, we have
\begin{equation}
    \label{eq:Sigma_1-result-again}
    \Sigma_1 = p^{4r}.
\end{equation}

In order to bound $\Sigma_2$, we observe that \eqref{eq:S-bound-up-to-r}, \eqref{eq:S-bound-case-B2-below-2r}, and \eqref{eq:bound-S-sum-above-2r-case-B2} hold in this context. Consequently,
\begin{equation}
\sum_{t>r}S(p^{t}) = p^{3r}(1 + O(p^{- 3/2})) - p^{2r}.
\end{equation}
It now follows from \eqref{eq:Sigma_2-splits} that
\begin{equation}
\label{eq:Sigma_2-result-again}
\Sigma_2 = p^{5r}(1 + O(p^{- 3/2})) - p^{4r}.
\end{equation}

Next, we consider $\Sigma_4$. Suppose that $r<t,t'\leq 2r$. Then
\begin{align}
&(p^{t+t'})^{-s}\abs{S^{*}(p^{t}, p^{t'}; a, b, a', b')}
\\
&\leq (p^{t + t'})^{1-s}
\abs[\Big]{
    \prod_{i=2}^{s}\sum_{x\leq p^{t}}e_{p^{t}}(a_ip^{r}(ka+b)x)
}
\cdot
\abs[\Big]{
    \prod_{i=2}^{s}\sum_{x\leq p^{t'}}e_{p^{t'}}(a_ip^{r}(ka'+b')x)
}.
\end{align}
Since $p\ndivs \gcd(a_2,\dots,a_s)$,
\begin{equation}
(p^{t+t'})^{-s} |S^{*}(p^{t}, p^{t'}; a, b, a', b')|
    \leq
    \begin{cases}
        1, & \text{if $p^{t-r}\divs (ka + b)$ and $p^{t'-r}\divs (ka' + b')$},
        \\
        0, &\text{otherwise}.
    \end{cases}
\end{equation}
We can then compute
\begin{equation}
|S^{*}(p^{t}, p^{t'})| \leq (p^{t} - p^{t-1})(p^{t'} - p^{t'-1})p^{2r},
\end{equation}
and so
\begin{equation}
\label{eq:Sigma_4-result}
|\Sigma_4| \leq (p^{3r}-p^{2r})^{2} = p^{6r} - 2p^{5r} + p^{4r}.
\end{equation}

We now estimate $\Sigma_5$.
Take $r<t\leq 2r$. As above, since $p\ndivs \gcd(a_2,\dots,a_s)$, 
\begin{equation}
p^{-t(s-1)}\prod_{i=2}^{s}\abs{S_i(p^{t}; a, b)}
\leq
\begin{cases}
1, & \text{if} \ p^{t-r}\divs (ka + b),
\\
0, & \text{if} \ p^{t-r} \ndivs (ka + b).
\end{cases}
\end{equation}
Now let $t'> 2r$ and suppose that $p\ndivs \gcd(a',b')$. As before, it follows from \eqref{eq:Si-bound-prime-pow} that
\begin{align}
p^{-t'} S_i(p^{t'}; a', b') 
&\ll p^{-t'/k} \gcd\bigl(p^{t'}, p^{2r}a_ia' k(k-1)/2, p^{r}a_i(ka ' + b ')\bigr)^{1/k}
\\
&\ll p^{-(t'-r)/k}\gcd(p^{t'-r}, p^{r}a', ka' + b')^{1/k}
\\
\label{eq:aux-Si-bound-case-B-above-2r}
&=  p^{-(t'-r)/k}\gcd(p^{r}, ka' + b')^{1/k},
\end{align}
for $i=2,\dots,s$. Moreover, under the assumption that $p^{t-r}$ divides $ka + b$,
with the notation from \eqref{eq:double-prime-notation}, we recover the same bound for $S_1(p^{t+ t'}; a '', b '')$:
\begin{align*} 
p^{-t-t'}S_1(p^{t+t'}; a '', b '') 
&\ll p^{-(t+t')/k}\gcd\bigl(p^{t+t'}, p^{2r}a '', p^{r}(ka '' + b '')\bigr)^{1/k}
\\
&= p^{-(t'-r)/k}\gcd\bigl(p^{t'-r}, p^{r}a' + p^{t'-t+r}a, ka' + b'  + p^{t'-t}(ka + b)\bigr)^{1/k}
\\
&= p^{-(t'-r)/k}\gcd(p^{r}, ka' + b')^{1/k}.
\end{align*}
These results combine to produce
\begin{equation}
(p^{t+t'})^{-s}S^{*}(p^{t}, p^{t'}; a, b, a', b')
\ll
    \begin{cases}
        p^{-(t'-r)s/k}\gcd(p^{r}, ka' + b')^{s/k}, & \text{if } p^{t-r}\divs (ka + b),
        \\
        0, &\text{if } p^{t-r} \ndivs (ka+b).
    \end{cases}
\end{equation}
Now
\begin{equation}
    S^{*}(p^{t}, p^{t'})
    \ll p^{t+r}\cdot p^{-(t'-r)s/k}\sum_{\substack{a', b' \leq p^{t'}\\ p\ndivs \gcd(a', b')}}\gcd(p^{r}, ka' + b')^{s/k}.
\end{equation}
The sum is
\begin{equation}
    \sum_{\substack{a', b' \leq p^{t'}\\ p\ndivs \gcd(a', b')}}\gcd(p^{r}, ka' + b')^{s/k}
    \ll \sum_{v = 0}^{r}p^{2t'-v}p^{vs/k}
    \ll p^{2t' -r(1-s/k)},
\end{equation}
so
\begin{equation}
    S^{*}(p^{t}, p^{t'}) \ll p^{t + 2t' -(t'-2r)s/k}.
\end{equation}
Therefore,
\begin{equation}
    \label{eq:Sigma_5-result}
    \Sigma_5 \ll p^{6r + 2 - s/k} \ll p^{6r - 3/2}.
\end{equation}

Finally, we bound $\Sigma_6$. Let $t,t'>2r$ and assume that $p\ndivs \gcd(a,b)$ and $p\ndivs \gcd(a', b')$. In this case, we use the bound \eqref{eq:aux-Si-bound-case-B-above-2r} to get
\begin{align}
&(p^{t+t'})^{-s}S^{*}(p^{t}, p^{t'}; a, b, a', b') \\
&\ll 
    p^{-(t-r)(s-1)/k}\gcd(p^{r}, ka + b)^{(s-1)/k}
    \cdot p^{-(t'-r)(s-1)/k}\gcd(p^{r}, ka' + b')^{(s-1)/k}.
\end{align}
This gives
\begin{align}
    S^{*}(p^{t}, p^{t'})
    &\ll
    \begin{multlined}[t]
    p^{-(t-r)(s-1)/k}\sum_{\substack{a, b \leq p^{t}\\ p\ndivs \gcd(a, b)}}\gcd(p^{r}, ka + b)^{(s-1)/k}
    \\
        \cdot
        p^{-(t'-r)(s-1)/k}\sum_{\substack{a', b' \leq p^{t'}\\ p\ndivs \gcd(a', b')}}\gcd(p^{r}, ka' + b')^{(s-1)/k}
    \end{multlined}
    \\
    &\ll 
    p^{2t -r -(t-2r)(s-1)/k}
    \cdot
    p^{2t' -r -(t'-2r)(s-1)/k}.
\end{align}
Therefore,
\begin{equation}
    \label{eq:Sigma_6-result}
    \Sigma_6 \ll (p^{3r + 2 - (s-1)/k})^{2} \leq p^{6r - 2}.
\end{equation}

The combination of \eqref{eq:Sigma_1-result-again}, \eqref{eq:Sigma_2-result-again}, \eqref{eq:Sigma_4-result}, \eqref{eq:Sigma_5-result} and \eqref{eq:Sigma_6-result} gives
\begin{equation}
\label{eq:chistar-bound-case-B}
\chi^{*}_p \leq p^{6r}(1+O(p^{-3/2})).
\end{equation}
Since $\prod_p (1+O(p^{- 3/2})) \ll 1$, we conclude that \eqref{eq:aux-singular-series-objective-B} holds in Case \ref{case:eq}.

\section{Configuration control}
\label{S:Config}

The generalised von Neumann lemmas in this section will ultimately tell us that $f_\unf$ and $f_\sml$ contribute negligibly to the relevant counting operator. The latter relies critically on the results of the previous two sections.

\begin{lem} 
[Fourier control]
\label{FourierControl}
If each $g_i$ is $1$-bounded and
\[
T(\bg) \ge \del N^{s-k-1},
\]
for some $\del \in (0,1)$, then
\[
\| \bg \|_{\FL} \gg \del^{2s} N^s.
\]
\end{lem}

\begin{proof} 
    Put $t = s_0 -\frac{1}{2}$. By H\"older's inequality,
\begin{align*}
    \del N^{s-k-1}
    &\le T(\bg)
    =\int_{\bT^2} |\tilde \bg(\alp,\bet)| \d \alp \d \bet \\
    &\le \| \bg \|_{\FL}^{(s-t)/s} \int_{\bT^2} \prod_i | \tilde{g_i}(a_i \alp, b_i \bet) |^{t/s} \d \alp \d \bet \\
    &\le \| \bg \|_{\FL}^{(s-t)/s} \prod_{i \le s} \Bigl( \int_{\bT^2} |\tilde{g_i}(a_i \alp, b_i \bet)|^{t} \d \alp \d \bet \Bigr)^{1/s}.
\end{align*}

First, suppose $k \ge 3$. Let $\zeta > 0$ be small in terms of $k$, then for each $i\in [s]$,
\begin{align}
&\int_{\T^{2}} \abs{\tilde g_i(a_i \alpha, b_i \beta)}^{t} \d \alpha \d \beta 
= 
\int_{\T^{2}} \abs{\tilde g_i(\alpha, \beta)}^{t} \d \alpha \d \beta    
\\
&\leq N^{(1-\zeta)/2} \int_{\T^{2}} \abs{\tilde g_i(\alpha, \beta)}^{s_0 -1} \d \alpha \d \beta    
+ \int_{\abs{\tilde g_i(\alpha, \beta)} \geq N^{1-\zeta}} \abs{\tilde g_i(\alpha, \beta)}^{t} \d \alpha \d \beta.
\end{align}
Observing that $s_0 -1$ is an even integer, by orthogonality and Lemma \ref{lm:moment-bound}, we get
\[
\int_{\bT^2} |\tilde{g_i}(\alp, \bet)|^{s_0 -1} \d \alp \d \bet 
\ll N^{s_0 -k - 2 + \epsilon}.
\]
Moreover, \cite[Theorem 1.4]{HH2018} implies
\begin{equation}
    \int_{\abs{\tilde g_i(\alpha, \beta)} \geq N^{1-\zeta}} \abs{\tilde g_i(\alpha, \beta)}^{t} \d \alpha \d \beta
    \ll N^{t -k - 1},
\end{equation}
and thus
\begin{equation}
    \label{eq:restricion-estimate}
    \int_{\T^{2}} \abs{\tilde g_i(a_i \alpha, b_i \beta)}^{t} \d \alpha \d \beta    
    \ll N^{t-k-1}.
\end{equation}
Therefore
\[
    \del N^{s-k- 1} \ll \| \bg \|_{\FL}^{(s-t)/s} N^{t-k - 1},
\]
and so
\[
\| \bg \|_{\FL} \gg \del^{s/(s-t)} N^s \ge \del^{2s} N^s.
\]

Now suppose instead that $k=2$. By \cite[Equation (3.115)]{Bou1993},
\[
    \int_{\bT^2} |\tilde{g_i}(a_i \alp, b_i \bet)|^{t} \d \alp \d \bet 
    = \int_{\bT^2} |\tilde{g_i}(\alp, \bet)|^{t} \d \alp \d \bet 
    \ll N^{t -3}
    \qquad (1 \le i \le s).
\]
The argument now follows as above.
\end{proof}

\begin{cor} [FL control]
\label{FL}
If
\[
T(\bg) \ge \del N^{s-k-1},
\]
for some $\del \in (0,1)$, then
\[
\| g_i \|_{\FL} \gg \del^{2s} N \qquad (1 \le i \le s).
\]
\end{cor}

\begin{lem}
[\texorpdfstring{$\ell^2$}{l2} control]
\label{l2}
If $g_1, \ldots, g_s: \cP \to \bC$ are $1$-bounded, and $N \ge N_0(M)$, then
\[
T(\bg) \ll T(1_\cP) \frac{\| g_i \|_2}{\|1_\cP\|_2}
\qquad (1 \le i \le s).
\]
\end{lem}

\begin{proof}
By symmetry, it suffices to prove this for $i=1$. Writing \eqref{xx'},
Cauchy--Schwarz gives
\[
|T(\bg)| \le \| g_1 \|_2 \Bigl(
\sum_{x_1 \in \cP}
\Bigl(
\sum_{\fF(\bx)=\fL(\bx)=0} 1_\cP(x_2) \cdots 1_\cP(x_s) \Bigr)^2 \Bigr)^{1/2}.
\]
Inserting \eqref{upshot} now yields
\[
T(\bg) \le \| g_1 \|_2 T(1_\cP) \# \cP^{-1/2}.
\]
\end{proof}

\section{Equidistribution}
\label{S:Equi}

Let $\mathcal P$ be as in \eqref{eq:P-def}. Recall that
\[
f_\str(n) = F(\btet^{(1)}n, \btet^{(k)} n^k),
\]
for $n \in \cP$. In this section, we will exploit the irrationality property of the frequencies $\btet^{(j)}$. Using Fourier series, this will enable us to express our count in terms of the corresponding density of solutions on a certain torus.
The implied constants are allowed to depend on $s$ and $k$.

\begin{lem} [Weyl-type bound]
\label{Weyl}
If
\[
\Big|
\sum_{\bn \in \cP^s} e \Bigl( \sum_{i \le s} (\alp_i n_i^k + \bet_i n_i) \Bigr) \Big| \ge \del N^s,
\]
then
:
\begin{enumerate}[(i)]
\item each $\alp_i$ is $(O(\del^{-O(1)}),N^k)$-rational,
\item each $\bet_i$ is $(O(\del^{-O(1)}),N)$-rational.
\end{enumerate}
\end{lem}

\begin{proof}
By symmetry, it suffices to prove the result for $i=1$. We may assume that $0 < \del < 1/2$, since otherwise we can argue with $\del/3$ in place of $\del$. By the triangle inequality, the assumption implies that
\[
\Big| \sum_{n \in \cP} e(\alp n^k + \bet n) \Big| \ge \del N,
\]
where $\alp = \alp_1$ and $\bet = \bet_1$. Writing
\[
\cI = \{ y \in \bZ: L < u + m y \le R \}
\]
yields
\[
\Big|
\sum_{y \in \cI} e(\alp (u+my)^k + \bet my)
\Big| \ge \del N.
\]
With 
\[
t = \lfloor (L-u)/m \rfloor, \qquad
z = y - t,
\qquad
\cI' = \cI - t,
\qquad v = u - mt,
\]
we now have
\[
\Big|
\sum_{z \in \cI'}
e(\alp(v + mz)^k + \bet m z) \Big| \ge \del N.
\]
Now \cite[Lemma 4.4]{GT2012} furnishes the $(O(\del^{-O(1)}),|\cP|^k)$-rationality of
$m^k \alp$, as well as the $(O(\del^{-O(1)}),|\cP|)$-rationality of $km u^{k-1} \alp + m \bet$. The result now follows from the triangle inequality and the bounds
\[
u \le m,
\qquad
|\cP| \le N/m.
\]
\end{proof}

\begin{cor}
\label{WeylCor}
Let $A \ge 1$. Suppose that $\balp \in \bT$ is $(A,N^k)$-irrational or that $\bet \in \bT$ is $(A,N)$-irrational. Then 
\[
\sum_{n \in \cP} e(\alp n^k + \bet n)
\ll M A^{-\Ome(1)} |\cP|.
\]
\end{cor}

\begin{proof}
Let $c > 0$ be the reciprocal of the largest implied constant appearing in Lemma \ref{Weyl}, and suppose
\[
\Big|
\sum_{n \in \cP} e(\alp n^k + \bet n) \Big|
\ge (cA)^{-c} N.
\]
Applying Lemma \ref{Weyl} with
\[
\del = (cA)^{-c}, \qquad
(\alp_i, \bet_i) = 
\begin{cases}
(\alp, \bet), &\text{if } i= 1, \\
(0,0), &\text{otherwise}
\end{cases}
\]
delivers a contradiction
. Thus, 
\[
\sum_{n \in \cP} e(\alp n^k + \bet n) \ll A^{-c} N \le M A^{-c} |\cP|.
\]
\end{proof}

\begin{lem} [Equidistribution I]
\label{equi1}
Let $A \ge M$. Let $d_1, d_k \le M$ be positive integers, and put $d = d_1 + d_k$. Let $F: \bT^d \to \bC$ be a $1$-bounded trigonometric polynomial of degree at most $M$. If 
each $\btet^{(i)} \in \bT^{d_i}$ is $(A,N^i)$-irrational, then
\[
\displaystyle
\bE_{n \in \cP} F(\btet^{(1)} n, \btet^{(k)} n^k) = \int_{\bT^d} F(\bgam) \d \bgam + O(M^{O(1)} A^{-\Ome(1)}).
\]
\end{lem}

\begin{proof}
Expanding $F$ as a Fourier series gives
\begin{equation}
\sum_{n \in \cP} F(\btet^{(1)}n, \btet^{(k)}n^k)
= \hat F(\bzero) |\cP| + \sum_{0 < \| \bm \|_1 \le M} \hat F(\bm) \sum_{n \in \cP}
e(\bm^{(1)} \cdot \btet^{(1)} n + \bm^{(k)} \cdot \btet^{(k)} n^k),
\end{equation}
where $\bm = (\bm^{(1)}, \bm^{(k)})$ with $\bm^{(i)} \in \bZ^{d_i}$. Note that $\bm^{(i)} \cdot \btet^{(i)}$ is $(A/M, N^i)$-irrational, for some $i$. Thus, by Corollary \ref{WeylCor},
\[
\sum_{n \in \cP} e(\bm^{(1)} \cdot \btet^{(1)} n + \bm^{(k)} \cdot \btet^{(k)} n^k) \ll M A^{-\Ome(1)} |\cP|.
\]
Substituting this into the Fourier expansion above delivers the claimed estimate.
\end{proof}

In what follows, we write $e_{\alp,\bet}(n) = e(\alp n^k + \bet n)$.

\begin{lem}
\label{OnMajorArc}
Suppose
\[
|T(1_\cP e_{\alp_1, \bet_1}, \ldots, 1_\cP e_{\alp_s, \bet_s})| \ge \del N^{s-k-1}.
\]
Then
, for all $i,j \in [s]$,
\begin{itemize}
\item $a_j \alp_i - a_i \alp_j$ is $(O(\del^{-O(1)}), N^k)$-rational
\item $b_j \bet_i - b_i \bet_j$ is $(O(\del^{-O(1)}),N)$-rational.
\end{itemize}
\end{lem}

\begin{proof}
By Lemma \ref{FourierControl}, there exist $\lam, \mu \in \bT$ such that 
\[
\sum_{\bn \in \cP^s}
e\Bigl(
\sum_{i \in [s]}
((a_i \lam + \alp_i) n_i^k + (b_i \mu + \bet_i) n_i) \Bigr) \gg \del^{O(1)} N^s.
\]
Now Lemma \ref{Weyl} tells us that each $a_i \lam + a_i$ is $(O(\del^{-O(1)}),N^k)$-rational. Thus, by the triangle inequality,
\[
a_j \alp_i - a_i \alp_j = a_j( a_i \gam + \alp_i) - a_i (a_j \alp + \alp_j)
\]
is also $(O(\del^{-O(1)}),N^k)$-rational. The second assertion follows similarly.
\end{proof}

\begin{cor}
\label{CountingEquidistribution}
Let $A \ge 1$.
Let $\alp_1, \ldots, \alp_s, \bet_1, \ldots, \bet_s$ be such that at least one of numbers
\[
a_j \alp_i - a_i \alp_j
\qquad (1 \le i, j \le s)
\]
is $(A,N^k)$-irrational or at least one of the numbers
\[
b_j \bet_i - b_i \bet_j
\qquad (1 \le i, j \le s)
\]
is $(A,N)$-irrational. Then 
\[
T(1_\cP e_{\alp_1, \bet_1}, \ldots, 1_\cP e_{\alp_s, \bet_s}) \ll A^{-\Ome(1)} N^{s-k-1}.
\]
\end{cor}

\begin{proof}
Let $c > 0$ be the reciprocal of the largest implied constant appearing in Lemma \ref{OnMajorArc}, and suppose
\[
|T(1_\cP e_{\alp_1, \bet_1}, \ldots, 1_\cP e_{\alp_s, \bet_s})| \ge (cA)^{-c} N^{s-k-1}.
\]
Then Lemma \ref{OnMajorArc} yields a contradiction.
\end{proof}

\begin{lem}
\label{parallel}
Let $a_1, \ldots, a_s \in \bZ$ with $\gcd(a_1, \ldots, a_s) = 1$, and let $\bm^{(1)}, \ldots, \bm^{(s)} \in \bZ^d$. Then
\[
a_j \bm^{(i)} = a_i \bm^{(j)} \qquad (1 \le i,j \le s)
\]
holds if and only if there exists $\bm \in \bZ^d$ such that
\[
\bm^{(i)} = a_i \bm
\qquad (1 \le i \le s).
\]
\end{lem}

\begin{proof}
Suppose $a_j \bm^{(i)} = a_i \bm^{(j)}$ for all $i$ and $j$. Then $\bm^{(1)}, \ldots, \bm^{(s)}$ are all proportional to some primitive vector $\bm^{(0)} \in \bZ^d$. For each $i$, let $c_i \in \bZ$ be such that $\bm^{(i)} = c_i \bm^{(0)}$. Then $a_j c_i = a_i c_j$ for all $i$ and $j$, so $\bc = \lam \ba$ for some $\lam \in \bZ$. Choosing $\bm = \lam \bm^{(0)}$ yields $\bm^{(i)} = a_i \bm$ for all $i$.

The converse is clear.
\end{proof}

For $d,s \in \bN$ and $a_1, \ldots, a_s \in \bZ$, define
\[
\cK(\ba, d) = \{ 
(\bgam^{(1)}, \ldots, \bgam^{(s)}) \in (\bT^d)^s:
a_1 \bgam^{(1)} + \cdots + a_s \bgam^{(s)} = \bzero \}.
\]
This is a compact abelian group that is naturally equipped with a Haar probability measure.

\begin{lem}
\label{ConstrainedOrth}
Let $a_1, \ldots, a_s \in \bZ$ be integers such that $\gcd(a_1, \ldots, a_s) = 1$, and let $\bm^{(1)}, \ldots, \bm^{(s)} \in \bZ^d$. Then
\[
\int_{\cK(\ba, d)}
e \Bigl(
\sum_{i \in [s]}
\bm^{(i)} \cdot \bgam^{(i)} \Bigr) \d \bgam =
\begin{cases}
1, &\text{if } a_i \bm^{(j)} = a_j \bm^{(i)} \text{ for all } i, j \\
0, &\text{otherwise}.
\end{cases}
\]
\end{lem}

\begin{proof}
First suppose $a_i \bm^{(j)} = a_j \bm^{(i)}$ for all $i,j$. Then, by Lemma \ref{parallel}, there exists $\bm \in \bZ^d$ such that $\bm^{(i)} = a_i \bm$ for all $i$. Consequently,
\[
\int_{\cK(\ba,d)}
e \Bigl( \sum_{i \in [s]} \bm^{(i)} \cdot \bgam^{(i)} \Bigr) \d \bgam
= \int_{\cK(\ba,d)}
e \Bigl( \bm \cdot \sum_{i \in [s]} a_i \bgam^{(i)} \Bigr) \d \bgam = 1.
\]

Now suppose instead that there exist $i,j$ with $a_i \bm^{(j)} \ne a_j \bm^{(i)}$. We claim that there exists $\tilde \bgam \in \cK(\ba, d)$ such that $\sum_{i \in [s]} \bm^{(i)} \cdot \tilde \bgam^{(i)} \ne \bzero \in \bT^d$. Assuming that such a vector $\tilde \bgam$ exists, translation-invariance of $\cK(\ba, d)$ and a change of variables give
\begin{align}
    &e \Bigl( \sum_{i \in [s]} \bm^{(i)} \cdot \tilde \bgam^{(i)} \Bigr) \int_{\cK(\ba,d)}
    e \Bigl( \sum_{i \in [s]} \bm^{(i)} \cdot \bgam^{(i)} \Bigr) \d \bgam \\
    &= \int_{\cK(\ba,d)}
    e \Bigl( \sum_{i \in [s]} \bm^{(i)} \cdot (\bgam^{(i)} + \tilde \bgam^{(i)}) \Bigr) \d \bgam
    = \int_{\cK(\ba,d)}
    e \Bigl( \sum_{i \in [s]} \bm^{(i)} \cdot \bgam^{(i)} \Bigr) \d \bgam,
\end{align}
whereupon $\int_{\cK(\ba,d)}
e \bigl( \sum_{i \in [s]} \bm^{(i)} \cdot \bgam^{(i)} \bigr) \d \bgam = 0$.

It remains to establish the claim. Fixing $i$ and $j$ such that $a_i \bm^{(j)} \ne a_j \bm^{(i)}$, let $k \in [d]$ be such that $a_i m_{j,k} \ne a_j m_{i,k}$. Choosing $\tilde \gam_{\ell, m} = 0$ if $\ell \notin \{ i, j \}$ or $m \ne k$, it now suffices to find $\tilde \gam_{i,k}$ and $\tilde \gam_{j,k}$ such that
\begin{align*}
a_i \tilde \gam_{i,k} + a_j \tilde \gam_{j,k} &= 0 \in \bT, \\
m_{i,k} \tilde \gam_{i,k} + m_{j,k} \tilde \gam_{j,k} &\ne 0 \in \bT,
\end{align*}
since the first condition ensures that $\tilde \bgam \in \cK(\ba, d)$ and the second condition ensures that $\sum_{\ell \in [s]} \bm^{(\ell)} \cdot \tilde \bgam^{(\ell)} \ne \bzero \in \bT^d$. As
\[
\begin{pmatrix}
a_i & a_j \\
m_{i,k} & m_{j,k}
\end{pmatrix}
\]
is an invertible real matrix, there exist $\alp, \bet \in \bR$ such that
\[
a_i \alp + a_j \bet = 0, \qquad
m_{i,k} \alp + m_{j,k} \bet = 1/2.
\]
Choosing $\tilde \gam_{i,k} = \alp \mmod 1$ and $\tilde \gam_{j,k} = \bet \mmod 1$ delivers the claim.
\end{proof}

\begin{lem}
[Equidistribution II]
\label{equi2}
Let $a_1, \ldots, a_s, b_1, \ldots, b_s \in \bZ$ be such that
$$
\gcd(a_1, \ldots, a_s) = \gcd(b_1, \ldots, b_s) = 1.
$$
Let $A \ge CM$, where $C = 2 \| (\ba, \bb) \|_\infty$.
Let $d = d_1 + d_k \le M$, and let $F: \bT^d \to \bC$ be a $1$-bounded trigonometric polynomial of degree at most $M$. If 
each $\btet^{(i)} \in \bT^{d_i}$ is $(A,N^i)$-irrational, then
\begin{align}
&\sum_{\bn \in \cV(\cP)} F(\btet^{(1)} n_1, \btet^{(k)} n_1^k) \cdots F(\btet^{(1)} n_s, \btet^{(k)} n_s^k) \\
&= T(1_\cP) \int_{\cK(\ba, d_k)} \int_{\cK(\bb, d_1)} \prod_{i \le s} F(\bgam^{(i)}, \tilde \bgam^{(i)}) \d \bgam \d  \tilde \bgam
+O(M^{O(1)} A^{-\Ome(1)} N^{s-k-1}).
\end{align}
\end{lem}

\begin{proof}
    Expanding the left hand side into Fourier series gives
    \begin{align}
        &\sum_{\bn \in \cV(\cP)} F(\btet^{(1)} n_1, \btet^{(k)} n_1^k) \cdots F(\btet^{(1)} n_s, \btet^{(k)} n_s^k) \\
        &=\sum_{\substack{{\bm^{(1)}, \ldots, \bm^{(s)} \in \bZ^{d_1}} \\
        {\tilde \bm^{(1)}, \ldots, \tilde \bm^{(s)} \in \bZ^{d_k}}}} \hat F(\bm^{(1)}, \tilde \bm^{(1)}) \cdots
        \hat F(\bm^{(s)}, \tilde \bm^{(s)}) \\
        &\qquad \qquad
        \cdot T(1_{\cP} e_{\bm^{(1)} \cdot \btet^{(1)},
            \tilde \bm^{(1)} \cdot \btet^{(k)}}, \ldots, 1_{\cP} e_{\bm^{(s)} \cdot \btet^{(1)},
        \tilde \bm^{(s)} \cdot \btet^{(k)}}).
    \end{align}
As $F$ has degree at most $M$, we note that if $\hat F(\bm^{(i)}, \tilde \bm^{(i)}) \ne 0$ then
\[
\| \bm^{(i)} \| + \| \tilde \bm^{(i)} \| \le M.
\]
Observe that if $a_j \tilde \bm^{(i)} \ne a_i \tilde \bm^{(j)}$ then $(a_j \tilde \bm^{(i)} - a_i \tilde \bm^{(j)}) \cdot \btet^{(k)}$ is $(A/(CM), N^k)$-irrational, and if $b_j \bm^{(i)} \ne b_i \bm^{(j)}$ then $(b_j \bm^{(i)} - b_i \bm^{(j)}) \cdot \btet^{(1)}$ is $(A/(CM), N)$-irrational. Thus, by Corollary \ref{CountingEquidistribution}, we have
\begin{align}
    &\sum_{\bn \in \cV(\cP)} F(\btet^{(1)} n_1, \btet^{(k)} n_1^k) \cdots F(\btet^{(1)} n_s, \btet^{(k)} n_s^k) \\
    &= \sum_{(\bm, \tilde \bm) \in \cW} \hat F(\bm^{(1)}, \tilde \bm^{(1)}) \cdots
    \hat F(\bm^{(s)}, \tilde \bm^{(s)}) T(1_\cP) 
    + O(M^{O(1)} A^{-\Ome(1)} N^{s-k-1}),
\end{align}
where $\cW$ is the set of $(\bm, \tilde \bm) \in \bZ^{(d_1+d_k)s}$ such that
\[
a_j \tilde \bm^{(i)} = a_i \tilde \bm^{(j)},
\quad
b_j \bm^{(i)} = b_i \bm^{(j)}
\qquad (1 \le i, j \le s).
\]
Finally, by Lemma \ref{ConstrainedOrth},
\begin{align}
    & \int_{\cK(\ba, d_k)} \int_{\cK(\bb, d_1)} \prod_{i \le s} F(\bgam^{(i)}, \tilde \bgam^{(i)}) \d \bgam \d \tilde \bgam
    \\
    &= 
    \begin{multlined}[t]
        \sum_{(\bm, \tilde \bm) \in \bZ^{(d_1+d_k)s}} 
        \hat F(\bm^{(1)}, \tilde \bm^{(1)}) \cdots
        \hat F(\bm^{(s)}, \tilde \bm^{(s)})
        \\
        \qquad \qquad \int_{\cK(\ba, d_k)} \int_{\cK(\bb, d_1)}
        e\Bigl(
        \sum_{i \in [s]} \bm^{(i)} \cdot \bgam^{(i)} \Bigr)
        e\Bigl(
        \sum_{i \in [s]} \tilde \bm^{(i)} \cdot \tilde \bgam^{(i)} \Bigr) \d \bgam \d \tilde \bgam
    \end{multlined} 
 \\ & = \sum_{(\bm, \tilde \bm) \in \cW} \hat F(\bm^{(1)}, \tilde \bm^{(1)}) \cdots
 \hat F(\bm^{(s)}, \tilde \bm^{(s)}).
\end{align}
\end{proof}

\section{Proof of the main theorem}
\label{S:Proof}

In this section, we prove Theorem \ref{MainThm}. 

\subsection{The compact setting}

\begin{defn}
[Kernel measure]
Given an $r\times m$ integer matrix $M$ and a compact abelian group $G$, we have a continuous homomorphism $G^m \to G^r$ given by $\bgam \mapsto M \bgam$.  The kernel $\ker_G M$ is a compact abelian group and therefore has an associated Haar probability measure $\mu_{\ker_G M}$. We note that if $f : G \to \bC$ is Borel measurable then  $\bgam \mapsto \prod_{i=1}^s f(\gamma_i)$ is Borel measurable on $G^s$, as is the restriction of this function to $\ker_G M$.
\end{defn}

We import a deep result from \cite{CSV2016}, which relies on the hypergraph regularity lemma. We will apply it with $G = \bT^d$ and
\begin{equation}
\label{system matrix}
M = 
\begin{pmatrix}
a_1 & \dots & a_s & 0 & \dots & 0 \\
0 & \dots & 0 & a_1' & \dots & a_s'
\end{pmatrix},
\end{equation}
and is this setting there is a simpler argument using Fourier analysis \cite{Tao2014}.

\begin{prop}
[Compact Szemer\'edi]
\label{compact Roth}
Let $M$ be an $r\times m$ integer matrix such that the greatest common divisor of the $r\times r$ minors is $1$. Suppose that the columns of $M$ sum to zero. For any $\delta > 0$, there exists $\eta = \eta(M, \delta) > 0$ such that for any compact abelian group $G$, if $\cA \subseteq G$ is  measurable with $\mu_{G}(\cA) \geq \delta$, then 
$$
\mu_{\ker_G M}(\cA^m \cap \ker_G M) \geq \eta.
$$
\end{prop}

\begin{cor}
\label{CompactRothCor}
Let $a_i, a_i'$ be integers with $a_1 + \dots + a_s = a_1' + \dots + a_s' = 0$ and $\gcd(a_1, \dots, a_s) = \gcd(a_1', \dots, a_s') = 1$.
For any $\delta > 0$, there exists $\eta = \eta(\ba, \ba', \delta) > 0$ such that if $f: \bT^{d} \times \bT^{d'} \to [0,1]$ is measurable and $\int_{\bT^{d+d'}}  f \geq \delta$, then 
$$
\int_{\cK(\ba,d)} \int_{\cK(\ba',d')}\prod_{i=1}^s f(\bgam^{(i)}, \tilde \bgam^{(i)}) \d \bgam \d \tilde \bgam \geq \eta.
$$
\end{cor}

\begin{proof}
We note that the columns of the matrix \eqref{system matrix} sum to zero. Furthermore, the greatest common divisor of the $2 \times 2$ minors is
\[
\gcd(\{
a_ia_j' : 1 \leq i, j \leq s \}) = 1.
\]
Let
\[
\cA = \{ (\bgam, \tilde \bgam) \in \bT^{d+d'} : f(\bgam, \tilde \bgam) \geq \delta/2 \}.
\]
Then the measure of $\cA$ is at least $\delta/2$, and applying Proposition \ref{compact Roth} with $G = \bT^{d+d'}$ gives
$$
\mu_{\ker_{G}M }(\cA^s\cap \ker_G M) \gg_{\ba, \ba', \delta} 1.
$$
The corollary follows on noting that
\[
\int_{\ker_G M}\prod_{i=1}^s f(\bgam^{(i)}, \tilde \bgam^{(i)}) \d \mu_{\ker_G M} \geq (\delta/2)^s \mu_{\ker_G M} (\cA^s\cap \ker_G M),
\]
and that $\ker_G M$ is isomorphic to $K(\ba,d) \times K(\ba', d')$.
\end{proof}

\subsection{Putting it all together}

We now proceed to prove Theorem \ref{MainThm}. Since there are infinitely many diagonal solutions, we may assume that $N \ge N_0(\cV, \del)$. We may also assume that the quantity $\eta = \eta(\del)$ from Corollary \ref{CompactRothCor} is non-decreasing, and that $\eta(\del) \le \del$. Let $C \ge 1$ be a large, positive constant. We define the parameter $\eps$ and the function $\cF$ by
\begin{equation}
\label{parameters}
\eps = \frac{\eta(\del/C)}{C},
\qquad
\cF(x) = \left( \frac{x}{\eta(\del/C)} \right)^C
\quad (x \ge 1).
\end{equation}
We apply Lemma \ref{ARL} with these choices of $\eps$ and $\cF$. This application introduces various data.

By multilinearity,
\[
T(1_\cA) = T(f + f_\unf) = T(f) + O(|T(\bg)|),
\]
for some functions $g_i \in \{ f, f_\unf \}$ at least one of which is $f_\unf$. Thus, by Corollary \ref{FL},
\begin{equation}
\label{TotalToSum}
T(1_\cA) = T(f) + O(N^{s-k-1} \cF(M)^{-\Ome(1)}).
\end{equation}

For any arithmetic function $g$ supported on $[N]$, we write
\[
T_\cP(g) = T(g 1_\cP).
\]
As $f \ge 0$,
\[
T(f) \ge T_\cP(f) = T_\cP(f_\str + f_\sml).
\]
Again by multilinearity, there exist functions $h_i \in \{ f_\str, f_\sml \}$ at least one of which is $f_\sml$, such that
\[
T_\cP(f_\str + f_\sml) = T_\cP(f_\str) + O(|T(h_1 1_\cP, \ldots, h_s 1_\cP)|).
\]
Now, by Lemma \ref{l2},
\begin{equation}
\label{SumToStructured}
T(f) \ge T_\cP(f_\str) - O(\eps T(1_\cP)).
\end{equation}

Lemma \ref{equi2} gives
\begin{equation}
\label{StructuredToCompact}
T_\cP(f_\str) = T(1_\cP) \int_{\cK(\ba,d_k)} \int_{\cK(\bb, d_1)} \prod_{i \le s} F(\bgam^{(i)}, \tilde \bgam^{(i)}) \d \bgam \d \tilde \bgam
 +
O \left(
\frac{M^{O(1)}}{\cF(M)^{\Ome(1)}} N^{s-k-1} \right).
\end{equation}
By Lemma \ref{equi1},
\begin{align*}
\int_{\bT^d} F(\bgam) \d \bgam
&= \bE_{n \in \cP} f_\str(n) +
O\left(
\frac{M^{O(1)}}{\cF(M)^{\Ome(1)}} \right)
\end{align*}
Since $\eta(\del/C) \le \del/C$, we now see from \eqref{parameters} that
\[
\int_{\bT^d} F(\bgam) \d \bgam \ge \del/C.
\]
Now, by Corollary \ref{CompactRothCor},
\[
\int_{\cK(\ba,d_k)} \int_{\cK(\bb, d_1)} \prod_{i \le s} F(\bgam^{(i)}, \tilde \bgam^{(i)}) \d \bgam \d \tilde \bgam \ge \eta(\del/C).
\]
Substituting this into \eqref{StructuredToCompact} yields
\[
T_\cP(f_\str) \ge \eta(\del/C) T(1_\cP) - O \left(
\frac{M^{O(1)}}{\cF(M)^{\Ome(1)}} N^{s-k-1} \right).
\]
By \eqref{T1a} or \eqref{T1b},
\[
T(1_\cP) \gg M^{-O(1)} N^{s-k-1}.
\]
Our choice \eqref{parameters} now ensures that
\[
T_\cP(f_\str) \gg \eta(\del/C) T(1_\cP).
\]
Combining this with \eqref{parameters}, \eqref{TotalToSum} and \eqref{SumToStructured} delivers
\[
T(1_\cA) \gg \eta(\del/C) M^{-O(1)} N^{s-k-1}
\]
as sought.

\bibliographystyle{plain}
\bibliography{bibliography}

@book{Bak1986,
  title     = "Diophantine Inequalities",
  author    = "Baker, R. C.",
  publisher = "Clarendon Press",
  series    = "London Mathematical Society Monographs",
  year      =  1986,
  address   = "Oxford, England",
  language  = "en"
}

@book{Vau1997,
  title     = "The Hardy--Littlewood method",
  author    = "Vaughan, R. C.",
  publisher = "Cambridge University Press",
  series    = "Cambridge tracts in mathematics",
  edition   = "second",
  year      =  1997,
  address   = "Cambridge, England"
}

@article{Sch1985,
  title     = {The density of integer points on homogeneous varieties},
  volume    = {154},
  ISSN      = {0001-5962},
  url       = {http://dx.doi.org/10.1007/BF02392473},
  DOI       = {10.1007/bf02392473},
  number    = {3--4},
  journal   = {Acta Math. },
  publisher = {International Press of Boston},
  author    = {Schmidt,  W. M.},
  year      = {1985},
  pages     = {243--296}
}

@article{Sze1975,
  title     = {On sets of integers containing no $k$ elements in arithmetic progression},
  volume    = {XXVII},
  ISSN      = {},
  url       = {},
  DOI       = {},
  number    = {},
  journal   = {Acta Arith.},
  publisher = {},
  author    = {E. Szemer\'edi},
  year      = {1975},
  pages     = {199--245}
}

@article{Har2016,
  title     = {Minor arcs, mean values, and restriction theory for exponential sums over smooth numbers},
  volume    = {152},
  ISSN      = {},
  url       = {},
  DOI       = {},
  number    = {},
  journal   = {Compos. Math.},
  publisher = {},
  author    = {A. J. Harper},
  year      = {2016},
  pages     = {1121--1158}
}

@book{Gre1969,
  title     = {Lectures on Forms in Many Variables},
  author    = {Greenberg, M. J.},
  isbn      = {9780805335538},
  lccn      = {lc68059230},
  series    = {Math Lecture Notes Series},
  url       = {https://books.google.es/books?id=org-AAAAIAAJ},
  year      = {1969},
  publisher = {W. A. Benjamin}
}

@article{BR2015,
  title     = {Rational points on linear slices of diagonal hypersurfaces},
  author    = {Br{\"u}dern, J. and Robert, O.},
  journal   = {Nagoya Math. J.},
  volume    = {218},
  pages     = {51--100},
  year      = {2015},
  publisher = {Cambridge University Press}
}

@article{Woo2019,
  author    = {Wooley, T. D.},
  title     = {Nested efficient congruencing and relatives of {V}inogradov's mean value theorem},
  journal   = {Proc. London Math. Soc.},
  volume    = {118},
  number    = {4},
  pages     = {942--1016},
  doi       = {https://doi.org/10.1112/plms.12204},
  url       = {https://londmathsoc.onlinelibrary.wiley.com/doi/abs/10.1112/plms.12204},
  eprint    = {https://londmathsoc.onlinelibrary.wiley.com/doi/pdf/10.1112/plms.12204},
  year      = {2019}
}

@article{Pap2005,
  author    = {Papi, M.},
  year      = {2005},
  pages     = {221--234},
  title     = {On the domain of the implicit function and applications},
  volume    = {3},
  journal   = {J. Inequal. Appl.},
  doi       = {10.1155/JIA.2005.221}
}

@article{Sch1982,
  title     = {Simultaneous rational zeros of quadratic forms, {S}eminar on {N}umber {T}heory ({P}aris 1980/1981)},
  author    = {Schmidt, W. M.},
  journal   = {Progr. Math},
  pages     = {281--307},
  year      = {1982},
  publisher = {Birkhauser},
  volume    = {22}
}

@article{Bou1993,
  title     = {Fourier transform restriction phenomena for certain lattice subsets and applications to nonlinear evolution equations, {P}art {I}: {S}chrödinger equations},
  author    = {Bourgain, J.},
  journal   = {Geom. Funct. Anal.},
  volume    = {3},
  number    = {2},
  pages     = {107--156},
  year      = {1993}
}

@article{CSV2016,
  title     = {On linear configurations in subsets of compact abelian groups, and invariant measurable hypergraphs},
  author    = {Candela, P. and Szegedy, B. and Vena, L.},
  journal   = {Ann. Comb.},
  volume    = {20},
  number    = {3},
  pages     = {487--524},
  year      = {2016},
  publisher = {Springer}
}

@incollection{GT2010,
  title     = {An arithmetic regularity lemma, an associated counting lemma, and applications},
  author    = {Green, B. and Tao, T.},
  booktitle = {An Irregular Mind},
  pages     = {261--334},
  year      = {2010},
  publisher = {Springer}
}

@article{GT2012,
  title     = {The quantitative behaviour of polynomial orbits on nilmanifolds},
  author    = {Green, B. and Tao, T.},
  journal   = {Ann. of Math. (2)},
  pages     = {465--540},
  volume    = {175},
  number    = {2},
  year      = {2012},
  publisher = {JSTOR}
}

@misc{Pre2020,
  title     = {Fourier methods in combinatorial number theory},
  author    = {Prendiville, S.},
  year      = {2020},
  howpublished = {\url{https://bit.ly/3caSFdt}},
}

@misc{Tao2014,
  title     = {A proof of {R}oth’s theorem},
  author    = {Tao, T.},
  howpublished = {\url{https://terrytao.wordpress.com/2014/04/24/a-proof-of-roths-theorem/}},
  year      = {2020}
}

@book{Tao2012,
  title     = {Higher order {F}ourier analysis},
  author    = {Tao, T.},
  year      = {2012},
  publisher = {American Mathematical Soc.}
}

@article{HW2022,
	title     = {Discrete Restriction for $(x,x^3)$ and Related Topics},
	author    = {Hughes, K. and Wooley, T. D.},
	doi       = {10.1093/imrn/rnab113},
	number    = {20},
	journal   = {Int. Math. Res. Not.},
	pages     = {15612-15631},
	publisher = {Oxford University Press},
	year      = {2022},
	volume    = {2022},
}

@article{HH2018,
  title     = {Restriction estimates of $\epsilon$-removal type for $k$-th powers and paraboloids},
  author    = {Henriot, K. and Hughes, K.},
  journal   = {Math. Ann.},
  year      = {2018},
  doi       = {https://doi.org/10.1007/s00208-018-1650-7},
  pages     = {963--998},
  volume    = {372},
}

@article{Woo2015,
author      = {Wooley, T. D.},
title       = {Mean value estimates for odd cubic {W}eyl sums},
_journal     = {Bulletin of the London Mathematical Society},
journal     = {Bull. London Math. Soc.},
volume      = {47},
number      = {6},
pages       = {946-957},
doi         = {https://doi.org/10.1112/blms/bdv066},
url         = {https://londmathsoc.onlinelibrary.wiley.com/doi/abs/10.1112/blms/bdv066},
eprint      = {https://londmathsoc.onlinelibrary.wiley.com/doi/pdf/10.1112/blms/bdv066},
year        = {2015}
}

@incollection{Rog1986,
  author    = {Rogovskaya, N. N.},
  title     = {An asymptotic formula for the number of solutions of a system of equations},
  booktitle = {Diophantine Approximations, Part II},
  publisher = {Moskov. Gos. Univ.},
  address   = {Moscow},
  year      = {1986},
  pages     = {78--84},
  language  = {Russian}
}

@article{Rot1953,
  title     = {On certain sets of integers},
  author    = {Roth, K. F.},
  journal   = {J. London Math. Soc.},
  year      = {1953},
  doi       = {},
  pages     = {104--109},
  volume    = {28},
}

@incollection{KM2023,
  title     = {Strong bounds for 3-progressions},
  author    = {Kelley, Z. and Meka, R.},
  booktitle = {2023 IEEE 64th Annual Symposium on Foundations of Computer Science---FOCS 2023},
  pages     = {933--973},
  year      = {2023},
  publisher = {IEEE Computer Society}
}

@article{Kei2014,
  title     = {On a diagonal quadric in dense variables},
  author    = {Keil, E.},
  journal   = {Glasg. Math. J.},
  year      = {2014},
  pages     = {601--628},
  volume    = {56},
}

@article{Hen,
  title = {Additive equations in dense variables via truncated restriction estimates},
  author = {Henriot, K.},
  journal = {Proc. London Math. Soc.},
  volume = {114},
  number = {5},
  pages = {927--959},
  doi = {10.1112/plms.12028},
  year = {2017},
}

@article{Hen2015,
  title     = {Logarithmic Bounds for Translation-Invariant Equations in Squares},
  author    = {Henriot, K.},
  journal   = {Int. Math. Res. Not.},
  year      = {2015},
  pages     = {12540--12562},
  volume    = {2015},
}

@article{BP2017,
  title     = {A transference approach to a {R}oth-type theorem in the squares},
  author    = {Browning, T. D. and Prendiville, S.},
  journal   = {Int. Math. Res. Not.},
  number   = {7},
  year      = {2017},
  pages     = {2219--248},
  volume    = {2017},
}

@article{Gre2005,
  title     = {Roth's theorem in the primes},
  author    = {Green, B. J.},
  journal   = {Ann. of Math. (2)},
  year      = {2005},
  pages     = {1609--1636},
  volume    = {161},
}

@article{Gre2005b,
  title     = {A {S}zemer\'edi-type regularity lemma in abelian groups, with applications},
  author    = {Green, B. J.},
  journal   = {Geom. Funct. Anal.},
  year      = {2005},
  pages     = {340--376},
  volume    = {15},
}

@article{GT2008,
  title     = {The primes contain arbitrarily long arithmetic progressions},
  author    = {Green, B.  and Tao, T.},
  journal   = {Ann. of Math. (2)},
  year      = {2008},
  pages     = {481--547},
  volume    = {167},
}

@article{CG2016,
  title     = {Combinatorial theorems in sparse random sets},
  author    = {Conlon, D.  and Gowers, W. T.},
  journal   = {Ann. of Math. (2)},
  year      = {2016},
  pages     = {367--454},
  volume    = {184},
}

@article{Pre2017,
  title     = {Four variants of the {F}ourier-analytic transference principle},
  author    = {Prendiville, S.},
  journal   = {Online J. Anal. Comb.},
  year      = {2017},
  pages     = {25 pp.},
  volume    = {12},
}

@article{ST2015,
  title     = {Hypergraph containers},
  author    = {Saxton, D. and Thomason, S.},
  journal   = {Invent. Math.},
  year      = {2015},
  pages     = {925--992},
  volume    = {201},
}

@article{CFSZ2021,
  title     = {The regularity method for graphs with few $4$-cycles},
  author    = {Conlon, D. and Fox, J. and Sudakov, B. and Zhao, Y.},
  journal   = {J. London Math. Soc. (2)},
  year      = {2021},
  pages     = {2376--2401},
  volume    = {104},
}

@misc{Kos,
  title     = {Counting solutions to invariant equations in dense sets},
  author    = {Ko{\'{s}}ciuszko, T.},
  year      = {2023},
  howpublished = {arXiv:2306.08567},
}

@incollection{FHHK2024,
  title     = {Strong bounds for 3-progressions},
  author    = {Filmus, Y. and Hatami, H. and Hosseini, K. and Kelman, E.},
  booktitle = {2023 IEEE 64th Annual Symposium on Foundations of Computer Science---FOCS 2024},
  pages     = {1559--1578},
  year      = {2024},
  publisher = {IEEE Computer Society}
}

@incollection{BMS2018,
  title     = {The method of hypergraph containers},
  author    = {Balogh, J. and Morris, R. and Samotij, W.},
  booktitle = {Proceedings of the International Congress of Mathematicians---Rio de Janeiro 2018. Vol. IV. Invited lectures},
  pages     = {3059--3092},
  year      = {2018},
  publisher = {World Scientific Publishing Co. Pte. Ltd., Hackensack, NJ}
}

@article{CCH,
  title     = {Arithmetic {R}amsey theory over {P}iatetski-{S}hapiro numbers},
  author    = {Chapman, J. and Chow, S. and Holdridge, P.},
  journal   = {J. Th\'eor. Nombres Bordeaux},
  year      = {to appear},
  pages     = {17 pp.},
  volume    = {},
}

@article{Cho2018,
  title     = {Roth--{W}aring--{G}oldbach},
  author    = {Chow, S.},
  journal   = {Int. Math. Res. Not.},
  year      = {2018},
  pages     = {2341--2374},
  volume    = {2018},
}

@article{Cha2022,
  title     = {Partition regularity for systems of diagonal equations},
  author    = {Chapman, J.},
  journal   = {Int. Math. Res. Not.},
  year      = {2022},
  pages     = {13272--13316},
  volume    = {2022},
}

@article{CC,
  title     = {Arithmetic {R}amsey theory over the primes},
  author    = {Chapman, J. and Chow, S.},
  journal   = {Proc. Roy. Soc. Edinburgh Sect. A},
  year      = {to appear},
  pages     = {47 pp.},
  volume    = {},
}

@article{CC2025,
  title     = {Generalised {R}ado and {R}oth criteria},
  author    = {Chapman, J. and Chow, S.},
  journal   = {Ann. Sc. Norm. Super. Pisa Cl. Sci. (5)},
  year      = {to appear},
  pages     = {1263--1312 },
  volume    = {XXVI},
}

@article{RSZZ2025,
  title     = {Roth-type theorem for nonlinear equations in {P}iatetski-{S}hapiro primes},
  author    = {Ren, X. and Sun, Y.-C. and Zhang, Q. and Zhang, R.},
  journal   = {Int. J. Number Theory},
  year      = {2025},
  pages     = {887--902},
  volume    = {21},
}

@article{Chi2019,
  title     = {A {R}oth-type theorem with mixed powers},
  author    = {Ching, T. W.},
  journal   = {Ramanujan J.},
  year      = {2020},
  pages     = {581--604},
  volume    = {52},
}

@article{Sal2020,
  title     = {On the {W}aring--{G}oldbach problem with almost equal summands},
  author    = {Salmensuu, J.},
  journal   = {Mathematika},
  year      = {2020},
  pages     = {255--296},
  volume    = {66},
}

@article{CLP2021,
  title     = {Rado’s criterion over squares and higher powers},
  author    = {Chow, S. and Lindqvist, S. and Prendiville, S.},
  journal   = {J. Eur. Math. Soc.},
  year      = {2021},
  pages     = {1925--1997},
  volume    = {23},
}

@article{Pre2021,
  title     = {Counting monochromatic solutions to diagonal {D}iophantine equations},
  author    = {Prendiville, S.},
  journal   = {Discrete Anal.},
  year      = {2021},
  pages     = {47 pp.},
  volume    = {Paper No. 14},
}

@article{Sch2021,
  title     = {A {D}iophantine {R}amsey theorem},
  author    = {Schoen, T. },
  journal   = {Combinatorica},
  year      = {2021},
  pages     = {581--599},
  volume    = {41},
}

@article{GL2019,
  title     = {Monochromatic {S}olutions to $x + y = z^2$},
  author    = {Green, B. J. and Lindqvist, S.},
  journal   = {Canad. J. Math.},
  year      = {2019},
  pages     = {579--605},
  volume    = {71},
}

@article{Lin2018,
  title     = {Partition {R}egularity of {G}eneralised {F}ermat {E}quations},
  author    = {Lindqvist, S.},
  journal   = {Combinatorica},
  year      = {2018},
  pages     = {1457--1483},
  volume    = {38},
}

@article{Pie2019,
  title     = {The {V}inogradov {M}ean {V}alue {T}heorem
[after {W}ooley, and {B}ourgain, {D}emeter and {G}uth]},
  author    = {Pierce, L.},
  journal   = {Asterisque},
  year      = {2019},
  pages     = {479--564},
  volume    = {407},
}

@misc{Lin,
  title     = {Weak approximation results for quadratic forms in four variables},
  author    = {Lindqvist, S.},
  year      = {2017},
  howpublished = {	arXiv:1704.00502},
}

@article{BL2019,
  title     = {Sieving rational points on varieties},
  author    = {Browning, T. and Loughran, D.},
  journal   = {Trans. Amer. Math. Soc.},
  year      = {2019},
  pages     = {5757--5785},
  volume    = {371},
}

@article{BDG2016,
  title     = {Proof of the main conjecture in {V}inogradov’s {M}ean {V}alue {T}heorem for degrees higher than three},
  author    = {Bourgain, J. and Demeter, C. and Guth, L.},
  journal   = {Ann. of Math. (2)},
  year      = {2016},
  pages     = {633--682},
  volume    = {184},
}

@misc{Ebe,
  title     = {An abelian arithmetic regularity lemma},
  author    = {Eberhard, S.},
  year      = {2016},
  howpublished = {arXiv:1606.09303},
}
\end{document}